\documentclass[10pt]{article}
\usepackage{hyperref}
\usepackage{amsmath}
\usepackage{amsfonts}
\usepackage{amssymb}
\usepackage{amsthm}

\usepackage[shortlabels]{enumitem}
\usepackage{microtype}
\usepackage{fullpage}
\usepackage{xparse}
\usepackage{mathrsfs}
\usepackage{mleftright}
\mleftright

\NewDocumentCommand{\R}{}{\mathbb{R}}
\NewDocumentCommand{\C}{}{\mathbb{C}}
\NewDocumentCommand{\N}{}{\mathbb{N}}

\NewDocumentCommand{\M}{}{\mathbb{M}}
\NewDocumentCommand{\Real}{}{\mathrm{Re}\:}
\NewDocumentCommand{\Imag}{}{\mathrm{Im}\:}

\NewDocumentCommand{\Ct}{}{\widetilde{C}}

\NewDocumentCommand{\Nplus}{}{\N_{+}}
\NewDocumentCommand{\Error}{}{\mathrm{Error}}

\NewDocumentCommand{\MetricSpace}{}{\mathfrak{M}}
\NewDocumentCommand{\Manifold}{}{\MetricSpace}
\NewDocumentCommand{\metric}{o o}{\rho\IfNoValueF{#1}{(#1\IfNoValueF{#2}{,#2})}}
\NewDocumentCommand{\measure}{}{\mu}
\NewDocumentCommand{\MetricBall}{m m}{B_{\metric}(#1,#2)}
\NewDocumentCommand{\BMetricBall}{m m}{B_{\metric}\left(#1,#2\right)}
\NewDocumentCommand{\MetricBallClosure}{m m}{\overline{\MetricBall{#1}{#2}}}

\NewDocumentCommand{\Compact}{}{\mathcal{K}}

\NewDocumentCommand{\Distributions}{o}{\mathcal{D}'\IfNoValueF{#1}{(#1)}}

\NewDocumentCommand{\Bno}{}{B^n(1)}

\NewDocumentCommand{\Norm}{m o}{\|#1 \|\IfNoValueF{#2}{_{#2}}}
\NewDocumentCommand{\BNorm}{m o}{\left\|#1 \right\|\IfNoValueF{#2}{_{#2}}}

\NewDocumentCommand{\OpNorm}{m o}{\|#1 \|\IfNoValueF{#2}{_{#2\rightarrow #2}}}

\NewDocumentCommand{\LpNorm}{m m o o o}{\Norm{#1}[\LpSpace{#2}[#3][#4][#5]]}
\NewDocumentCommand{\BLpNorm}{m m o o o}{\BNorm{#1}[\LpSpace{#2}[#3][#4][#5]]}

\NewDocumentCommand{\HsSpace}{m o}{H^{#1}\IfValueT{#2}{(#2)}}

\NewDocumentCommand{\HsNorm}{m m o}{\Norm{#1}[\HsSpace{#2}[#3]]}
\NewDocumentCommand{\BHsNorm}{m m o}{\BNorm{#1}[\HsSpace{#2}[#3]]}

\NewDocumentCommand{\LpSpace}{m o o o}{L^{#1}\IfNoValueF{#2}{(#2\IfNoValueF{#3}{,#3 \IfNoValueF{#4}{;#4} })}}
\NewDocumentCommand{\DomainSymbol}{}{\mathscr{D}}

\NewDocumentCommand{\Domain}{o}{\DomainSymbol\IfNoValueF{#1}{(#1)}}
\NewDocumentCommand{\DAinfty}{}{\DomainSymbol(A^{\infty})}
\NewDocumentCommand{\DAsinfty}{}{\DomainSymbol((A^{*})^{\infty})}

\NewDocumentCommand{\Yb}{}{\overline{Y}}

\NewDocumentCommand{\NSubsetSymbol}{}{\mathfrak{N}}
\NewDocumentCommand{\NSubsetx}{}{\NSubsetSymbol_x}
\NewDocumentCommand{\NSubsety}{}{\NSubsetSymbol_y}

\NewDocumentCommand{\OpQSymbol}{}{\mathscr{Q}}
\NewDocumentCommand{\OpQx}{o o}{\OpQSymbol_1\IfNoValueF{#1}{^{#1\IfNoValueF{#2}{,#2}}}}
\NewDocumentCommand{\OpQy}{o o}{\OpQSymbol_2\IfNoValueF{#1}{^{#1\IfNoValueF{#2}{,#2}}}}
\NewDocumentCommand{\OpP}{}{\mathscr{P}}
\NewDocumentCommand{\opP}{}{\mathscr{P}}
\NewDocumentCommand{\opL}{}{\mathscr{L}}

\NewDocumentCommand{\opR}{}{\mathscr{R}}
\NewDocumentCommand{\opE}{}{\mathscr{E}}

\NewDocumentCommand{\ip}{m m o}{\langle #1, #2\rangle\IfNoValueF{#3}{_{#3}}}
\NewDocumentCommand{\Bip}{m m o}{\left\langle #1, #2\right\rangle\IfNoValueF{#3}{_{#3}}}

\NewDocumentCommand{\CinftySpace}{s o o}{C^\infty\IfNoValueF{#2}{\IfBooleanT{#1}{\left}(#2\IfNoValueF{#3}{;#3}  \IfBooleanT{#1}{\right})}}
\NewDocumentCommand{\CinftySpacet}{s o o}{C_{t}^\infty\IfNoValueF{#2}{\IfBooleanT{#1}{\left}(#2\IfNoValueF{#3}{;#3}  \IfBooleanT{#1}{\right})}}
\NewDocumentCommand{\CinftySpacetx}{s o o}{C_{t,x}^\infty\IfNoValueF{#2}{\IfBooleanT{#1}{\left}(#2\IfNoValueF{#3}{;#3}  \IfBooleanT{#1}{\right})}}
\NewDocumentCommand{\CSpacex}{s m o}{C_x\IfBooleanT{#1}{\left}(#2\IfNoValueF{#3}{;#3}\IfBooleanT{#1}{\right})}
\NewDocumentCommand{\CSpacey}{m o}{C_y(#1\IfNoValueF{#2}{;#2})}

\NewDocumentCommand{\LpSpacewx}{m o o}{L_{w,x}^{#1}\IfNoValueF{#2}(#2\IfNoValueF{#3}{,#3})}
\NewDocumentCommand{\LpSpacewy}{m o o}{L_{w,y}^{#1}\IfNoValueF{#2}(#2\IfNoValueF{#3}{,#3})}
\NewDocumentCommand{\CzinftySpace}{o}{C_0^\infty\IfNoValueF{#1}{(#1)}}
\NewDocumentCommand{\CSpace}{m}{C(#1)}
\NewDocumentCommand{\ClocSpace}{m}{C_{\mathrm{loc}}(#1)}

\NewDocumentCommand{\CbinftySpace}{o o}{C_b^\infty\IfNoValueF{#1}{(#1\IfNoValueF{#2}{;#2})}}

\NewDocumentCommand{\ceil}{m}{\lceil #1 \rceil}

\NewDocumentCommand{\uh}{}{\hat{u}}

\NewDocumentCommand{\deltah}{}{\hat{\delta}}

\NewDocumentCommand{\Wd}{}{d}
\NewDocumentCommand{\WWd}{}{(W,d)}

\NewDocumentCommand{\degWd}{}{\deg}

\NewDocumentCommand{\CjSpace}{m o}{C^{#1}\IfNoValueF{#2}{(#2)}}

\NewDocumentCommand{\CjNorm}{m m o}{\Norm{#1}[\CjSpace{#2}[#3]]}

\NewDocumentCommand{\BCjNorm}{m m o}{\BNorm{#1}[\CjSpace{#2}[#3]]}

\NewDocumentCommand{\Ltip}{s m m o o o}{\IfBooleanTF{#1}{\left \langle #2,#3 \right\rangle}{\langle#2,#3 \rangle}_{\LpSpace{2}[#4][#5][#6]}}

\NewDocumentCommand{\HilbertSpace}{}{\mathscr{H}}

\NewDocumentCommand{\Hip}{s m m}{\IfBooleanTF{#1}{\left \langle #2,#3 \right\rangle}{\langle#2,#3 \rangle}_{\HilbertSpace}}
\NewDocumentCommand{\HNorm}{s m}{\IfBooleanTF{#1}{\left \| #2 \right\|}{\|#2 \|}_{\HilbertSpace}}

\NewDocumentCommand{\FormQSymbol}{}{\mathcal{Q}}
\NewDocumentCommand{\FormQ}{s o o}{\FormQSymbol\IfValueT{#2}{ \IfBooleanTF{#1}{\left( #2 \IfValueT{#3}{,#3} \right)}{(#2\IfValueT{#3}{,#3})}     }}
\NewDocumentCommand{\FormQb}{s o o}{\overline{\FormQSymbol}\IfValueT{#2}{ \IfBooleanTF{#1}{\left( #2 \IfValueT{#3}{,#3} \right)}{(#2\IfValueT{#3}{,#3})}     }}
\NewDocumentCommand{\FormQClosure}{s o o}{\widetilde{\FormQSymbol}\IfValueT{#2}{ \IfBooleanTF{#1}{\left( #2 \IfValueT{#3}{,#3} \right)}{(#2\IfValueT{#3}{,#3})}     }}
\NewDocumentCommand{\FormQClosureb}{s o o}{\overline{\widetilde{\FormQSymbol}}\IfValueT{#2}{ \IfBooleanTF{#1}{\left( #2 \IfValueT{#3}{,#3} \right)}{(#2\IfValueT{#3}{,#3})}     }}

\NewDocumentCommand{\sconst}{}{(\mathrm{s.c.})}
\NewDocumentCommand{\lconst}{}{(\mathrm{l.c.})}

\NewDocumentCommand{\TLNorm}{s m m}{\IfBooleanT{#1}{\left}\| #2 \IfBooleanT{#1}{\right}\|_{\mathscr{F}^{#3}_{2,2}\WWd}}

\NewDocumentCommand{\LtltNorm}{m}{\left\|#1\right\|_{L^2(\Manifold,\mu;\ell^2(\N))}}

\NewDocumentCommand{\Ropn}{}{\R^{1+n}}
\NewDocumentCommand{\Rn}{}{\R^{n}}
\NewDocumentCommand{\supp}{}{\mathrm{supp}}

\NewDocumentCommand{\LtlocSpace}{o}{L^2_{\mathrm{loc}}\IfValueT{#1}{(#1)}}
\NewDocumentCommand{\Mh}{}{\widehat{M}}

\NewDocumentCommand{\BanachSpaceY}{}{\mathscr{Y}}
\NewDocumentCommand{\BanachSpace}{}{\mathscr{X}}
\NewDocumentCommand{\BanachSpaceDual}{}{\BanachSpace'}
\NewDocumentCommand{\BanachSpacew}{}{\BanachSpace_w}

\NewDocumentCommand{\BanachSpaceDualw}{}{\BanachSpaceDual_w}
\NewDocumentCommand{\BanachSpaceDualwstar}{}{\BanachSpaceDual_{*}}
\NewDocumentCommand{\BanachSpacePairing}{s m m}{\IfBooleanT{#1}{\left}\langle #2,#3 \IfBooleanT{#1}{\right}\rangle_{\BanachSpace,\BanachSpaceDual}}
\NewDocumentCommand{\BanachSpaceNorm}{s m}{\IfBooleanT{#1}{\left}\| #2\IfBooleanT{#1}{\right}\|_{\BanachSpace}}
\NewDocumentCommand{\BanachSpaceDualNorm}{s m}{\IfBooleanT{#1}{\left}\| #2\IfBooleanT{#1}{\right}\|_{\BanachSpaceDual}}

\NewDocumentCommand{\EmptySet}{}{\varnothing}

\NewDocumentCommand{\opLxdelta}{}{\opL_{x,\delta}}

\begin{document}

\newtheorem{theorem}{Theorem}[section]
\newtheorem{corollary}[theorem]{Corollary}
\newtheorem{proposition}[theorem]{Proposition}
\newtheorem{lemma}[theorem]{Lemma}
\newtheorem{conjecture}[theorem]{Conjecture}
\newtheorem{problem}[theorem]{Problem}

\theoremstyle{remark}
\newtheorem{remark}[theorem]{Remark}

\theoremstyle{definition}
\newtheorem{definition}[theorem]{Definition}

\theoremstyle{definition}
\newtheorem{notation}[theorem]{Notation}

\theoremstyle{remark}
\newtheorem{example}[theorem]{Example}

\theoremstyle{definition}
\newtheorem{construction}[theorem]{Construction}
\newtheorem{numberedassumption}[theorem]{Assumption}
\newtheorem{condition}[theorem]{Condition}

\theoremstyle{definition}
\newtheorem{assumption}{Assumption}

\numberwithin{equation}{section}

\title{Hypoellipticity and Higher Order Gaussian Bounds}

\author{Brian Street\footnote{The author was partially supported by National Science Foundation Grant 2153069.}}
\date{}

\maketitle

\begin{abstract}
Let \((\mathfrak{M},\rho,\mu)\) be a metric measure space satisfying a doubling condition, \(p_0\in (1,\infty)\), and
\(T(t):L^{p_0}(\mathfrak{M},\mu)\rightarrow L^{p_0}(\mathfrak{M},\mu)\), \(t\geq 0\), a strongly continuous semi-group.
We provide sufficient conditions under which \(T(t)\) is given by integration against an integral kernel satisfying higher-order
Gaussian bounds of the form
\[
    \left| K_t(x,y) \right| \leq C \exp\left( -c \left( \frac{\rho(x,y)^{2\kappa}}{t} \right)^{\frac{1}{2\kappa-1}} \right) \mu\left( B_\rho\left(x,\rho(x,y)+t^{1/2\kappa}\right) \right)^{-1},
\]
where \(B_\rho\) denotes the metric ball.
We also provide conditions for similar bounds on ``derivatives'' of \(K_t(x,y)\)
and our results are localizable. If \(A\) is the generator of \(T(t)\) the main hypothesis
is that \(\partial_t -A\) and \(\partial_t-A^{*}\) satisfy a hypoelliptic estimate at every scale, uniformly in the scale.
We present applications to subelliptic PDEs.
\end{abstract}

\section{Introduction}
Let \((\MetricSpace,\metric)\) be a metric space and let \(\measure\) be a \(\sigma\)-finite Borel measure
on \(\MetricSpace\) (we will assume \(\measure\) has a ``doubling'' property--see Assumption \ref{Assumption::LocalDoubling}).
Fix \(p_0\in (1,\infty)\) and let \(T(t):\LpSpace{p_0}[\MetricSpace][\measure]\rightarrow \LpSpace{p_0}[\MetricSpace][\measure]\), \(t\geq 0\),
be a strongly continuous semi-group and let \((A,\Domain[A])\) be its generator (see \cite[Chapter 2, Section 1]{EngelNagelAShortCourseOnOperatorSemigroups}).
Our main motivation comes from the case when \(A\) is a partial differential operator
with smooth coefficients, though we proceed in this abstract setting.\footnote{The abstract setting offers some conveniences. For example, it allows
us to present a general result which
applies even when \(\Domain[A]\) plays a central role in understanding the operator (e.g., when studying homogeneous boundary value problems). See Section \ref{Section::MaxSub::BoundaryValueProblems}.}

The main theorem of this paper gives sufficient conditions under which \(T(t)\) is given by integration against a kernel \(K_t(x,y)\) satisfying
bounds like
\begin{equation}\label{Eqn::Intro::GaussianBoundsNoDerivs}
    \left| K_t(x,y) \right| \leq C \exp\left( -c \left( \frac{\metric[x][y]^{2\kappa}}{t} \right)^{\frac{1}{2\kappa-1}} \right) 
    \measure\left( 
        \MetricBall{x}{\metric[x][y]+t^{1/2\kappa}}
        \right)^{-1},
\end{equation}
for some \(\kappa\geq 1\).
Moreover, we more generally provide estimates on the ``derivatives'' of \(K_t(x,y)\). Abstractly, we are given operators \(X\) and \(Y\) acting
in the \(x\) and \(y\) variables, functions \(S_X,S_Y:(0,\infty)\rightarrow (0,\infty)\), and under certain assumptions we prove estimates like
\begin{equation}\label{Eqn::Intro::GaussianBoundsDerivs}
    \begin{split}
    \left| \partial_t^l X \Yb K_t(x,y) \right| \leq &C_l S_X(\metric[x][y]+t^{1/2\kappa})^{-1} S_Y(\metric[x][y]+t^{1/2\kappa})^{-1} ( \metric[x][y]+t^{1/2\kappa} )^{-2\kappa l}
    \\&\exp\left( -c \left( \frac{\metric[x][y]^{2\kappa}}{t} \right)^{\frac{1}{2\kappa-1}} \right) \measure\left( 
        \MetricBall{x}{\metric[x][y]+t^{1/2\kappa}}
        \right)^{-1}.
\end{split}
\end{equation}
Here, \(\Yb u= \overline{Y\overline{u}}\). In our main application (see Section \ref{Section::ApplicationNew}), \(X\) and \(Y\) are (possibly high order) differential operators.

The central assumption (see Assumptions \ref{Assumption::HypoellipticityI} and \ref{Assumption::HypoellipticityII}) is that the heat operators \(\partial_t -A\) and \(\partial_t-A^{*}\) satisfy a ``hypoelliptic estimate''
at every scale, uniformly in the scale. This is particularly useful in the study of subelliptic PDEs. Indeed, as we discuss in Section \ref{Section::ApplicationNew}, the estimates
we require are often consequences of a priori estimates which follow from well-known methods. One consequence is that one can show that
the parametricies for general maximally subelliptic PDEs satisfy good estimates, without using Lie groups, and only relying on elementary a priori estimates
and scaling techniques (see Section \ref{Section::ApplicationNew} for details);
this provides an approach to studying subelliptic PDEs when methods like the Rothschild-Stein lifting procedure \cite{RothschildSteinHypoellipticDifferentialOperatorsAndNilpotentGroups}
do not apply.

Important for such applications is that the theory we present is localizable. Our main assumptions and theorem are stated with this in mind.
The approach we use
involves proving the heat kernel has a certain Gevrey regularity in the \(t\) variable.
The connection between Gevrey regularity and heat operators dates back to Gevrey's original paper on the subject \cite{GevreySurLaNatureAnalytiqueDesSolutionsDesEquationsAuxDeriveesPartiellesPremierMemoire}.
The method we use was first used by Jerison and S\'anchez-Calle \cite{JerisonSanchezCalleEstimatesForTheHeatKernelForASumOfSquaresOfVectorFields}
to study the heat kernel for H\"ormander's sub-Laplacian; and they attribute the basic idea to Stein. Higher order Gaussian bounds for a certain homogeneous operator on nilpotent Lie groups
were established with this approach by Hebisch \cite{HebischSharpPointwiseEstimateForTheKernelsOfTheSemigroupGeneratedBySumsOfEvenPowersOfVectorFieldsOnHomogeneousGroups}.
The author used these methods to study heat operators associated to some general maximally subelliptic PDEs in  \cite[Chapter 8]{StreetMaximalSubellipticity};
however, the proof there is closely intertwined with the operators in question and it is not made clear that this is a general principle that
applies in many similar situations. This paper aims to fix that deficiency
by presenting an abstract theorem which will be used in forthcoming work by the author to address other settings
(see Section \ref{Section::MaxSub::BoundaryValueProblems}).

See Remark \ref{Rmk::MaxSub::SubLapalcian::OtherProofs} for some citations to the literature for related
results for PDEs.

\begin{remark}\label{Rmk::Intro::NonSharp}
    There are many papers studying Gaussian bounds for various heat equations, using various methods.
    Most of these focus on the case \(\kappa=1\) in \eqref{Eqn::Intro::GaussianBoundsDerivs},
    though there are many studying higher \(\kappa\) as well; see Remark \ref{Rmk::MaxSub::SubLapalcian::OtherProofs}
    for some references to the literature which are closely related to the results of this paper.
    One point, which sets apart our analysis, is the following. Most previous results
    studied the case when the generator was elliptic and sharp regularity results for the generator were already known
    (or other situations where sharp results for the generator were known).
    As we describe in Section \ref{Section::Subellip::MaximalSub::Parametricies} and Remark \ref{Rmk::Subellip::Proof::epsilon0NotOptimal}, non-sharp
    sub-elliptic estimates for the heat operator can be used with the methods in this paper to establish higher-order
    Guassian bounds for the heat semi-group.  These Gaussian bounds can be then used to establish sharp subelliptic results for the generator.
    This is particularly important in the study of subelliptic PDEs (as described in Section \ref{Section::ApplicationNew}),
    where non-sharp subelliptic estimates are much easier to establish than sharp subelliptic estimates.
    This idea was already partially visible in the proof of Hebisch \cite{HebischSharpPointwiseEstimateForTheKernelsOfTheSemigroupGeneratedBySumsOfEvenPowersOfVectorFieldsOnHomogeneousGroups},
    though the regularity theory for the generator in that case was already well-understood.
    This method will be particularly useful in future study of boundary value problems, where sharp results
    are not known (see Section \ref{Section::MaxSub::BoundaryValueProblems}).
\end{remark}

\section{The Setting and Assumptions}\label{Section::Assumptions}
Throughout this paper, \(\left( \MetricSpace,\metric \right)\) is a metric space and \(\measure\) is a \(\sigma\)-finite Borel measure
on \(\MetricSpace\); let \(\MetricBall{x}{\delta}:=\left\{ y\in \MetricSpace :\metric[x][y]<\delta \right\}\)
and \(\MetricBallClosure{x}{\delta}\) its closure.
Fix \(p_0\in (1,\infty)\) and \(T(t):\LpSpace{p_0}[\MetricSpace][\measure]\rightarrow \LpSpace{p_0}[\MetricSpace][\measure]\), \(t\geq 0\),
a strongly continuous semi-group.

Fix \(\kappa\geq 1\), \(\NSubsetx,\NSubsety\subseteq \MetricSpace\), and \(\delta_0\in (0,\infty]\). We will prove results like \eqref{Eqn::Intro::GaussianBoundsDerivs}
for \(x\in \NSubsetx\), \(y\in \NSubsety\), \(\metric[x][y]^{2\kappa}+t\lesssim \delta_0^{2\kappa}\); in particular, if
\(\NSubsetx=\NSubsety=\MetricSpace\) and \(\delta_0=\infty\), then we will establish results like \eqref{Eqn::Intro::GaussianBoundsDerivs}
for all \(x,y\in \MetricSpace\) and \(t>0\).

We let \(p_0'\in (1,\infty)\) be the dual exponent to \(p_0\) (i.e., \(1/p_0+1/p_0'=1\)) and write for \(f\in \LpSpace{p_0}[\MetricSpace][\measure]\),
\(g\in \LpSpace{p_0'}[\MetricSpace][\measure]\),
\begin{equation*}
    \ip{g}{f}=\int g\overline{f}\: d\measure.
\end{equation*}
For an operator \(S\), we let \(S^{*}\) be the adjoint with respect to this sesqui-linear form.

Since \(\LpSpace{p_0}[\MetricSpace][\measure]\) is reflexive, the adoint semi-group \(T(t)^{*}:\LpSpace{p_0'}[\MetricSpace][\measure]\rightarrow \LpSpace{p_0'}[\MetricSpace][\measure]\)
is also strongly continuous (see  \cite[page 9]{EngelNagelAShortCourseOnOperatorSemigroups}).
Let \(A\) and \(A^{*}\), with dense domains \(\Domain[A]\) and \(\Domain[A^{*}]\), be the generators of \(T(t)\) and \(T(t)^{*}\), respectively.
\(A^{*}\) is the adjoint of \(A\); see
 \cite[Chapter 2, Section 2.5]{EngelNagelAShortCourseOnOperatorSemigroups}.
Set
\begin{equation*}
    \DAinfty:=\bigcap_{j=1}^{\infty}\Domain[A^j],\quad \DAsinfty:=\bigcap_{j=1}^{\infty}\Domain[(A^{*})^j],
\end{equation*}
with the usual Fr\'{e}chet topologies.\footnote{That \(\Domain[A^m]=\cap_{j=1}^m\Domain[A^j]\) with norm
\(\sum_{j=0}^m \|A^j v\|\)
is a Banach space
follows from \cite[Chapter 2, Proposition 2.15]{EngelNagelAShortCourseOnOperatorSemigroups}. It then follows
that \(\DAinfty\) is a Fr\'{e}chet space; similarly for \(\DAsinfty\).}

Fix two continuous operators \(X:\DAinfty\rightarrow \LpSpace{p_0}[\MetricSpace][\measure]\) and \(Y:\DAsinfty\rightarrow\LpSpace{p_0'}[\MetricSpace][\measure]\),
and non-decreasing functions \(S_X,S_Y:(0,\infty)\rightarrow (0,\infty)\).

\begin{remark}
    Of particular interest is the special case when \(X=I\), \(Y=I\) (the identities on their respective spaces), and \(S_X=S_Y=C^{-1}\)
    (where \(C\geq 1\) is a constant).
    In this case, we establish bounds like \eqref{Eqn::Intro::GaussianBoundsNoDerivs}.
\end{remark}

\begin{remark}\label{Rmk::Assumption::Symmetric}
    A key aspect of the assumptions which follow is that they are symmetric in \(T(t)\) and \(T(t)^{*}\). More precisely,
    the assumptions are the same for \(T(t),A,\NSubsetx,p_0,X,S_X\) and \(T(t)^{*},A^{*},\NSubsety,p_0',Y,S_Y\).
\end{remark}

\begin{assumption}[Locality]\label{Assumption::Locality}
    \(\forall x_0\in \NSubsetx\), \(\delta\in (0,\delta_0)\), the \(\LpSpace{p_0}[\MetricSpace][\measure]\)-closure of
    \begin{equation*}
        \left\{ \phi\in \DAinfty : A^j\phi\big|_{\MetricBall{x_0}{\delta}}=0, \forall j\geq 0 \right\}
    \end{equation*}
    contains \(\LpSpace{p_0}[\MetricSpace\setminus \MetricBallClosure{x_0}{\delta}][\measure]\). We assume the same
    for \(\NSubsety\), \(A^{*}\), \(p_0'\) in place of 
    \(\NSubsetx\), \(A\), \(p_0\).
\end{assumption}

\begin{remark}\label{Rmk::Locality::LocalityUsuallyObvious}
    In our main application in Section \ref{Section::ApplicationNew}, Assumption \ref{Assumption::Locality} is immediate: \(A\) is a partial differential operator with smooth coefficients (on a manifold \(\MetricSpace\)) and we
    have \(\forall x\in \MetricSpace\), \(\forall \delta>0\),
    \begin{equation*}
        \CzinftySpace[\MetricSpace\setminus \MetricBallClosure{x}{\delta}]\subseteq \CzinftySpace[\MetricSpace]\subseteq \DAinfty \cap \DAsinfty
    \end{equation*}
    and \(\CzinftySpace[\MetricSpace\setminus \MetricBallClosure{x}{\delta}]\)
    is dense in both \(\LpSpace{p_0}[\MetricSpace\setminus \MetricBallClosure{x}{\delta}][\measure]\)
    and \(\LpSpace{p_0'}[\MetricSpace\setminus \MetricBallClosure{x}{\delta}][\measure]\).
    When \(A\) is not a partial differential operator, Assumption \ref{Assumption::Locality} is restrictive, and assumes some
    kind of locality which is likely false for any given non-local operator.
\end{remark}

    \subsection{Assumptions on the measure}
    \begin{assumption}[Local Positivity]\label{Assumption::LocalPositivity}
    For all non-empty open sets 
    \begin{equation*}
        U\subseteq \left( \bigcup_{x_0\in \NSubsetx} \MetricBall{x_0}{\delta_0} \right)\bigcup\left( \bigcup_{y_0\in \NSubsety} \MetricBall{y_0}{\delta_0} \right),
    \end{equation*}
    \(\mu(U)>0\).
\end{assumption}

\begin{remark}\label{Rmk::Assump::Measure::ContinuousUnique}
    If \(f\in \LpSpace{p}[\MetricSpace][\measure]\) and \(U\) is an open set as in Assumption \ref{Assumption::LocalPositivity},
    then we can ask if \(f\big|_U\) agrees almost everywhere with a continuous function.
    When it does, Assumption \ref{Assumption::LocalPositivity} implies that this continuous version is unique.
    Henceforth, we say \(f\big|_U\) is continuous to mean that it agrees almost everywhere with such a continuous function,
    and we identify it with this unique function. 
\end{remark}

\begin{assumption}[Local Doubling]\label{Assumption::LocalDoubling}
    There exists \(D_1\geq 1\), \(\forall x_0\in \NSubsetx\), \(\forall x\in \MetricBall{x_0}{\delta_0}\),
    \(\forall \delta\in (0, (\delta_0-\metric[x_0][x])/2)\),
    \begin{equation}\label{Eqn::LocalDoubling::DoublingEquation}
        \measure(\MetricBall{x}{2\delta})\leq D_1 \measure(\MetricBall{x}{\delta})<\infty.
    \end{equation}
    We assume the same for \(\NSubsetx\) replaced with \(\NSubsety\) (with the same constant \(D_1\)).
    See Remark \ref{Rmk::Assumption::RemoveQuantifers} for an assumption which implies this, is easier
    to understand, and often appears in applications.
\end{assumption}

\begin{remark}
    In the application we present
    in Section \ref{Section::ApplicationNew}, Assumption \ref{Assumption::LocalDoubling} was established by Nagel, Stein, and Wainger \cite{NagelSteinWaingerBallsAndMetricsDefinedByVectorFieldsIBasicProperties}.
    See Theorem \ref{Thm::MaxSub::Scaling::NSWScaling} \ref{Item::MaxSub::Scaling::Doubling}.
\end{remark}

    \subsection{Hypoellipticity Assumptions}
    By  \cite[Chapter 1, Proposition 1.4]{EngelNagelAShortCourseOnOperatorSemigroups}, there exists \(M_0\geq 1\) and \(\omega_0\in \R\)
such that
\begin{equation}\label{Eqn::Assump::Hypo::IntroduceM0Andomega0}
    \OpNorm{T(t)}[\LpSpace{p_0}[\MetricSpace]]\leq M_0 e^{\omega_0 t},\quad \forall t\geq 0.
\end{equation}
Henceforth, we fix such an \(M_0\) and \(\omega_0\).

\begin{definition}\label{Defn::Assump::Hypo::HeatFunction}
    For \(x\in \MetricSpace\), \(\delta>0\), and \(I\subseteq \R\) an open interval, we say
    \(u(t):I\rightarrow \DAinfty\) is an \((A-\omega_0,x,\delta)\)-heat function if
    \begin{enumerate}[(i)]
        \item\label{Item::Assump::Hypo::HeatFunction::SmoothExceptOnePoint} \(\exists t_0\in I\) with \(u\big|_{[t_0,\infty)\cap I} \in \CinftySpace*[\lbrack t_0,\infty)\cap I][\DAinfty]\)
            and \(u\big|_{(-\infty,t_0)\cap I} \in  \CinftySpace*[(-\infty,t_0)\cap I][\DAinfty]\).
        \item\label{Item::Assump::Hypo::HeatFunction::LocalSmooth} \(u\big|_{I\times \MetricBall{x}{\delta}}\in \CinftySpace*[I][\LpSpace{p_0}[\MetricBall{x}{\delta}]]\).
        \item\label{Item::Assump::Hypo::HeatFunction::Derivatives} For all \(j\geq 1\), \(\partial_t^j \left(u(t)\big|_{I\times \MetricBall{x}{\delta}}  \right)=\left( \left( \delta^{2\kappa}\left( A-\omega_0 \right) \right)^j u(t) \right)\big|_{I\times \MetricBall{x}{\delta}}\).
    \end{enumerate}
    We similarly define an \(\left( A^{*}-\omega_0,y,\delta \right)\)-heat function by replacing \(A,p_0\) with \(A^{*},p_0'\).
\end{definition}

\begin{remark}
    The only way \(\kappa\) appears in our assumptions is through the use of heat functions.
\end{remark}

\begin{remark}
    In most cases we are interested in, the case \(j\geq 2\) of Definition \ref{Defn::Assump::Hypo::HeatFunction} \ref{Item::Assump::Hypo::HeatFunction::Derivatives} follows from the case \(j=1\);
    for example, when \(A\) is a partial differential operator. However, for general operators it does not, due to the localization involved.
\end{remark}

\begin{remark}
    Note that Definition \ref{Defn::Assump::Hypo::HeatFunction} \ref{Item::Assump::Hypo::HeatFunction::SmoothExceptOnePoint}
    allows \(u(t,y)\) to be (possibly) discontinuous at one point \(t_0\in I\); however,
    Definition \ref{Defn::Assump::Hypo::HeatFunction} \ref{Item::Assump::Hypo::HeatFunction::LocalSmooth}
    requires this discontinuity to avoid \(y\in \MetricBall{x}{\delta}\).  It is possible
    \(u(t)\in \CinftySpace[I][\DAinfty]\); in this case any \(t_0\in I\) will do.
\end{remark}

\begin{notation}
    For a metric space \(\MetricSpace\), we let \(\CSpace{\MetricSpace}\) denote the Banach space of bounded continuous functions
    and \(\ClocSpace{\MetricSpace}\) denote the 
    space of continuous functions (with the compact-open topology).
\end{notation}

\begin{assumption}[Hypoellipticity I]\label{Assumption::HypoellipticityI}
    \(\exists a_1\in (0,1/2]\), \(\forall x_0\in \NSubsetx\), \(\delta\in (0,\delta_0)\), \(\forall u:(-1/2,1/2)\rightarrow \DAinfty\)
    an \((A-\omega_0,x_0,\delta)\)-heat function,
    \begin{enumerate}[(i)]
        \item\label{Item::Assumption::HypoellipticityI::Qualitative} \(Xu(t,x)\big|_{(-a_1,a_1)\times \MetricBall{x_0}{a_1\delta}}\in \CinftySpace[(-a_1,a_1)][\CSpace{\MetricBall{x_0}{a_1\delta}}]\). See Remark \ref{Rmk::Assump::Measure::ContinuousUnique} for comments on the meaning of this.
        \item\label{Item::Assumption::HypoellipticityI::Quantitative} We have,
            \begin{equation}\label{Eqn::Assump::Hyp::HypoellipticityIBound}
                \sup_{\substack{t\in (-a_1,a_1)\\ x\in \MetricBall{x_0}{a_1\delta}}} \left| S_X(\delta) Xu(t,x) \right|
                \leq \measure(\MetricBall{x_0}{\delta})^{-1/p_0} \LpNorm{u}{p_0}[(-1/2,1/2)\times \MetricBall{x_0}{\delta/2}][dt\times d\measure].
            \end{equation}
        \item\label{Item::Assumption::HypoellipticityI::CommuteDerivs} 
        \(\partial_t X u(t,x)\big|_{(-a_1,a_1)\times \MetricBall{x_0}{a_1\delta}} =X \delta^{2\kappa} (A-\omega_0) u(t,x)\big|_{(-a_1,a_1)\times \MetricBall{x_0}{a_1\delta}} \).
    \end{enumerate}
    And we assume the same for \(A, X, S_X, p_0, \NSubsetx\) replaced with \(A^{*}, Y, S_Y, p_0', \NSubsety\).
\end{assumption}

\begin{assumption}\label{Assumption::DoublingOfSXandSY}
    \(\exists D_2\geq 1\), \(\forall \delta\in (0,\infty)\), \(S_X(2\delta)\leq D_2S_X(\delta)\) and \(S_Y(2\delta)\leq D_2 S_Y(\delta)\).
\end{assumption}

\begin{assumption}[Hypoellipticity II]\label{Assumption::HypoellipticityII}
    \(\exists C_1\geq 1\), \(a_2\in (0,1]\), \(\forall x_0\in \NSubsetx\), 
    \(\forall x\in \MetricBall{x_0}{\delta_0}\), \(\forall \delta\in (0,\delta_0-\metric[x_0][x])\),
    \(\forall u:(-1,1)\rightarrow \DAinfty\) an \((A-\omega_0,x,\delta)\)-heat function,
    \begin{equation}\label{Eqn::Assump::Hyp::HypoellipticityIIBound}
        \LpNorm{\partial_t u}{p_0}[(-a_2,a_2)\times \MetricBall{x}{a_2\delta}][dt\times d\measure]
        \leq C_1 \LpNorm{u}{p_0}[(-1,1)\times \MetricBall{x}{\delta}][dt\times d\measure],
    \end{equation}
    and we assume the same for \(A,p_0,\NSubsetx\) replaced with \(A^{*},p_0', \NSubsety\).
    See Remark \ref{Rmk::Assumption::RemoveQuantifers} for an assumption which implies this, is easier
    to understand, and often appears in applications.
\end{assumption}

\begin{remark}
    Assumptions \ref{Assumption::HypoellipticityI} and \ref{Assumption::HypoellipticityII} are our main ``hypoellipticity'' assumptions.
    Indeed, if \(\partial_t -A\) is a hypoelliptic partial differential operator\footnote{A partial differential operator, \(\OpP\), is said to
    be hypoelliptic if whenever \(\OpP u\) is smooth near a point for a distribution \(u\), then \(u\) must also be smooth near that point. See \cite[Chapter 52]{TrevesTopologicalVectorSpacesDistributionsAndKernels}
    for a discussion of hypoellipticity. In particular, equation (52.1) of that reference shows that hypoelliptic operators always satisfy estimates 
    like 
    \eqref{Eqn::Assump::Hyp::HypoellipticityIBound} and \eqref{Eqn::Assump::Hyp::HypoellipticityIIBound}
    at a single scale.} and \(u(t):I\rightarrow \DAinfty\) is an \((A-\omega_0, x, \delta)\) heat function, then using that
    \(e^{\omega_0 t}(\partial_t-A)e^{-\omega_0t}=\partial_t-(A-\omega_0)\), we see \(u(\delta^{-2\kappa}t,x)\) is smooth on \((\delta^{2\kappa}I)\times \MetricBall{x}{\delta}\);
    and therefore \(u(t,x)\) is smooth on \(I\times \MetricBall{x}{\delta}\).
    Then, if \(X\) is a partial differential operator, the left hand sides of \eqref{Eqn::Assump::Hyp::HypoellipticityIBound} and \eqref{Eqn::Assump::Hyp::HypoellipticityIIBound}
    are finite. Assumptions \ref{Assumption::HypoellipticityI} and \ref{Assumption::HypoellipticityII} take this and assume that it is true in a
    uniform and ``scale-invariant'' way.
\end{remark}

\begin{remark}\label{Rmk::Assumption::Hypo::XandpartialtCommute}
    Assumption \ref{Assumption::HypoellipticityI} \ref{Item::Assumption::HypoellipticityI::CommuteDerivs} informally says that \(\partial_t\) and \(X\)
    commute--this is immediate in our application in Section \ref{Section::ApplicationNew} where \(X\) is a partial differential operator in the 
    \(x\)-variable.
\end{remark}

\begin{remark}\label{Rmk::Assumption::RemoveQuantifers}
    Both Assumptions \ref{Assumption::LocalDoubling} and \ref{Assumption::HypoellipticityII} involve
    several quantifiers.  In the main application we describe (see Section \ref{Section::ApplicationNew}) a stronger property is true, which is
    perhaps easier to understand.  Namely,
    set 
    \begin{equation*}
        \Omega_x:=\bigcup_{x_0\in \NSubsetx} \MetricBall{x_0}{\delta_0}. 
    \end{equation*}
    Then, 
    one can assume instead that
    \eqref{Eqn::LocalDoubling::DoublingEquation} holds \(\forall x\in \Omega_x\), \(\forall \delta\in (0,\delta_0)\)
    (and the same with \(\NSubsetx\) replaced with \(\NSubsety\)).
    Similarly, one can assume instead that 
    \eqref{Eqn::Assump::Hyp::HypoellipticityIIBound} holds
    \(\forall x\in \Omega_x\), \(\forall \delta\in (0,\delta_0)\)
    (and the same for \(A, p_0, \NSubsetx\) replaced with \(A^{*}, p_0', \NSubsety\)).
\end{remark}

\section{The Main Result}\label{Section::Results}
We take the setting and assumptions as in Section \ref{Section::Assumptions}.

Our main theorem (Theorem \ref{Thm::Results::MainThm}) concerns the integral kernel \(X\Yb K_t(x,y)\), which we define as the integral kernel of the operator \(XT(t)Y^{*}=XT(t/2) (YT(t/2)^{*})^{*} \); however
this operator does not a priori make sense. Our first theorem shows that under our assumptions, we can define this operator.

\begin{theorem}\label{Thm::Results::DefineOperator}
    There are open neighborhoods \(U_x,U_y\subseteq (0,\infty)\times \MetricSpace\) of \((0,\infty)\times \NSubsetx\) and \((0,\infty)\times \NSubsety\), respectively,
    such that the following holds.
    For \(t>0\), let \(U_{x,t}:=\left\{ v\in \MetricSpace : (t,v)\in U_x \right\}\) and similarly for \(U_{y,t}\).
    Then,
    \begin{equation*}
        \OpQx(t):u\mapsto \left( XT(t)u \right)\big|_{U_{x,t}}, \quad \OpQy(t):u\mapsto \left( Y T(t)^{*} u \right)\big|_{U_{y,t}},
    \end{equation*}
    initially defined as operators \(\DAinfty\rightarrow \LpSpace{p_0}[U_{x,t}][\measure]\) and \(\DAsinfty\rightarrow \LpSpace{p_0'}[U_{y,t}][\measure]\), respectively,
    extend to continuous operators
    \begin{equation*}
        \OpQx(t):\LpSpace{p_0}[\MetricSpace][\measure]\rightarrow \ClocSpace{U_{x,t}},\quad \OpQy(t):\LpSpace{p_0'}[\MetricSpace][\measure]\rightarrow \ClocSpace{U_{y,t}}.
    \end{equation*}
\end{theorem}

We define \(XT(t)Y^{*}:=\OpQx(t/2)\OpQy(t/2)^{*}:\ClocSpace{U_{y,t}}^{*}\rightarrow \ClocSpace{U_{x,t}}\).
Since \(\ClocSpace{U_{y,t}}\) denotes the space of  continuous functions on \(U_{y,t}\),
the Dirac delta function, \(\delta_{y_0}(y)\), is an element of \(\ClocSpace{U_{y,t}}^{*}\),
for each \(y_0 \in U_{y,t}\); we may therefore consider
\(XT(t)Y^{*}\delta_{y_0}\), which defines a continuous function on \(U_{x,t}\).


\begin{definition}\label{Defn::Results::AdmissibleConsts}
    We say \(C\in \R\) is an admissible constant if \(C\) can be chosen to depend only on 
    \(\kappa\), \(D_1\), \(C_1\), \(a_1\), \(a_2\), and \(M_0\).
    We say \(C\in \R\) is a \(j\)-admissible constant if \(C\) can be chosen to depend only on anything an admissible constant can depend on,
    \(j\), and \(D_2\).
    We write \(A\lesssim B\) (respectively, \(A\lesssim_j B\)) if \(A\leq C B\) where \(C\geq 0\) is an admissible (respectively, \(j\)-admissible) constant.
    We write
    \(A\approx B\) for \(A\lesssim B\) and \(B\lesssim A\).
\end{definition}

\begin{theorem}\label{Thm::Results::MainThm}
        There exists an open neighborhood \(U\subseteq (0,\infty)\times \MetricSpace\times \MetricSpace\) of \((0,\infty)\times\NSubsetx\times \NSubsety\)
        and a unique function \(X\Yb K_{t}(x,y):U\rightarrow \C\) such that 
        \begin{equation*}
            X T(t) Y^{*} \delta_{y}(x) = X\Yb K_t(x,y), \quad \forall (t,x,y)\in U.
        \end{equation*}
        This function satisfies:
        \begin{itemize}
            \item For each fixed \(x\) and \(y\), \(t\mapsto X \Yb K_t(x,y)\) is smooth on its domain.
            \item For each \(j\in \N\), and 
            each fixed \(y\), \((t,x)\mapsto \partial_t^{j}X \Yb K_t(x,y)\) is continuous on its domain and for each fixed \(x\),
            \((t,y)\mapsto \partial_t^{j}X\Yb K_t(x,y)\) is continuous on its domain.
        \end{itemize}
        There exists an admissible constant \(c>0\) such that \(\forall j\in \N\), \(t>0\), \(x\in \NSubsetx\), \(y\in \NSubsety\),
        \begin{equation}\label{Eqn::Results::MainGaussianBounds}
            \begin{split}
            \left| \partial_t^j X\Yb K_t(x,y) \right|
            \lesssim_j &
            e^{\omega_0 t} 
            \left( \left| \omega_0 \right| + \left( \left( \metric[x][y]+t^{1/2\kappa} \right)\wedge \delta_0 \right)^{-2\kappa} \right)^j
            \\& \times S_X\left(  \left( \metric[x][y]+t^{1/2\kappa} \right)\wedge \delta_0 \right)^{-1}
            S_Y\left(  \left( \metric[x][y]+t^{1/2\kappa} \right)\wedge \delta_0 \right)^{-1}
            \\&\times \exp\left( -c \left( \frac{\left( \metric[x][y]\wedge \delta_0 \right)^{2\kappa}}{t} \right)^{1/(2\kappa-1)} \right)
            \\&\times \measure\left( \BMetricBall{x}{ \left( \metric[x][y] + t^{1/2\kappa} \right)\wedge \delta_0 }  \right)^{-1/p_0}
            \measure\left( \BMetricBall{y}{ \left( \metric[x][y] + t^{1/2\kappa} \right)\wedge \delta_0 }  \right)^{-1/p_0'}.
            \end{split}
        \end{equation}
        Here, and in the rest of the paper, \(a\wedge b=\min\left\{ a,b \right\}\).
\end{theorem}

\begin{remark}\label{Rmk::Results::CanChangeToSingleMeasureInMainEstimate}
    When \(\metric[x][y] + t^{1/2\kappa}\leq \delta_0/2\), one can replace
    \begin{equation*}
        \measure\left( \BMetricBall{x}{ \left( \metric[x][y] + t^{1/2\kappa} \right)\wedge \delta_0 }  \right)^{-1/p_0}
            \measure\left( \BMetricBall{y}{ \left( \metric[x][y] + t^{1/2\kappa} \right)\wedge \delta_0 }  \right)^{-1/p_0'}
    \end{equation*}
    in \eqref{Eqn::Results::MainGaussianBounds}
    with the more familiar quantity
    \begin{equation*}
        \measure\left( \BMetricBall{x}{ \metric[x][y] + t^{1/2\kappa}  }  \right)^{-1},
    \end{equation*}
    because in this case,
    \begin{equation*}
        \measure\left( \BMetricBall{x}{ \metric[x][y] + t^{1/2\kappa}  }  \right) \approx \measure\left( \BMetricBall{y}{ \metric[x][y] + t^{1/2\kappa}  }  \right)
    \end{equation*}
    by Assumption \ref{Assumption::LocalDoubling}.
\end{remark}

\begin{remark}\label{Rmk::Results::WorksForMatrixKernels}
    Theorem \ref{Thm::Results::MainThm} holds more generally with \(\LpSpace{p_0}[\MetricSpace][\measure]\) replaced with
    \(\LpSpace{p_0}[\MetricSpace][\measure][\C^N]\) for some \(N\); with the same proof.
    In this case, \(K_t(x,y)\) takes values in \(N\times N\) complex matricies.
    There are many other possible such generalizations (for example, one could work with vector bundles over a manifold \(\Manifold\))
    that can also be studied with the same proof.
\end{remark}

\begin{remark}\label{Rmk::Results::cIsAdmissible}
    The constant \(c>0\) in Theorem \ref{Thm::Results::MainThm} is an admissible constant
    and therefore does not depend on \(j\) or on
    the operators \(X\) and \(Y\) except via the constant \(a_1\) (see Definition \ref{Defn::Results::AdmissibleConsts}).
    This is important in our main application (see 
    Remark \ref{Rmk::Subellip::MaxSub::ChoicesInApplication}) where we apply this estimate for an infinite number of \(X\) and \(Y\) and all \(j\),
    and use that \(c>0\) does not change (because we show that the same constant \(a_1\) can be used in each application of Theorem \ref{Thm::Results::MainThm}).
\end{remark}

\begin{remark}
    See Remark \ref{Rmk::Subellip::MaxSub::ChoicesInApplication} for choices of \(X\) and \(Y\)
    in a concrete application.
\end{remark}

\begin{remark}
    Admissible constants do not depend on \(p_0\) (see Definition \ref{Defn::Results::AdmissibleConsts}). If one keeps track of constants
    in the proofs which follow, there is often a factor of \(C^{1/p_0}\), where \(C\geq 1\) is an admissible constant.
    Since \(C^{1/p_0}\leq C\), this can be bounded by an admissible constant.
\end{remark}

\section{Application: Subelliptic PDEs}\label{Section::ApplicationNew}
The results in this paper apply to a wide range of \emph{maximally subelliptic}
partial differential operators; these are operators defined in terms of
H\"ormander vector fields.  We state our general result in Theorem \ref{Thm::SubMainThm::MainThm}, and give 
several examples in Proposition \ref{Prop::Subellip::Examples::Examples}. In this introductory section, we describe some
easy to understand operators where our general result applies.
Moreover, as described in Corollary \ref{Cor::Subellip::MaxSub::HeatImpliesMaxSub}, this method gives a new approach to proving the maximal subellipticity
of certain operators.

Let \(\Manifold\) be a connected, smooth manifold of dimension \(n\), and let \(\measure\) be a smooth,
strictly positive density on \(\Manifold\)--in particular,
\(\measure\) induces a \(\sigma\)-finite measure on \(\Manifold\)
which in any local coordinates is given by a smooth, positive function
times Lebesgue measure.
Let \(W_1,\ldots, W_r\) be smooth vector fields on \(\Manifold\) satisfying
\emph{H\"ormander's condition}: the Lie algebra generated by \(W_1,\ldots, W_r\)
spans the tangent space at every point. 

We use \(*\) to denote the formal \(\LpSpace{2}[\Manifold][\mu]\) adjoint
(when working with partial differential operators), or the true \(\LpSpace{2}[\Manifold][\mu]\) adjoint
(when working with abstract densely defined operators). For all the operators we consider,
whenever both interpretations make sense, the two interpretations agree; though this sometimes requires proof.
When there is ambiguity,
we are explicit about which we mean.

Several examples where the results in this section apply are described in Proposition \ref{Prop::Subellip::Examples::Examples}.
For now, the reader can keep in mind the following examples which are described in more detail in Proposition \ref{Prop::Subellip::Examples::Examples}:
\begin{enumerate}[(i)]
    \item\label{Item::SubNewIntro::HorSublap} Fix \(n_1,\ldots, n_r\in \Nplus\). Consider \(\sum_{j=1}^r \left( W_j^{n_j} \right)^{*}W_j^{n_j}\).
        When \(n_1=\cdots=n_r=1\), this is known as H\"ormander's sub-Laplacian and was the operator studied
        in \cite{JerisonSanchezCalleEstimatesForTheHeatKernelForASumOfSquaresOfVectorFields}.
        For general \(n_j\), though in the special case of translation invariant operators on some nilpotent Lie groups,
        this was the operator studied in \cite{HebischSharpPointwiseEstimateForTheKernelsOfTheSemigroupGeneratedBySumsOfEvenPowersOfVectorFieldsOnHomogeneousGroups}.
    \item\label{Item::SubNewIntro::MaxSubStarMaxSub} 
    More generally than \ref{Item::SubNewIntro::HorSublap}, we consider any operator of the form
        \(\opP^{*}\opP\), where \(\opP\) is a maximally subelliptic partial differential operator (see Definition \ref{Defn::SubEllip::MaxSub::MaxSub}).
    \item Suppose \(r=2\) and let \(n_1,n_2\in \Nplus\). We consider the operator
        \(\opL_{\alpha} = \left( W_1^{n_1} \right)^{*}W_1^{n_1} + \left( W_2^{n_2} \right)^{*}W_2^{n_2}+\alpha \left( W_1^{n_1} \right)^{*} W_2^{n_2}\),
        where \(\alpha\in \C\) with \(|\alpha|<2\). Note that, unlike \ref{Item::SubNewIntro::HorSublap}
        and \ref{Item::SubNewIntro::MaxSubStarMaxSub}, \(\opL_{\alpha}\) is not symmetric.
        We also consider more general non-symmetric operators in Proposition \ref{Prop::Subellip::Examples::Examples}.
    \item As we describe in Proposition \ref{Prop::Subellip::Examples::Examples} \ref{Item::Subellip::Examples::Examples::PerturbLowerOrder}, we can often add lower order terms to the above examples and obtain new examples.
\end{enumerate}

Let \(\opL\) be any of the above operators.
Let 
\((A,\Domain[A])\) 
be a generator of a strongly continuous semigroup \(T(t)\) on  \(\LpSpace{2}[\Manifold][\measure]\)
with \(\CzinftySpace[\Manifold]\subseteq \Domain[A]\) and  \(A=-\opL\big|_{\Domain[A]}\),
where \(\opL\) is taken in the sense of distributions
(as Propositions \ref{Prop::SubEllip::Examples::ConditionMConclusions} and \ref{Prop::Subellip::Examples::Examples}  show, 
such a choice of \(T(t)\) always exists in the above examples).
In Theorem \ref{Thm::SubMainThm::MainThm}, we show that
\begin{equation*}
    T(t) f(x) = \int K_t(x,y) f(y)\: d\measure(y),
\end{equation*}
where, on each compact subset of \(\Manifold\), \(K_t(x,y)\) satisfies higher order Gaussian bounds in terms of an adapted
Carnot--Carathe\'odory metric. Moreover, the same is true of appropriate derivatives of \(K_t(x,y)\).

As we describe in Corollary \ref{Cor::Subellip::MaxSub::HeatImpliesMaxSub}, this gives an approach to constructing a parametrix for \(\opL\) which has good estimates;
namely, \(\int_0^1 T(t)\: dt\) is such a parametrix.

    \subsection{The main theorem for subelliptic PDEs}\label{Section::Subellip::MainThm}
    As above, let \(\Manifold\) be a connected, smooth manifold
of dimension \(n\),
endowed with a smooth, strictly positive density \(\measure\);
with an abuse of notation, we denote the associated measure by \(\measure\).

\begin{definition}
    Let \(W=\left\{ W_1,\ldots,W_r \right\}\) be a finite collection smooth vector fields on \(\Manifold\).
    We say \(W\) satisfies H\"ormander's condition on \(\Manifold\) if the Lie algebra generated by \(W_1,\ldots, W_r\)
    spans the tangent space at every point.
\end{definition}

\begin{definition}
    Let \(W=\left\{ W_1,\ldots, W_r \right\}\) be H\"ormander vector fields on \(\Manifold\).
    To each \(W_j\) assign a formal degree \(\Wd_j\in \Nplus\).
    We write \(\WWd=\left\{ (W_1,\Wd_1),\ldots, (W_r,\Wd_r) \right\}\) and say \(\WWd\)
    is a set of H\"ormander vector fields with formal degrees.
\end{definition}

\begin{definition}\label{Defn::SubMainTheorem::OrderedMultiIndex}
    Let \(\alpha=(\alpha_1,\ldots, \alpha_L)\in \left\{ 1,\ldots, r \right\}^L\) be a list of elements of \(\left\{ 1,\ldots, r \right\}\).
    We set \(W^{\alpha}=W_{\alpha_1}W_{\alpha_2}\cdots W_{\alpha_L}\), \(|\alpha|=L\), and \(\degWd(\alpha)=\Wd_{\alpha_1}+\Wd_{\alpha_2}+\cdots+\Wd_{\alpha_L}\).
\end{definition}

Let \(\WWd=\left\{ (W_1,\Wd_1),\ldots, (W_r,\Wd_r) \right\}\) be a set of H\"ormander vector fields with formal degrees on \(\Manifold\).
\(\WWd\) induces a Carnot-Carath\'eodory metric, \(\metric\), on \(\Manifold\). We define this metric by
defining its metric balls. Namely, for \(x\in \Manifold\) and \(\delta>0\),
\begin{equation*}
\begin{split}
     \MetricBall{x}{\delta}:= \Bigg\{
        y\in \Manifold \: \Bigg|\:& \exists \gamma:[0,1]\rightarrow \Manifold, \gamma(0)=x,\gamma(1)=y,
        \\&\gamma \text{ is absolutely continuous},
        \\&\gamma'(t) =\sum_{j=1}^r e_j(t)\delta^{\Wd_j} W_j(\gamma(t))\text{ almost everywhere},
        \\& e_j \in \LpSpace{\infty}[{[0,1]}], \BLpNorm{\sum_{j=1}^r |e_j|^2}{\infty}[{[0,1]}]<1
     \Bigg\},
\end{split}
\end{equation*}
and we let \(\metric\) be the corresponding metric.
It is a classical theorem of Chow \cite{ChowUberSystemeVonLinearenPartillen}
that \(\metric\) is indeed a metric on \(\Manifold\)--see  \cite[Lemma 3.1.7]{StreetMaximalSubellipticity} for an exposition.

Fix \(\kappa\in \Nplus\) such that \(\Wd_j\) divides \(\kappa\), \(\forall j\).
For each \(\alpha\) and \(\beta\) with \(\degWd(\alpha),\degWd(\beta)\leq \kappa\),
let \(a_{\alpha,\beta}\in \CinftySpace[\Manifold]\). 
Define a partial differential operator, \(\opL\), acting on distributions by
\begin{equation}\label{Eqn::SubMainThem::FormulaForL}
    \opL:=\sum_{\degWd(\alpha),\degWd(\beta)\leq \kappa}\left( W^{\alpha} \right)^{*} a_{\alpha,\beta} W^{\beta},
\end{equation}
where \(*\) denotes the formal \(\LpSpace{2}[\Manifold][\measure]\) adjoint.


Our main assumption is the following maximal subellipticity-type assumption:
\begin{numberedassumption}\label{Assump::SubMainTheorem::MaxSubTypeIII}
    \(\forall \Omega\Subset \Manifold\) open and relatively compact,\footnote{We write \(A\Subset B\)
    to denote that \(A\) is a relatively compact subset of \(B\).}
    \(\exists C_1,C_2\geq 0\), \(\forall f\in \CzinftySpace[\Omega]\),
    \begin{equation*}
        \sum_{j=1}^r \BLpNorm{ W_j^{\kappa/\Wd_j} f}{2}[\Manifold][\measure]^2
        \leq C_1 \Real \Ltip*{f}{\opL f}[\Manifold][\measure]
        +C_2 \BLpNorm{f}{2}[\Manifold][\measure]^2.
    \end{equation*}
\end{numberedassumption}

\begin{theorem}\label{Thm::SubMainThm::MainThm}
    In the above setting, the following holds.
    Let 
\((A,\Domain[A])\) 
be a generator of a strongly continuous semigroup \(T(t)\) on  \(\LpSpace{2}[\Manifold][\measure]\)
with \(\CzinftySpace[\Manifold]\subseteq \Domain[A]\) and  \(A=-\opL\big|_{\Domain[A]}\), where \(\opL\)
is taken in the sense of distributions (such a semi-group
may not exist in general, and we assume it to exist\footnote{See Proposition \ref{Prop::Subellip::Examples::Examples} for examples where such a semi-group exists.}). Let \(\omega_0\in \R\) be as in \eqref{Eqn::Assump::Hypo::IntroduceM0Andomega0}.

    Then, there exists a unique smooth function
    \(K_t(x,y)\in \CinftySpace[(0,\infty)\times \Manifold\times \Manifold]\)
    such that
    \begin{equation}\label{Eqn::SubMainThm::DefineKt}
        T(t)g(x) = \int K_t(x,y) g(y)\: d\mu(y),\quad \forall g\in \LpSpace{2}[\Manifold][\mu].
    \end{equation}
    Moreover, for every \(\Compact\Subset \Manifold\) compact,
    \(\exists \delta_0=\delta_0(\Compact)\in (0,1]\), \(\exists c=c(\Compact)>0\),
    \(\forall \alpha,\beta\), \(\forall j\in \N\), \(\exists C=C(\Compact,\alpha,\beta,j)\),
    \(\forall t>0\), \(\forall x,y\in \Compact\),
    \begin{equation}\label{Eqn::SubMainThm::MainBound}
    \begin{split}
          \left| \partial_t^j W_x^{\alpha} W_y^{\beta} K_t(x,y) \right|
         \leq&
         C e^{\omega_0t}
         \left( \left( \metric[x][y]+t^{1/2\kappa} \right)\wedge \delta_0 \right)^{-2\kappa j -\degWd(\alpha)-\degWd(\beta)}
         \\&\times \exp\left( -c \left( \frac{\left( \metric[x][y]\wedge \delta_0 \right)^{2\kappa}}{t} \right)^{1/(2\kappa-1)} \right)
            \\&\times \measure\left( \BMetricBall{x}{ \left( \metric[x][y] + t^{1/2\kappa} \right)\wedge \delta_0 }  \right)^{-1/2}
            \measure\left( \BMetricBall{y}{ \left( \metric[x][y] + t^{1/2\kappa} \right)\wedge \delta_0 }  \right)^{-1/2}.
    \end{split}
    \end{equation}
\end{theorem}

See Section \ref{Section::Subellip::Proof} for the proof of Theorem \ref{Thm::SubMainThm::MainThm}.

\begin{remark}
    In Theorem \ref{Thm::SubMainThm::MainThm}, \(\delta_0>0\) is small and depends on \(\Compact\).
    If \(A\) is homogeneous under some dilation structure, then one may instead take \(\delta_0=\infty\) and prove the estimate
    for all \(x,y\in \Manifold\) (instead of on compact sets), using the same proof. For example, this is how the main result
    of Hebisch \cite{HebischSharpPointwiseEstimateForTheKernelsOfTheSemigroupGeneratedBySumsOfEvenPowersOfVectorFieldsOnHomogeneousGroups}
    is obtained (and our work is based on that paper).
\end{remark}

\begin{remark}
    In Theorem \ref{Thm::SubMainThm::MainThm}, many of the constants depend on the compact set \(\Compact\).
    This is because many qualitative properties automatically hold uniformly on compact sets. If one were to instead use appropriate quantitative assumptions
    which were uniform over \(\Manifold\), then one could replace \(\Compact\) with \(\Manifold\) throughout Theorem \ref{Thm::SubMainThm::MainThm}
    with the same proof.
\end{remark}

\begin{remark}
    One can replace 
    \begin{equation*}
        \measure\left( \MetricBall{x}{\left( \metric[x][y]+t^{1/2\kappa} \right)\wedge \delta_0} \right)^{-1/2}
        \measure\left( \MetricBall{y}{\left( \metric[x][y]+t^{1/2\kappa} \right)\wedge \delta_0} \right)^{-1/2}
    \end{equation*}
    in Theorem \ref{Thm::SubMainThm::MainThm} with the more familiar quantity
    \begin{equation*}
        \measure\left( \MetricBall{x}{\left( \metric[x][y]+t^{1/2\kappa} \right)\wedge \delta_0} \right)^{-1}.
    \end{equation*}
    Indeed, when \(\metric[x][y]+t^{1/2\kappa}\) is small, this follows from Remark \ref{Rmk::Results::CanChangeToSingleMeasureInMainEstimate},
    and when it is large then both quantities are \(\approx 1\)--see \cite[Corollary 3.3.9]{StreetMaximalSubellipticity}.
\end{remark}

\begin{remark}\label{Rmk::Subellip::MaxSub::ChoicesInApplication}
    We prove Theorem \ref{Thm::SubMainThm::MainThm} by applying 
    Theorem \ref{Thm::Results::MainThm} with the following choices:
    \begin{itemize}
        \item \(p_0=2\),
        \item \(\NSubsetx=\NSubsety=\Compact\),
        \item \(\delta_0=\delta_0(\Compact)>0\) will be a small number which is related to our scaling result (Theorem \ref{Thm::MaxSub::Scaling::NSWScaling}),
        \item Fix \(\psi\in \CzinftySpace[\Manifold]\) with \(\psi=1\) on a neighborhood of \(\Compact\). We use
            \(X=\psi W^{\alpha}\), \(Y=\psi W^{\beta}\), \(S_X(\delta)=\epsilon_{\alpha}\delta^{\degWd(\alpha)}\),
            \(S_X(\delta)=\epsilon_{\beta}\delta^{\degWd(\beta)}\), \emph{for each} \(\alpha\) and \(\beta\);
            here \(\epsilon_\alpha,\epsilon_\beta>0\) are small constants to be chosen later.
    \end{itemize}
    Thus, \eqref{Eqn::SubMainThm::MainBound} is proved by applying Theorem \ref{Thm::Results::MainThm}
    and infinite number of times (once for each \(\alpha\) and \(\beta\)). 
    We show the constant \(c=c(\Compact)>0\) is independent of \(\alpha\), \(\beta\), and \(j\)
    by using the idea described in Remark \ref{Rmk::Results::cIsAdmissible}.
\end{remark}

\begin{remark}
    The reader may notice a slight difference in the term 
    \(\left( \left( \metric[x][y]+t^{1/2\kappa} \right)\wedge \delta_0 \right)^{-2\kappa j}\)
    from
    \eqref{Eqn::SubMainThm::MainBound}
    compared to the similar term 
    \(\left( \left| \omega_0 \right| + \left( \left( \metric[x][y]+t^{1/2\kappa} \right)\wedge \delta_0 \right)^{-2\kappa} \right)^j\)
    from \eqref{Eqn::Results::MainGaussianBounds}.
    Roughly speaking, this is because we allow the implicit constant in \eqref{Eqn::SubMainThm::MainBound}
    to depend on \(\omega_0\), while we do not allow the same in \eqref{Eqn::Results::MainGaussianBounds}.
    This is described in more detail at the start of Section \ref{Section::MaxSub::CompleteProof}.
\end{remark}

Importantly, the assumptions of Theorem \ref{Thm::SubMainThm::MainThm}
are symmetric in \(\opL\) and \(\opL^{*}\), as the next proposition shows.

\begin{proposition}\label{Prop::Subellip::MaxSub::SymmetricAssumps}
    We have:
    \begin{enumerate}[(i)]
        \item\label{Item::Subellip::MaxSub::AdjointIsRightForm::opLStarForm}  If \(\opL\) is of the form \eqref{Eqn::SubMainThem::FormulaForL}, then \(\opL^{*}\) is of the form
            \eqref{Eqn::SubMainThem::FormulaForL}.
        \item\label{Item::Subellip::MaxSub::AdjointIsRightForm::opLStarMaxSub} If \(\opL\) satisfies Assumption \ref{Assump::SubMainTheorem::MaxSubTypeIII}, then \(\opL^{*}\)
    also satisfies Assumption \ref{Assump::SubMainTheorem::MaxSubTypeIII}.
        \item\label{Item::Subellip::MaxSub::AdjointIsRightForm::AdjointGenerator} If \((A,\Domain[A])\) is the generator for a strongly continuous semi-group \(T(t)\) on \(\LpSpace{2}[\Manifold][\mu]\), then
            \((A^{*}, \Domain[A^{*}])\) is the generator for the strongly continuous semi-group \(T(t)^{*}\).
    \end{enumerate}
    For the next two parts, 
    let \((A,\Domain[A])\) be a densely defined operator on \(\LpSpace{2}[\Manifold][\mu]\),
    and let \(\opL\) be a partial differential operator on \(\Manifold\) with smooth coefficients.
    Suppose \(A=-\opL\big|_{\Domain[A]}\), where \(\opL\) is taken in the sense of distributions. Then,
    \begin{enumerate}[resume*]
        \item\label{Item::Subellip::MaxSub::AdjointIsRightForm::ContainsCzinfty} \(\CzinftySpace[\Manifold]\subseteq \Domain[A^{*}]\) and \(A^{*}\big|_{\CzinftySpace[\Manifold]}=-\opL^{*}\big|_{\CzinftySpace[\Manifold]}\).
        \item\label{Item::Subellip::MaxSub::AdjointIsRightForm::EqualsDeriv} If \(\CzinftySpace[\Manifold]\subseteq \Domain[A]\), then \(A^{*}=-\opL^{*}\big|_{\Domain[A^{*}]}\).
    \end{enumerate}
\end{proposition}
\begin{proof}
    \ref{Item::Subellip::MaxSub::AdjointIsRightForm::opLStarForm} and \ref{Item::Subellip::MaxSub::AdjointIsRightForm::opLStarMaxSub} follow 
    immediately from the definitions.
    \ref{Item::Subellip::MaxSub::AdjointIsRightForm::AdjointGenerator} follows from 
    \cite[Chapter 2, Section 2.5]{EngelNagelAShortCourseOnOperatorSemigroups}.

    We recall \(\Domain[A^{*}]=\left\{ u\in \LpSpace{2} : \exists w\in \LpSpace{2}, \Ltip{w}{v}=\Ltip{u}{Av},\forall v\in \Domain[A]\right\}\),
    and \(A^{*}u\) is defined to be the unique \(w\in \LpSpace{2}\) with \(\Ltip{w}{v}=\Ltip{u}{Av}\), \(\forall v\in \Domain[A]\).

    \ref{Item::Subellip::MaxSub::AdjointIsRightForm::ContainsCzinfty}:
    Suppose \(g\in \CzinftySpace[\Manifold]\). Then,
    \begin{equation*}
        \Ltip{g}{Av}=\Ltip{g}{-\opL v} = \Ltip{-\opL^{*} g}{v},\quad \forall v\in \Domain[A].
    \end{equation*}
    \ref{Item::Subellip::MaxSub::AdjointIsRightForm::ContainsCzinfty} follows.

    \ref{Item::Subellip::MaxSub::AdjointIsRightForm::EqualsDeriv}: Suppose \(\CzinftySpace[\Manifold]\subseteq \Domain[A]\).
    Then, for \(u\in \Domain[A^{*}]\) and \(g\in \CzinftySpace[\Manifold]\), we have
    \begin{equation*}
        \Ltip{A^{*}u}{g}=\Ltip{u}{A g} = \Ltip{u}{-\opL g}=\Ltip{-\opL^{*} u}{g},
    \end{equation*}
    establishing \ref{Item::Subellip::MaxSub::AdjointIsRightForm::EqualsDeriv}.
\end{proof}

\begin{remark}\label{Rmk::MaxSub::SubLapalcian::OtherProofs}
    There are many methods in the literature to prove results similar Theorem \ref{Thm::SubMainThm::MainThm}--too many to list here.
    Especially when \(\kappa=1\) (in particular when \(\opL\) is H\"ormander's sub-Laplacian), there are a number of possible proofs. For example, Melrose \cite{MelrosePropagationForTheWaveGroup} presents
    a proof for H\"ormander's sub-Laplacian using the finite propagation speed of the wave equation (which only holds when \(\kappa=1\));
    see  \cite[Section 2.6]{StreetMultiParameterSingularIntegrals} for an exposition of this approach, which is an instance of a general
    phenomenon (see \cite{SikoraRieszTransformGaussianBoundsAndTheMethodOfWaveEquation}).
    See, also, 
    \cite{JerisonSanchezCalleEstimatesForTheHeatKernelForASumOfSquaresOfVectorFields} on which this work is based and
    \cite{KusuokaStroockApplicationsOfTheMalliavinCalculusIII,VaropoulosSaloffCosteCoulhonAnalysisAndGeometryOnGroups}.
    For higher order operators (\(\kappa>1\)), 
    less has been done, and many of the proofs from \(\kappa=1\) do not generalize to higher \(\kappa\). We have already mentioned the work of Hebisch \cite{HebischSharpPointwiseEstimateForTheKernelsOfTheSemigroupGeneratedBySumsOfEvenPowersOfVectorFieldsOnHomogeneousGroups},
    on which this work is based. See the work of Dungey 
    \cite{DungeyHigherOrderOperatorsAndGaussianBoundsOnLieGroupsOfPolynomialGrowth}
    for a different approach on Lie groups based on cut-off functions and an idea of Davies \cite{DaviesUniformlyEllipticOperatorsWithMeasurableCoefficients}.
    Dungey and Davies' approach also seems likely to be useful for general maximally subelliptic (or similar) operators, especially those with rough coefficients;
    though is more closely tied to the form of the operator being studied and seems harder to adapt to settings 
    like Section \ref{Section::MaxSub::BoundaryValueProblems}. 
    See \cite{BarbatisSharpHeatKernelBoundsAndFinslerTypeMetrics,BarbatisExplicitEstimatesOnTheFundamentalSolution,AuscherTchamitchianSquareRootProblemForDivergenceOperators}
    and the references in \cite[Section 5.2]{DaviesLpSpectralTheoryOfHigherOrderEllipticDifferentialOperators}
    for some other results in the elliptic setting whose proofs might shed light on subelliptic operators.
\end{remark}

    \subsection{Maximal Subellipticity}\label{Section::Subellip::MaximalSubellip}
    An important lens through which to view Theorem \ref{Thm::SubMainThm::MainThm}
is via the theory of maximal subellipticity.  As we explain,
Theorem \ref{Thm::SubMainThm::MainThm} can be used to deduce higher order Gaussian
bounds for some general maximally subelliptic heat equations.
This provides an approach to obtain parametrices and sharp results for maximally subelliptic
PDEs, without using  Lie groups.
In fact, this was the approach used by the author in \cite[Chapter 8]{StreetMaximalSubellipticity};
however, in that reference the proof closely used many properties of the partial differential operators
in question, and it was not made clear that this is a more general approach.
This general approach will be used by the author in a forthcoming work on boundary value problems
(see Section \ref{Section::MaxSub::BoundaryValueProblems}).

Let \(\WWd=\left\{ \left( W_1,\Wd_1 \right),\ldots, \left( W_r,\Wd_r \right) \right\}\)
be H\"ormander vector fields with formal degrees on a smooth manifold \(\Manifold\),
and let \(\measure\) be a smooth, strictly positive density on \(\Manifold\).
Fix \(\kappa,N\in \Nplus\) such that \(\Wd_j\) divides \(\kappa\) for every \(j\).
For each \(\alpha\) with \(\degWd(\alpha)\leq \kappa\) let \(a_\alpha\in \CinftySpace[\Manifold][\C^N]\).
Define
\begin{equation}\label{Eqn::SubEllip::MaxSub::FormulaForP}
    \opP:=\sum_{\degWd(\alpha)\leq \kappa} a_\alpha(x) W^{\alpha}.
\end{equation}

\begin{definition}\label{Defn::SubEllip::MaxSub::MaxSub}
    We say \(\opP\) given by \eqref{Eqn::SubEllip::MaxSub::FormulaForP} is maximally subelliptic
    of degree \(\kappa\) with respect to \(\WWd\) if for every relatively compact open set
    \(\Omega\Subset \Manifold\), there exists \(C_\Omega\geq 0\) such that
    \(\forall f\in \CzinftySpace[\Omega]\),
    \begin{equation*}
        \sum_{j=1}^r \BLpNorm{ W_j^{\kappa/\Wd_j} f }{2}[\Manifold][\mu]
        \leq C_\Omega
        \left( \BLpNorm{\opP f}{2}[\Manifold][\mu][\C^N] + \BLpNorm{f}{2}[\Manifold][\mu]   \right).
    \end{equation*}
\end{definition}

A maximally subelliptic operator of degree \(\kappa\) has a left parametrix which is
locally a ``singular integral operator of order \(-\kappa\)''.
This is due to many authors--initial results were due to Folland and Stein \cite{FollandSteinEstimatesForTheDbarComplexAndAnalysis}
and Rothschild and Stein \cite{RothschildSteinHypoellipticDifferentialOperatorsAndNilpotentGroups},
and some important examples were due to Nagel, Rosay, Stein, and Wainger \cite{NagelRosaySteinWaingerEstimatesForTheBergmanAndSzego}
and Koenigh \cite{KoenigOnMaximalSobolevAndHolderEstiamtes}; see \cite[Chapter 8]{StreetMaximalSubellipticity}
for a general theory and a history of these ideas.
In the notation of \cite[Definition 5.11.1]{StreetMaximalSubellipticity},
there is an operator \(S\in \mathscr{A}^{-\kappa}_{\mathrm{loc}}(\WWd)\)
such that \(S\opP\equiv I\) modulo smoothing operators.
In fact, the existence of such a parametrix is equivalent to \(\opP\)
being maximally subelliptic \cite[Theorem 8.1.1(i)\(\Leftrightarrow\)(vii)]{StreetMaximalSubellipticity}.


\begin{corollary}\label{Cor::Subellip::MaxSub::HeatImpliesMaxSub}
    Let \(\opL\) satisfy the assumptions of Theorem \ref{Thm::SubMainThm::MainThm};
    in particular we assume
    Assumption \ref{Assump::SubMainTheorem::MaxSubTypeIII} and
    the existence of a semi-group \(T(t)\) with generator \(\left( A,\Domain[A] \right)\)
    with \(\CzinftySpace[\Manifold]\subseteq \Domain[A]\) and
     \(A=-\opL\big|_{\Domain[A]}\).
    Then \(\opL\) is maximally subelliptic of degree \(2\kappa\) with respect to \(\WWd\).
    Moreover, if \(T(t)\) is the semi-group from Theorem  \ref{Thm::SubMainThm::MainThm},
    then \(S:=\int_0^1 T(t)\: dt\) is a two-sided parametrix for \(\opL\) and
     \(S\in \mathscr{A}^{-2\kappa}_{\mathrm{loc}}(\WWd)\).
\end{corollary}
\begin{proof}[Comments on the proof]
    Let \((A,\Domain[A])\) be the generator for \(T(t)\).
    \cite[Chapter 2, Lemma 1.3]{EngelNagelAShortCourseOnOperatorSemigroups} shows
    \begin{equation*}
        \opL S f= -AS f = f-T(1) f, \quad\forall f\in \LpSpace{2}[\Manifold][\measure],
    \end{equation*}
    \begin{equation*}
        S \opL f= -S Af = f-T(1) f, \quad\forall f\in \Domain[A]\supseteq \CzinftySpace[\Manifold].
    \end{equation*}
    Since \(T(1)\) is integration against a smooth function (by Theorem \ref{Thm::SubMainThm::MainThm}), 
    this shows that \(S\) is a two-sided parametrix for \(\opL\).

    Using the Gaussian bounds guaranteed by Theorem \ref{Thm::SubMainThm::MainThm},
    the proof of \cite[Corollary 5.11.16]{StreetMaximalSubellipticity} shows that
    \(S\in \mathscr{A}^{-2\kappa}_{\mathrm{loc}}(\WWd)\) (see that reference for details); in that reference it was assumed that \((-A,\Domain[A])\)
    was a non-negative self-adjoint operator, though the same proof works more generally for generators
    of strongly continuous semigroups. From here, that \(\opL\) is maximally subelliptic of degree \(2\kappa\) with respect to \(\WWd\)
    follows from \cite[Theorem 8.1.1(vii)\(\Rightarrow\)(i)]{StreetMaximalSubellipticity}.
\end{proof}

\begin{remark}
    Corollary \ref{Cor::Subellip::MaxSub::HeatImpliesMaxSub} 
    gives a new way to show some operators are maximally subelliptic.
    See Remark \ref{Rmk::Subellip::Examples::ExamplesAreMaxSub}.
\end{remark}

        \subsubsection{Parametricies}\label{Section::Subellip::MaximalSub::Parametricies}
        Corollary \ref{Cor::Subellip::MaxSub::HeatImpliesMaxSub} provides parametrices
for a wide range of subelliptic PDEs (see Proposition \ref{Prop::Subellip::Examples::Examples} and Remark \ref{Rmk::Subellip::Examples::LeftParametrixForP} for examples).
Dating back to the work of Folland \cite{FollandSubellipticEstimatesAndFunctionSpacesOnNilpotentLieGroups}, Rothschild and Stein \cite{RothschildSteinHypoellipticDifferentialOperatorsAndNilpotentGroups},
Fefferman and Phong \cite{FeffermanPhongSubellipticEigenvalueProblems}, Fefferman and S\'{a}nchez-Calle \cite{FeffermanSanchezCalleFundamentalSolutionsForSecondOrderOperators},
and Nagel, Stein, and Wainger \cite{NagelSteinWaingerBallsAndMetricsDefinedByVectorFieldsIBasicProperties} the following outline 
is often used to create parametricies for
maximally subelliptic PDEs (or other similar kinds of subelliptic PDEs):
\begin{enumerate}[(1)]
    \item\label{Item::MaxSub::Intro::ProveSubelliptic} Prove an a priori subelliptic estimate for \(\OpP\), using methods pioneered by H\"ormander \cite{HormanderHypoellipticSecondOrder}, Kohn \cite{KohnLecturesOnDegenerateEllipticProblems},
        and others. These estimates need not be sharp; sharp estimates are established in \ref{Item::MaxSub::Intro::SharpEstimates}.
    \item Show that a similar subelliptic estimate holds at every Carnot-Carath\'eodory scale, uniformly in the scale, by rescaling the proof in \ref{Item::MaxSub::Intro::ProveSubelliptic}.
        One general framework for this rescaling was introduced by Nagel, Stein, and Wainger \cite{NagelSteinWaingerBallsAndMetricsDefinedByVectorFieldsIBasicProperties}. See Section \ref{Section::MaxSub::Scaling}.
    \item\label{Item::MaxSub::Intro::Parametrix} Use the subelliptic estimate at each scale to provide good bounds on a parametrix, to show that this parametrix
        is a ``singular integral operator'' as referenced above.
    \item\label{Item::MaxSub::Intro::SharpEstimates} Use the parametrix from \ref{Item::MaxSub::Intro::Parametrix} to obtain sharp estimates in a variety of function spaces;
        see \cite{FollandSteinEstimatesForTheDbarComplexAndAnalysis}, \cite{FollandSubellipticEstimatesAndFunctionSpacesOnNilpotentLieGroups}, \cite{RothschildSteinHypoellipticDifferentialOperatorsAndNilpotentGroups}, \cite{NagelRosaySteinWaingerEstimatesForTheBergmanAndSzego}, \cite{KoenigOnMaximalSobolevAndHolderEstiamtes}, and  \cite[Chapter 8]{StreetMaximalSubellipticity}
        for examples.
    \item If relevant, use the parametrix from \ref{Item::MaxSub::Intro::Parametrix} or the sharp results from \ref{Item::MaxSub::Intro::SharpEstimates}
        to study related heat operators.  See, e.g., \cite{JerisonSanchezCalleEstimatesForTheHeatKernelForASumOfSquaresOfVectorFields}.
\end{enumerate}

A main difficulty when implementing this general outline is \ref{Item::MaxSub::Intro::Parametrix}.
For translation invariant, \textit{homogeneous}, maximally subelliptic operators on a graded nilpotent Lie group,
Folland \cite{FollandSubellipticEstimatesAndFunctionSpacesOnNilpotentLieGroups} used the homogenity to create
a homogeneous fundamental solution and achieve \ref{Item::MaxSub::Intro::Parametrix}. Outside a translation invariant, homogenenous setting,
step \ref{Item::MaxSub::Intro::Parametrix} can be more difficult.

Rothschild and Stein \cite{RothschildSteinHypoellipticDifferentialOperatorsAndNilpotentGroups} gave a general procedure
to lift a maximally subelliptic operator to a high dimensional graded nilpotent Lie group; see also \cite{GoodmanNilpotentLieGroupsStructureAndApplications}.
When the lifted operator is again maximally subelliptic, Folland's fundamental solution for the lifted operator can be used to create a parametrix.
Unfortunately, the lifted operator is not always maximally subelliptic; see  \cite[Section 4.5.7]{StreetMaximalSubellipticity}.
Recently 
Androulidakis, Moshen, and Yunken
\cite{AndroulidakisMoshenYunkenAPseudodifferentialCalculusForMaximallHypoelliptic}
have presented an approach using representation theory of nilpotent Lie groups to create parametrices for general maximally subelliptic operators.

When studying a subellipic PDE where approximating by a nilpotent Lie group is not available, one often needs to prove an additional
estimate to achieve step \ref{Item::MaxSub::Intro::Parametrix}. See  \cite[Proposition 3.2]{KoenigOnMaximalSobolevAndHolderEstiamtes}
for a typical example of this kind of estimate. In the examples we know of, this additional estimate is tailored to the case at hand; and when
working with new PDEs it can be 
very difficult (if not impossible) to establish the needed estimate with a priori methods.

Corollary \ref{Cor::Subellip::MaxSub::HeatImpliesMaxSub} uses a different method
to create a parametrix, which does not require such an additional estimate, and does not use Lie groups.
The proof of Corollary \ref{Cor::Subellip::MaxSub::HeatImpliesMaxSub} proceeds as follows:
\begin{enumerate}[(1')]
    \item\label{Item::MaxSub::Intro::NewMethod::ProveSubelliptic} Prove an a priori subelliptic estimate for \(\partial_t-A\) and \(\partial_t-A^{*}\); for example, using the methods of \cite{KohnLecturesOnDegenerateEllipticProblems}.
        As before, this estimate need not be sharp. See Section \ref{Section::Subellip::Proof::UnitScale} and in particular Remark \ref{Rmk::Subellip::Proof::epsilon0NotOptimal}.
    \item\label{Item::MaxSub::Intro::NewMethod::Scale} Show similar estimates hold at every scale, uniformly in the scale, by rescaling the proof in \ref{Item::MaxSub::Intro::NewMethod::ProveSubelliptic}; for example, by using
        the methods of Nagel, Stein, and Wainger \cite{NagelSteinWaingerBallsAndMetricsDefinedByVectorFieldsIBasicProperties}. See Section \ref{Section::MaxSub::Scaling} for the basic scaling result, and Proposition \ref{Prop::SubProof::CompleteProof::ScaledEstimates} for some scaled estimates.
    \item\label{Item::MaxSub::Intro::NewMethod::GaussianBounds} Use the estimates from \ref{Item::MaxSub::Intro::NewMethod::Scale} along with Theorem \ref{Thm::Results::MainThm} to deduce local higher order Gaussian bounds
        for the semigroup \(e^{tA}\). See Section \ref{Section::MaxSub::CompleteProof}.
    \item\label{Item::MaxSub::Intro::NewMethod::Parametrix} Create a parametrix for \(-A\) by integrating this semi-group: \(\int_0^1 e^{tA}\: dt\), and use the Gaussian bounds
        to show that this parametrix is a ``singular integral operator of order \(-2\kappa\).'' See Corollary \ref{Cor::Subellip::MaxSub::HeatImpliesMaxSub}.
    \item Use the parametrix from \ref{Item::MaxSub::Intro::NewMethod::Parametrix} to deduce sharp estimates, as in \ref{Item::MaxSub::Intro::SharpEstimates}. See Remark \ref{Rmk::Subellip::Proof::epsilon0NotOptimal}.
\end{enumerate}

Unlike the previous method, \ref{Item::MaxSub::Intro::NewMethod::GaussianBounds} requires essentially no new estimates beyond
what is proved in \ref{Item::MaxSub::Intro::NewMethod::Scale}. This means that one can proceed using only scaling techniques
and standard techniques for a priori estimates to apply Theorem \ref{Thm::Results::MainThm} and deduce sharp regularity results.
As a result, this simplifies previous proofs, and is more easily adapted to new situations.

    \subsection{Examples}
    In this section, we present several examples where Theorem \ref{Thm::SubMainThm::MainThm}
applies.
Before we do so, we address the following issue. In Theorem \ref{Thm::SubMainThm::MainThm}
we assume the existence of a strongly continuous semi-group 
satisfying certain properties
(e.g., the generator
agrees with \(-\opL\) on its domain).  In the examples we present, the existence
such a semi-group follows easily from the theory of sectorial operators and the Freidrichs Extension,
and we describe this first.

Let \((\opR, \Domain[\opR])\) be a densely defined operator on a Hilbert space \(\HilbertSpace\).
\begin{definition}[{See \cite[Chapter 5, Section 3.10]{KatoPerturbationTheory}}]\label{Defn::Subellip::Examples::SectorialOp}
    \((\opR, \Domain[\opR])\) is said to be sectorial with vertex \(\gamma\in \R\)
    if \(\exists \theta\in [0,\pi/2)\) such that
    \begin{equation*}
        \left\{ \Hip{f}{\opR f} : f\in \Domain[\opR], \HNorm{f}=1 \right\}
        \subseteq\left\{ \zeta\in \C : \left| \mathrm{arg}(\zeta-\gamma) \right|\leq \theta \right\}.
    \end{equation*}
\end{definition}

\begin{remark}\label{Rmk::Subellip::Examples::EquivFormForSectoral}
    Note that \((\opR, \Domain[\opR])\) is sectorial (for some vertex \(\gamma\in \R\)) if and only if \(\exists C_1,C_2\geq 0\),
    \begin{equation*}
        \left| \Imag \Hip{f}{\opR f} \right|
        \leq C_1 \Real \Hip{f}{\opR f}+ C_2 \HNorm{f}^2,\quad \forall f\in \Domain[\opR].
    \end{equation*}
\end{remark}

\begin{example}\label{Example::Subellip::Examples::Symmetric}
    Suppose \((\opR, \Domain[\opR])\) is a symmetric operator which is bounded from below (see \cite[Chapter 5, Section 3.10]{KatoPerturbationTheory}); i.e.,
    \(\opR\) is symmetric and \(\exists \gamma\in \R\) with
    \begin{equation*}
        \gamma \HNorm{f}^2\leq \Hip{f}{\opR f}, \quad \forall f\in \Domain[\opR].
    \end{equation*}
    Then, \((\opR, \Domain[\opR])\) is sectorial with vertex \(\gamma\).
    However, sectorial operators need not be symmetric.
\end{example}

\begin{proposition}[The Freidrichs Extension]\label{Prop::Subellip::Examples::FormClosure}
    Let \((\opR, \Domain[\opR])\) be a sectorial operator with vertex \(\gamma\in \R\).
    Then, there exists an unbounded operator \(\left( A,\Domain[A] \right)\) on \(\HilbertSpace\)
    such that:
    \begin{enumerate}[(i)]
        \item\label{Item::Subellip::Examples::FormClosure::DomainContainment} \(\Domain[\opR]\subseteq \Domain[A]\),
        \item\label{Item::Subellip::Examples::FormClosure::Extension} \(A\big|_{\Domain[\opR]}=-\opR\),
        \item\label{Item::Subellip::Examples::FormClosure::mSectorial} \(\left( -A,\Domain[A] \right)\) is m-sectorial with vertex \(\gamma\) (see \cite[Chapter 5, Section 3.10]{KatoPerturbationTheory}),
        \item\label{Item::Subellip::Examples::FormClosure::Bound} 
        \(\left( A,\Domain[A] \right)\) is the generator for a strongly continuous semi-group \(T(t)\)
        on \(\HilbertSpace\) satisfying
        \begin{equation}\label{Eqn::Subellip::Examples::FormClosure::Bound}
          \Norm{T(t)}[\HilbertSpace\rightarrow\HilbertSpace]\leq e^{-\gamma t},\quad \forall t\geq 0.  
        \end{equation}
        \item\label{Item::Subellip::Examples::FormClosure::LimitsInDist} \(\forall u\in \Domain[A]\) there exists 
        a sequence \(g_j\in \Domain[\opR]\) 
            such that \(g_j\xrightarrow{j\rightarrow \infty} u\) in \(\HilbertSpace\) and
            \begin{equation*}
                \Hip{f}{\opR g_j}\xrightarrow{j\rightarrow \infty} \Hip{f}{-A u},\quad \forall f\in \Domain[\opR].
            \end{equation*}
    \end{enumerate}
\end{proposition}
\begin{proof}[Comments on the proof]
    This result is standard and described in \cite[Chapter 6, Section 2.3]{KatoPerturbationTheory};
    however, to directly demonstrate all the above properties, we expand on the proof given there.
    Define the sesqui-linear form \(\FormQ[f][g]\) with domain \(\Domain[\opR]\) by
    \begin{equation*}
        \FormQ[f][g]=\Hip{f}{\opR g},\quad f,g\in \Domain[\opR].
    \end{equation*}
    By \cite[Chapter 6, Section 1.5, Theorem 1.27]{KatoPerturbationTheory},
    \((\FormQ,\Domain[\opR])\) is closeable; let \((\FormQClosure, \Domain[\FormQClosure])\) denote the closure;
    note that \(\Domain[\opR]\) is a core for \((\FormQClosure, \Domain[\FormQClosure])\).
    By \cite[Chapter 6, Section 1.4, Theorem 1.18]{KatoPerturbationTheory}, \((\FormQClosure, \Domain[\FormQClosure])\)
    is a closed sectorial form with vertex \(\gamma\).
    \cite[Chapter 6, Section 2.1, Theorem 2.1]{KatoPerturbationTheory} associates to \((\FormQClosure, \Domain[\FormQClosure])\)
    an m-sectorial operator \((-A,\Domain[A])\) satisfying \(\FormQClosure[f][g]=\Hip{f}{-A g}\), \(\forall f\in \Domain[\FormQClosure]\) and
    \(g\in \Domain[A]\) (and this operator is maximal with this property--see  \cite[Chapter 6, Section 2.1, Theorem 2.1]{KatoPerturbationTheory} for details).
    Let \(g\in \Domain[\opR]\). Then
    \begin{equation}\label{Eqn::Subellip::Examples::FormClosure::Tmp1}
        \FormQClosure[f][g]=\FormQ[f][g]=\Hip{f}{\opR g},\quad \forall f\in \Domain[\opR].
    \end{equation}
    Since \(\Domain[\opR]\) is a core for \((\FormQClosure, \Domain[\FormQClosure])\),
    \eqref{Eqn::Subellip::Examples::FormClosure::Tmp1} and \cite[Chapter 6, Section 2.1, Theorem 2.1(iii)]{KatoPerturbationTheory} 
    shows \(g\in \Domain[A]\)
    with \(\opR g=-A g\). This completes the proof of \ref{Item::Subellip::Examples::FormClosure::DomainContainment},
     \ref{Item::Subellip::Examples::FormClosure::Extension}, and \ref{Item::Subellip::Examples::FormClosure::mSectorial}.
    For \ref{Item::Subellip::Examples::FormClosure::LimitsInDist}, 
    let \(u\in \Domain[A]\).
    Since
    \(\Domain[A]\subseteq \Domain[\FormQClosure]\) and \((\FormQClosure, \Domain[\FormQClosure])\)
    is the closure of \((\FormQ,\Domain[\opR])\), there exist \(g_j\in \Domain[\opR]\)
    with \(g_j\rightarrow u\) in \(\HilbertSpace\)
    and 
    \begin{equation*}
        \Hip{f}{\opR g_j}=\FormQ[f][g_j]\rightarrow \FormQClosure[f][u]=\Hip{f}{-A u},\quad \forall f\in \Domain[\opR],
    \end{equation*}
    establishing \ref{Item::Subellip::Examples::FormClosure::LimitsInDist}.

    \ref{Item::Subellip::Examples::FormClosure::Bound} follows from 
    \ref{Item::Subellip::Examples::FormClosure::mSectorial} and
    the fact that if 
    \(\left( -A,\Domain[A] \right)\) is an m-sectorial operator with vertex \(\gamma\in \R\),
    then \(\left( A,\Domain[A] \right)\) is the generator of a strongly continuous semi-group
    satisfying \eqref{Eqn::Subellip::Examples::FormClosure::Bound}.
    In fact, by definition (see \cite[Chapter 5, Section 3.10]{KatoPerturbationTheory})
    we have \(\left( -A-\gamma,\Domain[A]  \right)\) is m-accretive.
    From here, \ref{Item::Subellip::Examples::FormClosure::Bound} follows from the Lumer--Phillips Theorem
    \cite[Chapter 2, Theorem 3.15]{EngelNagelAShortCourseOnOperatorSemigroups}.

    In fact, \cite[Chapter 9, Section 1, Theorem 1.24]{KatoPerturbationTheory} shows \(T(t)\)
    is a holomorphic semi-group, though we do not use this fact.

\end{proof}

\begin{remark}
    In the case when \((\opR,\Domain[\opR])\) is symmetric and bounded below (see Example \ref{Example::Subellip::Examples::Symmetric}),
    the proof in Proposition \ref{Prop::Subellip::Examples::FormClosure} yields a self-adjoint extension
    of \((\opR,\Domain[\opR])\) with the same lower bound. This may be the most familiar version of the Freidrichs Extension.
\end{remark}

Now suppose \(\WWd=\left\{ \left( W_1,\Wd_1 \right),\ldots, \left( W_r,\Wd_r \right) \right\}\)
are H\"ormander vector fields with formal degrees on the connected, smooth manifold \(\Manifold\) and \(\mu\)
is a smooth, strictly positive density on \(\Manifold\).
Fix \(\kappa\in \Nplus\) such that \(\Wd_j\) divides \(\kappa\), \(\forall j\).
We consider partial differential operators of the form
\begin{equation}\label{Eqn::SubEllip::Examples::opLFormula}
    \opL=\sum_{\degWd(\alpha),\degWd(\beta)\leq \kappa} \left( W^{\alpha} \right)^{*}a_{\alpha,\beta}(x)W^{\beta},
    \quad a_{\alpha,\beta}\in \CinftySpace[\Manifold],
\end{equation}
where \(*\)-denote the formal \(\LpSpace{2}[\Manifold][\mu]\) adjoint.

For notational convenience we introduce a condition on \(\opL\).
\begin{condition}\label{Cond::SubEllip::Examples::ConditionM}
    We say \(\opL\) given by \eqref{Eqn::SubEllip::Examples::opLFormula} satisfies
    Condition \ref{Cond::SubEllip::Examples::ConditionM} with vertex \(\gamma\in \R\) if:
    \begin{itemize}
        \item Assumption \ref{Assump::SubMainTheorem::MaxSubTypeIII} holds, and
        \item \((\opL, \CzinftySpace[\Manifold])\) is sectorial with vertex \(\gamma\).
    \end{itemize}
    We say \(\opL\) satisfies Condition \ref{Cond::SubEllip::Examples::ConditionM} if \(\exists \gamma\in \R\) such that
    \(\opL\) satisfies Condition \ref{Cond::SubEllip::Examples::ConditionM} with vertex \(\gamma\).
\end{condition}

\begin{proposition}\label{Prop::SubEllip::Examples::ConditionMConclusions}
    Suppose \(\opL\) given by \eqref{Eqn::SubEllip::Examples::opLFormula} satisfies Condition \ref{Cond::SubEllip::Examples::ConditionM} with vertex \(\gamma\in \R\).
    Then,
    \begin{enumerate}[(i)]
        \item\label{Item::SubEllip::Examples::ConditionMConclusions::MaxSub} \(\opL\) is maximally subelliptic of degree \(2\kappa\) with respect to \(\WWd\),
        \item\label{Item::SubEllip::Examples::ConditionMConclusions::ExistsSemiGroup} There exists a strongly continuous semi-group \(T(t)\) on \(\LpSpace{2}[\Manifold][\mu]\), satisfying \(\|T(t)\|\leq e^{-\gamma t}\),
    with generator \(\left( A,\Domain[A] \right)\) such that
    \(\CzinftySpace[\Manifold]\subseteq \Domain[A]\) and \(-A=\opL\big|_{\Domain[A]}\), 
    where \(\opL\)
    is taken in the sense of distributions.
        \item\label{Item::SubEllip::Examples::ConditionMConclusions::GaussianBounds} 
        For any strongly continuous semi-group \(T(t)\) satisfying the conclusions of \ref{Item::SubEllip::Examples::ConditionMConclusions::ExistsSemiGroup},
     \(T(t)\) satisfies the conclusions of Theorem \ref{Thm::SubMainThm::MainThm} with \(\omega_0=-\gamma\).
    \end{enumerate}

\end{proposition}
\begin{proof}
    \ref{Item::SubEllip::Examples::ConditionMConclusions::ExistsSemiGroup}: Since Condition \ref{Cond::SubEllip::Examples::ConditionM}
    assumes \((\opL, \CzinftySpace[\Manifold])\) is  sectorial with vertex \(\gamma\),
    Proposition \ref{Prop::Subellip::Examples::FormClosure}
    (applied with \((\opR,\Domain[\opR])=(\opL,\CzinftySpace[\Manifold])\))
    gives the existence of the semi-group
    \(T(t)\), satisfying \(\|T(t)\|\leq e^{-\gamma t}\), with generator \(\left( A,\Domain[A] \right)\), 
    \(\CzinftySpace[\Manifold]\subseteq \Domain[A]\) and
     \(A\big|_{\CzinftySpace[\Manifold]}=-\opL\).
    
     To see \(-A=\opL\big|_{\Domain[A]}\), where \(\opL\)
    is taken in the sense of distributions take \(u\in \Domain[A]\).
    Proposition \ref{Prop::Subellip::Examples::FormClosure}
    \ref{Item::Subellip::Examples::FormClosure::LimitsInDist}
    shows that there exist \(g_j\in \CzinftySpace[\Manifold]\) with \(g_j\rightarrow u\) in
    \(\LpSpace{2}[\Manifold][\mu]\) and
    \begin{equation*}
        \Ltip{f}{\opL g_j}\rightarrow \Ltip{f}{-A u},\quad \forall f\in \CzinftySpace[\Manifold].
    \end{equation*}
    However, we also have \(\opL g_j\rightarrow \opL u\) in distribution, and therefore
    \begin{equation*}
        \Ltip{f}{\opL g_j}\rightarrow \Ltip{f}{\opL u},\quad \forall f\in \CzinftySpace[\Manifold].
    \end{equation*}
    We conclude \(\opL u=-A u\).


    Since Condition \ref{Cond::SubEllip::Examples::ConditionM} includes that Assumption \ref{Assump::SubMainTheorem::MaxSubTypeIII} holds,
    \ref{Item::SubEllip::Examples::ConditionMConclusions::GaussianBounds} follows immediately from Theorem \ref{Thm::SubMainThm::MainThm}.

    \ref{Item::SubEllip::Examples::ConditionMConclusions::MaxSub} follows from
    Corollary \ref{Cor::Subellip::MaxSub::HeatImpliesMaxSub}, whose conditions
    are verified by \ref{Item::SubEllip::Examples::ConditionMConclusions::ExistsSemiGroup}
    and Condition \ref{Cond::SubEllip::Examples::ConditionM}.
\end{proof}

In light of Proposition \ref{Prop::SubEllip::Examples::ConditionMConclusions} to give examples
where Theorem \ref{Thm::SubMainThm::MainThm} applies, it suffices to give examples where
Condition \ref{Cond::SubEllip::Examples::ConditionM} holds.  The next proposition gives such examples.

\begin{proposition}\label{Prop::Subellip::Examples::Examples}
    The following examples satisfy Condition \ref{Cond::SubEllip::Examples::ConditionM},
    and therefore (by Proposition \ref{Prop::SubEllip::Examples::ConditionMConclusions})
    are maximally subelliptic and are associated to semi-groups satisfying higher order Gaussian bounds
    as in Theorem \ref{Thm::SubMainThm::MainThm}.
    Throughout, \(\Manifold\) is a connected, smooth manifold endowed with a smooth, strictly positive
    density \(\mu\). \(\WWd=\left\{\left( W_1,\Wd_1 \right),\ldots, \left( W_r, \Wd_r \right)  \right\}\)
    are H\"ormander vector fields with formal degrees on \(\Manifold\); either chosen in the example
    or assumed given. All adjoints are formal \(\LpSpace{2}[\Manifold][\mu]\) adjoints.
    \begin{enumerate}[(a)]
        \item\label{Item::Subellip::Examples::Examples::HormanderSubLap} Let \(W_1,\ldots, W_r\) be H\"ormander vector fields on \(\Manifold\), and let \(n_1,\ldots, n_r\in \Nplus\)
            be given. Set
            \begin{equation}\label{Eqn::Subellip::Examples::Examples::HormanderSubLap}
                \opL:=\sum_{j=1}^r \left( W_1^{n_1} \right)^{*} W_1^{n_1}+\cdots + \left( W_r^{n_r} \right)^{*} W_r^{n_r}.
            \end{equation}
            Let \(\kappa\in \Nplus\) be the least common multiple of \(n_1,\ldots,n_r\) and set
            \(\Wd_j=\kappa/n_j\in \Nplus\).  Then, \(\opL\) satisfies Condition \ref{Cond::SubEllip::Examples::ConditionM} with
            vertex \(0\) with
            this choice of \(\WWd\).
        \item\label{Item::Subellip::Examples::Examples::MaxSub} Let \(\opP\) be as in \eqref{Eqn::SubEllip::MaxSub::FormulaForP} and assume
            \(\opP\) is maximally subelliptic of degree \(\kappa\) with respect to \(\WWd\)
            as in Definition \ref{Defn::SubEllip::MaxSub::MaxSub}. Then, \(\opL=\opP^{*}\opP\)
            satisfies Condition \ref{Cond::SubEllip::Examples::ConditionM} with vertex \(0\).

        \item\label{Item::Subellip::Examples::Examples::GeneralSubLap}
            More generally than \ref{Item::Subellip::Examples::Examples::HormanderSubLap},
        let \(W_1,\ldots, W_r\) be H\"ormander vector fields on \(\Manifold\), and let \(n_1,\ldots, n_r\in \Nplus\)
            be given. For each \(1\leq i,j\leq r\), let \(a_i^j\in \CinftySpace[\Manifold]\).
            We suppose
            \begin{enumerate}[label=(\alph{enumi}.\roman*)]
                \item\label{Item::Subellip::Examples::Examples::GeneralSubLap::MaxSub} \(\forall \Compact\Subset \Manifold\) compact, \(\exists A=A(\Compact)\geq 0\),
                    \(\forall \xi=(\xi_1,\ldots, \xi_r)\in \C^r\), \(\forall x\in \Compact\),
                    \begin{equation*}
                        |\xi|^2 \leq A \Real \sum_{i,j} \xi_i \overline{a_i^j(x) \xi_j}.
                    \end{equation*}
                \item\label{Item::Subellip::Examples::Examples::GeneralSubLap::Sectorial} \(\exists B\geq 0\), \(\forall \xi=(\xi_1,\ldots, \xi_r)\in \C^r\), \(\forall x\in \Manifold\),
                    \begin{equation*}
                        \left| \Imag \xi_i \overline{a_i^j(x) \xi_j} \right|
                        \leq B \Real \xi_i \overline{a_i^j(x) \xi_j}.
                    \end{equation*}
            \end{enumerate}
            Let \(\kappa\in \Nplus\) be the least common multiple of \(n_1,\ldots,n_r\) and set
            \(\Wd_j=\kappa/n_j\in \Nplus\).
            Then, \(\opL=\sum_{i,j=1}^r \left( W_i^{n_i} \right)^{*} a_i^j W_j^{n_j}\)
            satisfies Condition \ref{Cond::SubEllip::Examples::ConditionM} with this choice of \(\WWd\).

        \item\label{Item::Subellip::Examples::Examples::Perturb} Suppose \(\opL\) is as in \eqref{Eqn::SubEllip::Examples::opLFormula} and satisfies Condition \ref{Cond::SubEllip::Examples::ConditionM}.
            Let \(\opE\) be a partial differential operator of the form
            \begin{equation*}
                \opE=\sum_{\degWd(\alpha),\degWd(\beta)\leq \kappa}  \left( W^{\alpha} \right)^{*} b_{\alpha,\beta} W^{\beta},\quad b_{\alpha,\beta}\in \CinftySpace[\Manifold].
            \end{equation*}
            Suppose \(\exists a\in [0,1)\), \(c\in [0,\infty)\) with
            \begin{equation}\label{Eqn::Subellip::Examples::Examples::Perturb::AssumedBound}
                \left| \Ltip{f}{\opE f}[\Manifold][\mu] \right|\leq a \Real \Ltip{f}{\opL f}[\Manifold][\mu] + c\LpNorm{f}{2}[\Manifold][\mu]^2,\quad \forall f\in \CzinftySpace[\Manifold].
            \end{equation}
            Then, \(\opL+\opE\) satisfies Condition \ref{Cond::SubEllip::Examples::ConditionM}.
        \item\label{Item::Subellip::Examples::Examples::PerturbMaxtrix} Let \(W_1,\ldots, W_r\) be H\"ormander vector fields on \(\Manifold\), and let \(n_1,\ldots, n_r\in \Nplus\)
            be given.  Let \(\kappa,\Wd_j\in \Nplus\) be as in \ref{Item::Subellip::Examples::Examples::HormanderSubLap}
            and define \(\opL\) by \eqref{Eqn::Subellip::Examples::Examples::HormanderSubLap}.
            Let \(B(x)=(b_{i,j}(x))\in \CinftySpace[\Manifold][\M^{r\times r}(\C)]\) and suppose \(\exists a\in [0,1)\)
            with
            \begin{equation}\label{Eqn::Subellip::Examples::Examples::PerturbMaxtrix::BBound}
                \left| \ip{\xi}{B(x)\xi} \right| \leq a |\xi|^2, \quad \forall x\in \Manifold, \xi\in \C^r.
            \end{equation}
            Then, \(\opL + \sum_{j=1}^r \left( W_i^{n_i} \right)^{*} b_{i,j} W_j^{n_j}\)
            satisfies Condition \ref{Cond::SubEllip::Examples::ConditionM} with \(\WWd\) as in \ref{Item::Subellip::Examples::Examples::HormanderSubLap}.
        \item\label{Item::Subellip::Examples::Examples::PerturbScalar} Let \(r=2\) and \(W_1,W_2\) H\"ormander vector fields on \(\Manifold\).
            Let \(n_1,n_2\in \Nplus\). For \(\alpha\in \C\), set
            \begin{equation*}
                \opL_{\alpha}:= \left( W_1^{n_1} \right)^{*} W_1^{n_1} + \left( W_2^{n_2} \right)^{*} W_2^{n_2} + \alpha \left( W_1^{n_1} \right)^{*} W_2^{n_2}.
            \end{equation*}
             Let \(\kappa\in \Nplus\) be the least common multiple of \(n_1,n_2\) and set
            \(\Wd_j=\kappa/n_j\in \Nplus\). Then, if \(|\alpha|<2\), \(\opL_\alpha\) satisfies Condition \ref{Cond::SubEllip::Examples::ConditionM}
            with \(\WWd=\left\{ \left( W_1,\Wd_1 \right),\left( W_2,\Wd_2 \right) \right\}\).

        \item\label{Item::Subellip::Examples::Examples::PerturbLowerOrder} Suppose \(\Manifold\) is compact and  \(\opL\) is as in \eqref{Eqn::SubEllip::Examples::opLFormula} and satisfies Condition \ref{Cond::SubEllip::Examples::ConditionM}.
             Let \(\opE\) be a partial differential operator of the form
            \begin{equation*}
                \opE=\sum_{\substack{\degWd(\alpha),\degWd(\beta)\leq \kappa \\ \degWd(\alpha)+\degWd(\beta)<2\kappa}}  \left( W^{\alpha} \right)^{*} b_{\alpha,\beta} W^{\beta},\quad b_{\alpha,\beta}\in \CinftySpace[\Manifold].
            \end{equation*}
            Then, \(\opL+\opE\) satisfies Condition \ref{Cond::SubEllip::Examples::ConditionM}.
    \end{enumerate}
\end{proposition}
\begin{proof}
    \ref{Item::Subellip::Examples::Examples::MaxSub}: With \(\opL=\opP^{*}\opP\), Assumption \ref{Assump::SubMainTheorem::MaxSubTypeIII}
    is just a restatement of Definition \ref{Defn::SubEllip::MaxSub::MaxSub}.
    \(\left( \opP^{*}\opP, \CzinftySpace[\Manifold] \right)\) is clearly a non-negative symmetric operator
    and therefore is sectorial with vertex \(0\) (see Example \ref{Example::Subellip::Examples::Symmetric}).

    \ref{Item::Subellip::Examples::Examples::HormanderSubLap}: 
    This is the special case of \ref{Item::Subellip::Examples::Examples::MaxSub} with 
    \(N=r\) in Definition \ref{Defn::SubEllip::MaxSub::MaxSub},
    and \(\opP f=\left(W_1^{n_1}f, W_2^{n_2}f,\ldots, W_r^{n_r} f  \right)\).  It is immediate from the definitions
    that \(\opP\) is maximally subelliptic of degree \(\kappa\) with respect to \(\WWd\).

    \ref{Item::Subellip::Examples::Examples::GeneralSubLap}:
    Using \ref{Item::Subellip::Examples::Examples::GeneralSubLap::Sectorial}, we have \(\forall u\in \CzinftySpace[\Manifold]\),
    \begin{equation*}
        \begin{split}
            &\left| \Imag \Ltip{u}{\opL u}[\Manifold][\mu] \right|
            =\left| \sum_{i,j} \Imag \Ltip{W_i^{n_i} u}{a_i^j W_j^{n_j} u}[\Manifold][\mu] \right|
            \\&\leq B  \sum_{i,j} \Real \Ltip{W_i^{n_i} u}{a_i^j W_j^{n_j} u}[\Manifold][\mu] 
            =B \Real \Ltip{u}{\opL u}[\Manifold][\mu].
        \end{split}
    \end{equation*}
    By Remark \ref{Rmk::Subellip::Examples::EquivFormForSectoral}, this shows \(\left( \opL, \CzinftySpace[\Manifold] \right)\)
    is sectoral.
    Fix \(\Omega\Subset \Manifold\) open and relatively compact.
    Let \(A\) be as in \ref{Item::Subellip::Examples::Examples::GeneralSubLap::MaxSub} with
    \(\Compact=\overline{\Omega}\).  We have, \(\forall u\in \CzinftySpace[\Omega]\),
    \begin{equation*}
    \begin{split}
         & \sum_{j=1}^r \BLpNorm{W_i^{n_i} u}{2}[\Manifold][\mu]^2
         \leq A \sum_{i,j} \Real \Ltip{W_i^{n_i} u}{a_i^j W_j^{n_j} u}[\Manifold][\mu] 
         =A \Real \Ltip{u}{\opL u}[\Manifold][\mu].
    \end{split}
    \end{equation*}
    This establishes Assumption \ref{Assump::SubMainTheorem::MaxSubTypeIII} and completes the proof.

    \ref{Item::Subellip::Examples::Examples::Perturb}: 
    That \((\opL+\opE, \CzinftySpace[\Manifold])\) is sectorial follows directly from \cite[Chapter 6, Section 1.6, Theorem 1.33]{KatoPerturbationTheory};
    however, we include the proof here as the same initial estimate \eqref{Eqn::Subellip::Examples::Examples::Perturb::Tmp1} is also used to establish Assumption \ref{Assump::SubMainTheorem::MaxSubTypeIII}
    for \(\opL+\opE\).

    \eqref{Eqn::Subellip::Examples::Examples::Perturb::AssumedBound}
    gives, for \(f\in \CzinftySpace[\Manifold]\),
    \begin{equation*}
    \begin{split}
         &\Real \Ltip*{f}{\left(\opL+\opE\right)f}\geq \Real \Ltip*{f}{\opL f} -\left| \Ltip*{f}{\opE g} \right|
         \geq (1-a) \Real \Ltip*{f}{\opL f} - c\LpNorm{f}{2}^2.
    \end{split}
    \end{equation*}
    Therefore,
    \begin{equation}\label{Eqn::Subellip::Examples::Examples::Perturb::Tmp1}
        \Real \Ltip*{f}{\opL f}\leq (1-a)^{-1}\Real \Ltip*{f}{\left(\opL+\opE\right)f} + c(1-a)^{-1}\LpNorm{f}{2}^2.
    \end{equation}
    The assumption \((\opL,\CzinftySpace[\Manifold])\) is sectorial is equivalent to 
    \begin{equation}\label{Eqn::Subellip::Examples::Examples::Perturb::Tmp2}
        \left| \Imag \Ltip{f}{\opL f} \right|
        \leq C_1 \Real \Ltip{f}{\opL f}+ C_2 \LpNorm{f}{2}^2,\quad \forall f\in \CzinftySpace[\Manifold],
    \end{equation}
    for some \(C_1,C_2\geq 0\);
    see Remark \ref{Rmk::Subellip::Examples::EquivFormForSectoral}.
    Combining \eqref{Eqn::Subellip::Examples::Examples::Perturb::Tmp1} and \eqref{Eqn::Subellip::Examples::Examples::Perturb::Tmp2},
    we see
    \begin{equation}\label{Eqn::Subellip::Examples::Examples::Perturb::Tmp3}
        \left| \Imag \Ltip{f}{\opL f} \right|
        \leq C_1' \Real \Ltip{f}{(\opL+\opE) f}+ C_2' \LpNorm{f}{2}^2,\quad \forall f\in \CzinftySpace[\Manifold].
    \end{equation}
    Using \eqref{Eqn::Subellip::Examples::Examples::Perturb::Tmp3} and \eqref{Eqn::Subellip::Examples::Examples::Perturb::AssumedBound},
    we see, \(\forall f\in \CzinftySpace[\Manifold]\),
    \begin{equation*}
        \left| \Imag \Ltip{f}{(\opL+\opE) f} \right|
        \leq \left| \Imag \Ltip{f}{\opL f} \right| + \left| \Ltip{f}{\opE f} \right|
        \leq C_1'' \Real \Ltip{f}{(\opL+\opE) f}+ C_2'' \LpNorm{f}{2}^2.
    \end{equation*}
    This shows that \((\opL+\opE,\CzinftySpace[\Manifold])\) is sectoral (see Remark \ref{Rmk::Subellip::Examples::EquivFormForSectoral}).

    Next, we verify that \(\opL+\opE\) satisfies Assumption \ref{Assump::SubMainTheorem::MaxSubTypeIII}.
    Fix \(\Omega\Subset \Manifold\) open and relatively compact. Since \(\opL\) satisfies Assumption \ref{Assump::SubMainTheorem::MaxSubTypeIII},
    there exist \(C_1, C_2\geq 0\), such that \(\forall f\in \CzinftySpace[\Omega]\),
    \begin{equation}\label{Eqn::Subellip::Examples::Examples::Perturb::Tmp4}
        \sum_{j=1}^r \BLpNorm{ W_j^{\kappa/\Wd_j} f}{2}[\Manifold][\measure]^2
        \leq C_1 \Real \Ltip*{f}{\opL f}[\Manifold][\measure]
        +C_2 \BLpNorm{f}{2}[\Manifold][\measure]^2.
    \end{equation}
    Combining \eqref{Eqn::Subellip::Examples::Examples::Perturb::Tmp1}
    and \eqref{Eqn::Subellip::Examples::Examples::Perturb::Tmp4} shows
    \begin{equation*}
        \sum_{j=1}^r \BLpNorm{ W_j^{\kappa/\Wd_j} f}{2}[\Manifold][\measure]^2
        \leq C_1' \Real \Ltip*{f}{(\opL+\opE) f}[\Manifold][\measure]
        +C_2' \BLpNorm{f}{2}[\Manifold][\measure]^2, \quad \forall f\in \CzinftySpace[\Omega],
    \end{equation*}
    establishing  Assumption \ref{Assump::SubMainTheorem::MaxSubTypeIII} for \(\opL+\opE\)
    and completing the proof of \ref{Item::Subellip::Examples::Examples::Perturb}.

    \ref{Item::Subellip::Examples::Examples::PerturbMaxtrix}: Let \(\opE:=\sum_{j=1}^r \left( W_i^{n_i} \right)^{*} b_{i,j} W_j^{n_j}\).
    Then, \eqref{Eqn::Subellip::Examples::Examples::PerturbMaxtrix::BBound} implies, for \(f\in \CzinftySpace[\Manifold]\),
    \begin{equation*}
        \left| \Ltip{f}{\opE f} \right|
        \leq a\sum_{j=1}^r \BLpNorm{W_j^{n_j}f}{2}^2
        =a \Real \Ltip{f}{\opL f}.
    \end{equation*}
    From here, \ref{Item::Subellip::Examples::Examples::PerturbMaxtrix} follows from 
    \ref{Item::Subellip::Examples::Examples::HormanderSubLap} and 
    \ref{Item::Subellip::Examples::Examples::Perturb}.

    \ref{Item::Subellip::Examples::Examples::PerturbScalar}:
    This is the special case of \ref{Item::Subellip::Examples::Examples::PerturbMaxtrix} with \(r=2\),
    \(b_{1,2}=\alpha\), \(b_{i,j}=0\) for \((i,j)\ne (1,2)\).

    \ref{Item::Subellip::Examples::Examples::PerturbLowerOrder}:
    To prove this result, we require some new notation and some results from \cite{StreetMaximalSubellipticity}. We write \(A\leq \sconst B_1+\lconst B_2\) to mean
    \(\forall \epsilon>0\), \(\exists C_\epsilon\geq 0\), \(A\leq \epsilon B_1+C_{\epsilon}B_2\);
    here \(\sconst\) and \(\lconst\) stand for ``small constant'' and ``large constant,'' respectively.
    In \cite[Section 6.2]{StreetMaximalSubellipticity}, Triebel--Lizorkin spaces with respect to \(\WWd\)
    were defined (here we take \(\mathcal{K}=\Manifold\) and \(\psi=1\) in that reference, using the fact
    that \(\Manifold\) is compact).  In particular, operators \(D_j\), \(j\in \N\), were given and the corresponding Triebel--Lizorkin norm
    was defined by
    \begin{equation}\label{Eqn::Subellip::Examples::Examples::PerturbLowerOrder::Tmp0}
        \TLNorm{f}{s}:=\LtltNorm{ \left\{ 2^{js}D_j f \right\}_{j\in \N} }.
    \end{equation}
    \cite[Corollary 6.2.14 (6.5)]{StreetMaximalSubellipticity} shows
    \begin{equation}\label{Eqn::Subellip::Examples::Examples::PerturbLowerOrder::EquivNorms}
        \TLNorm{f}{\kappa}^2\approx \sum_{\deg(\alpha)\leq \kappa} \BLpNorm{W^{\alpha}f}{2}^2\approx \sum_{j=1}^r \BLpNorm{ W_j^{\kappa/\Wd_j} f}{2}[\Manifold][\measure]^2 +\BLpNorm{f}{2}[\Manifold][\measure]^2.
    \end{equation}
    \cite[Propositions 6.2.11 and 6.2.13]{StreetMaximalSubellipticity} show
    \begin{equation}\label{Eqn::Subellip::Examples::Examples::PerturbLowerOrder::Tmp1}
        \sum_{\deg(\alpha)\leq \kappa-1} \BLpNorm{W^{\alpha}f}{2}
        \lesssim \TLNorm{f}{\kappa-1}.
    \end{equation}
    Using \eqref{Eqn::Subellip::Examples::Examples::PerturbLowerOrder::Tmp0} and \cite[Proposition 6.2.13]{StreetMaximalSubellipticity}, we have
    \begin{equation}\label{Eqn::Subellip::Examples::Examples::PerturbLowerOrder::Tmp2}
        \begin{split}
            &\TLNorm{f}{\kappa-1} =\LtltNorm{ \left\{ 2^{j(\kappa-1)}D_j f \right\}_{j\in \N} }
            \\&\leq \sconst \LtltNorm{ \left\{ 2^{j\kappa}D_j f \right\}_{j\in \N} }
            +\lconst \LtltNorm{ \left\{ D_j f \right\}_{j\in \N} }
            \\&=\sconst \TLNorm{f}{\kappa} + \lconst\TLNorm{f}{0}^2
            \\&=\sconst \TLNorm{f}{\kappa} + \lconst\LpNorm{f}{2}^2
        \end{split}
    \end{equation}
    Combining \eqref{Eqn::Subellip::Examples::Examples::PerturbLowerOrder::Tmp1}, \eqref{Eqn::Subellip::Examples::Examples::PerturbLowerOrder::Tmp2},
    and \eqref{Eqn::Subellip::Examples::Examples::PerturbLowerOrder::EquivNorms}, we see
    \begin{equation}\label{Eqn::Subellip::Examples::Examples::PerturbLowerOrder::Tmp3}
    \begin{split}
         &\sum_{\deg(\alpha)\leq \kappa-1} \BLpNorm{W^{\alpha}f}{2}
         \leq \sconst \sum_{j=1}^r \BLpNorm{ W_j^{\kappa/\Wd_j} f}{2}^2 +
         \lconst \BLpNorm{f}{2}^2.
    \end{split}
    \end{equation}

    Since \(\Manifold\) is compact, Assumption \ref{Assump::SubMainTheorem::MaxSubTypeIII}
    for \(\opL\) (with \(\Omega=\Manifold\)) implies
    \(\exists C_1,C_2\geq 0\),
    \begin{equation}\label{Eqn::Subellip::Examples::Examples::PerturbLowerOrder::Tmp6}
        \sum_{j=1}^r \BLpNorm{ W_j^{\kappa/\Wd_j} f}{2}^2
        \leq C_1 \Real \Ltip*{f}{\opL f}
        +C_2 \BLpNorm{f}{2}^2, \quad \forall f\in \CinftySpace[\Manifold].
    \end{equation}
    Now, consider, for \(f\in \CinftySpace[\Manifold]\),
    \begin{equation}\label{Eqn::Subellip::Examples::Examples::PerturbLowerOrder::Tmp4}
    \begin{split}
         &\left| \Ltip*{f}{\opE f}\right|
         \leq \sum_{\substack{\degWd(\alpha),\degWd(\beta)\leq \kappa \\ \degWd(\alpha)+\degWd(\beta)<2\kappa}}  \left| \Ltip*{W^{\alpha} f}{b_{\alpha,\beta} W^{\beta}f} \right|
         \lesssim \sum_{\substack{\degWd(\alpha)\leq \kappa-1\\\degWd(\beta)\leq\kappa}} \BLpNorm{W^{\alpha}f}{2}\BLpNorm{W^{\beta}f}{2}
         \\&\leq \sconst \sum_{\degWd(\alpha)\leq \kappa} \BLpNorm{W^{\alpha}f}{2}^2 + \lconst \sum_{\degWd(\alpha)\leq \kappa-1} \BLpNorm{W^{\alpha}f}{2}^2
         \\&\leq \sconst \sum_{\degWd(\alpha)\leq \kappa} \BLpNorm{W^{\alpha}f}{2}^2 + \sconst \sum_{j=1}^r \BLpNorm{ W_j^{\kappa/\Wd_j} f}{2}^2  + \lconst \BLpNorm{f}{2}^2,
    \end{split}
    \end{equation}
    where the final estimate used \eqref{Eqn::Subellip::Examples::Examples::PerturbLowerOrder::Tmp3}.
    Applying \eqref{Eqn::Subellip::Examples::Examples::PerturbLowerOrder::EquivNorms}
    to the first term on the right-hand side of \eqref{Eqn::Subellip::Examples::Examples::PerturbLowerOrder::Tmp4}
    shows
    \begin{equation}\label{Eqn::Subellip::Examples::Examples::PerturbLowerOrder::Tmp5}
    \begin{split}
         &\left| \Ltip*{f}{\opE f}\right|
         \leq \sconst \sum_{j=1}^r \BLpNorm{ W_j^{\kappa/\Wd_j} f}{2}^2  + \lconst \BLpNorm{f}{2}^2.
    \end{split}
    \end{equation}
    Combining \eqref{Eqn::Subellip::Examples::Examples::PerturbLowerOrder::Tmp5} and
    \eqref{Eqn::Subellip::Examples::Examples::PerturbLowerOrder::Tmp6} shows
    \begin{equation}\label{Eqn::Subellip::Examples::Examples::PerturbLowerOrder::Tmp7}
        \left| \Ltip*{f}{\opE f}\right|
        \leq \sconst \Real \Ltip*{f}{\opL f} + \lconst \BLpNorm{f}{2}^2, \quad \forall f\in \CinftySpace[\Manifold].
    \end{equation}
    Using \eqref{Eqn::Subellip::Examples::Examples::PerturbLowerOrder::Tmp7}, \ref{Item::Subellip::Examples::Examples::PerturbLowerOrder} follows from
    \ref{Item::Subellip::Examples::Examples::Perturb}.
\end{proof}

\begin{remark}\label{Rmk::Subellip::Examples::ExamplesAreMaxSub}
    Proposition \ref{Prop::Subellip::Examples::Examples} shows that the examples given there
    are maximally subelliptic. This is often not obvious directly from the definitions.
    For example, consider H\"ormander's sub-Laplacian \(\opL=W_1^{*}W_1+\cdots+W_r^{*}W_r\),
    where \(W_1,\ldots, W_r\) are vector fields satisfying H\"ormander's condition.
    It is a result of Rothschild and Stein \cite{RothschildSteinHypoellipticDifferentialOperatorsAndNilpotentGroups}
    that \(\opL\) is maximally subelliptic of degree \(2\) with respect to
    \(\left\{ \left( W_1,1 \right),\ldots, \left( W_r,1 \right) \right\}\);
    however, this does not follow directly from the definitions.
    Proposition \ref{Prop::Subellip::Examples::Examples} gives another proof of this fact
    which generalizes to many other situations (for example, all the operators described in that proposition).
\end{remark}

\begin{remark}\label{Rmk::Subellip::Examples::LeftParametrixForP}
    In light of Corollary \ref{Cor::Subellip::MaxSub::HeatImpliesMaxSub},
    the operators in Proposition \ref{Prop::Subellip::Examples::Examples} all have
    nice two-sided parametrices. This is particularly useful in \ref{Item::Subellip::Examples::Examples::MaxSub}
    where if \(S\) is the two-sided parametrix for \(\opP^{*}\opP\), then
    \(S\opP^{*}\) is a left parametrix for \(\opP\).
    This is how left parametrices for general maximally subelliptic operators were
    constructed in \cite[Theorem 8.1.1]{StreetMaximalSubellipticity}.
\end{remark}

    \subsection{Proof of Theorem \texorpdfstring{\ref{Thm::SubMainThm::MainThm}}{\ref*{Thm::SubMainThm::MainThm}}}\label{Section::Subellip::Proof}
    In this section, we present the proof of Theorem \ref{Thm::SubMainThm::MainThm}.
This is separated into three main parts:
\begin{enumerate}[(I)]
    \item\label{Item::SubProof::Intro::APriori} A priori subelliptic estimates at the unit scale, established using techniques of Kohn \cite{KohnLecturesOnDegenerateEllipticProblems}. 
    See Section \ref{Section::Subellip::Proof::UnitScale}.
    \item\label{Item::SubProof::Intro::Scaling} A scaling result based on the work of Nagel, Stein, and Wainger \cite{NagelSteinWaingerBallsAndMetricsDefinedByVectorFieldsIBasicProperties}. 
    See Section \ref{Section::MaxSub::Scaling}.
    \item\label{Item::SubProof::Intro::Conclusion} Putting \ref{Item::SubProof::Intro::APriori}
    and \ref{Item::SubProof::Intro::Scaling}
    together to conclude hypoelliptic estimates at every scale, uniformly
        in the scale. This shows that Theorem \ref{Thm::Results::MainThm} applies to complete the proof. See Section \ref{Section::MaxSub::CompleteProof}.
\end{enumerate}
\ref{Item::SubProof::Intro::APriori} and \ref{Item::SubProof::Intro::Scaling} are both well-known,
and we just describe the main results and refer the reader to other sources for the proofs.
\ref{Item::SubProof::Intro::Conclusion} is the part of the proof of Theorem \ref{Thm::SubMainThm::MainThm}
which is new.

        \subsubsection{A priori subelliptic estimates at the unit scale}\label{Section::Subellip::Proof::UnitScale}
        In this section, we present a priori subelliptic estimates at the unit scale.
In the proof of Proposition \ref{Prop::SubProof::CompleteProof::ScaledEstimates}, below, we apply the results of this section an infinite number of times, and it is important
that the estimates here are uniform over this infinite number of applications.
Because of this, we are explicit about which quantities the constant in the main estimate \eqref{Eqn::Subellip::Proof::Apriori::Est} depends on.

We work on the unit ball \(B^n(1)\subseteq \R^n\).  Let \(\CbinftySpace[B^n(1)]\)
denote the Fr\'echet space of those smooth functions \(f:B^n(1)\rightarrow \C\)
all of whose derivatives are bounded.
Let \(\sigma\in \CbinftySpace[B^n(1)][(0,\infty)]\) be bounded below.
The role of \(\sigma\) is to account for the fact that while we prove results using Lebesgue
measure in this section, when we apply the results, we need to work with another measure
(\(\sigma(x)\: dx\)); see Remark \ref{Rmk::SubProof::QualSubellip::PullBackOp}.

Let \(W_1,\ldots,W_r\) be smooth H\"ormander vector fields on \(B^n(1)\),
such that when written in standard coordinates \(W_j=\sum b_j^k \partial_{x_k}\)
with \(b_j^k\in \CbinftySpace[B^n(1)][\R]\).
We assume \(W_1,\ldots, W_r\) are H\"ormander vector fields of order \(m\) in the sense that if \(X_1,\ldots, X_q\)
are all the commutators of \(W_1,\ldots, W_r\) up to order \(m\), then
\begin{equation}\label{Eqn::Subellip::Proof::Apriori::LowerBoundOfDet}
    0<\tau:=\inf_{x\in B^n(1)}\max_{j_1,\ldots, j_n} \left| \det\left( X_{j_1}(x)|\cdots | X_{j_n}(x) \right) \right|.
\end{equation}
Assign to each \(W_j\) a formal degree \(\Wd_j\in \Nplus\), and define \(\degWd(\alpha)\)
as in Definition \ref{Defn::SubMainTheorem::OrderedMultiIndex}.

Fix \(\kappa\in \Nplus\) such that \(\Wd_j\) divides \(\kappa\) for each \(j\).
For each \(\degWd(\alpha),\degWd(\beta)\leq \kappa\), let \(a_{\alpha,\beta}\in \CbinftySpace[B^n(1)]\).
Define a partial differential operator
\begin{equation}\label{Eqn::Subellip::Proof::Apriori::opLFormula}
    \opL:=\sum_{\degWd(\alpha),\degWd(\beta)\leq \kappa} \sigma^{-1} \left( W^{\alpha} \right)^{*} \sigma a_{\alpha,\beta} W^{\beta},
\end{equation}
where \(*\) denotes the \(\LpSpace{2}[B^n(1)]\) adjoint.

\begin{numberedassumption}\label{Assumption::Subellip::Proof::Apriori}
    There exists \(C_1,C_2\geq 0\), \(\forall f\in \CzinftySpace[B^n(1)]\),
    \begin{equation*}
        \sum_{j=1}^r \BLpNorm{W_j^{\kappa/\Wd_j} f}{2}[B^n(1)]^2
        \leq C_1 \Real \Ltip*{f}{\sigma \opL f}[B^n(1)] + C_2 \BLpNorm{f}{2}[B^n(1)]^2.
    \end{equation*}
\end{numberedassumption}

Consider \(\Ropn\) with coordinates \((t,x)\in \R\times \R^n\cong \Ropn\). Let \(\tau\) be the dual variable
to \(t\) and \(\xi\) be the dual variable to \(x\). We work with the nonisotropic \(\LpSpace{2}\)-Sobolev spaces
\(\HsSpace{s}[\Ropn]\) with norm:
\begin{equation}\label{Eqn::Subellip::Proof::Apriori::HsRopnNorm}
    \HsNorm{u}{s}[\Ropn]:=\BLpNorm{ \left( 1+|\tau|^{4\kappa}+|\xi|^2 \right)^{s/2}\hat{u}(\tau,\xi)}{2}[\Ropn].
\end{equation}
We also use the isotropic \(\LpSpace{2}\)-Sobolev spaces \(\HsSpace{s}[\Rn]\) with norm
\begin{equation}\label{Eqn::Subellip::Proof::Apriori::HsRnNorm}
    \HsNorm{f}{s}[\Rn]:=\BLpNorm{ \left( 1+|\xi|^2 \right)^{s/2}\hat{f}(\xi)}{2}[\Rn].
\end{equation}
If we say \(\HsNorm{u}{s}[\Ropn]<\infty\) it means \(u\in \HsSpace{s}[\Ropn]\) (and similarly
for any other Sobolev space).

For \(\phi_1,\phi_2\in \CzinftySpace[\Ropn]\), we write \(\phi_1\prec \phi_2\) to mean
\(\phi_2=1\) on a neighborhood of \(\supp(\phi)\).

\begin{proposition}\label{Prop::Subellip::Proof::Apriori}
    Fix \(\epsilon_0:=\min\left\{ 1/m\max\{\Wd_j:1\leq j\leq r\}, 1/2 \right\}>0\) (see Remark \ref{Rmk::Subellip::Proof::epsilon0NotOptimal}).
    In the above setting (assuming Assumption \ref{Assumption::Subellip::Proof::Apriori})
    the following holds.
    \(\forall \phi_1,\phi_2\in \CzinftySpace[\R\times B^n(1)]\) with \(\phi_1\prec \phi_2\),
    \(\forall s\in \R\), \(\exists C_{s,\phi_1,\phi_2}\geq 0\), \(\forall u\in \Distributions[\R\times B^n(1)]\),
    \begin{equation}\label{Eqn::Subellip::Proof::Apriori::Est}
        \BHsNorm{\phi_1 u}{s+\epsilon_0}[\Ropn]
        \leq C_{s,\phi_1,\phi_2}
        \left( \BHsNorm{\phi_2 \left( \partial_t+\opL \right)u}{s}[\Ropn] + \BLpNorm{\phi_2 u}{2}[\Ropn] \right),
    \end{equation}
    where if the right-hand side is finite, so is the left-hand side.
    There is a constant \(L\in \N\), depending only on \(n\), \(r\), \(s\), and \(\max\{\Wd_j\}\) such that
    \(C_{s,\phi_1,\phi_2}\geq 0\) can be chosen to depend only on
    \(n\), \(r\), \(m\), \(s\), \(\phi_1\), \(\phi_2\), \(\max\{\Wd_j\}\), and upper bounds for
    \(\tau^{-1}\), \(\max_{\alpha,\beta} \CjNorm{a_{\alpha,\beta}}{L}\), \(\CjNorm{\sigma}{L}\),
    \(\max_{j,k}\CjNorm{b_j^k}{L}\), and \(\sup_{x\in B^n(1)}\sigma(x)^{-1}\).
\end{proposition}
\begin{proof}[Comments on the proof]
    This follows from standard methods based on ideas due to Kohn \cite{KohnLecturesOnDegenerateEllipticProblems}. 
    See \cite[Proposition 8.3.6]{StreetMaximalSubellipticity} for a priori estimates for \(\opL\) instead of \(\partial_t+\opL\).
    For the added twist to deal with \(\partial_t+\opL\) instead of \(\opL\) see \cite[Lecture 3]{KohnLecturesOnDegenerateEllipticProblems}.
    In the forthcoming paper \cite{StreetAPrioriEstimatesForMaximallySubellipticQuadraticForms}, the author proves
    a similar result in the more difficult setting of boundary value problems. The proof there  adapts to this easier setting.

    We make a few comments which will help the reader see this via standard techniques.
    Fix \(\gamma_0\in (0,1)\) with \(\phi_2\in \CzinftySpace[\R\times B^n(\gamma_0)]\).
    \cite[Lemma 8.3.3(i)]{StreetMaximalSubellipticity} shows
    \begin{equation*}
        \sum_{\degWd(\alpha)\leq \kappa}\BLpNorm{W^{\alpha} f}{2}
        \lesssim \sum_{j=1}^r \BLpNorm{W_j^{\kappa/\Wd_j}f}{2}+\BLpNorm{f}{2},\quad \forall f\in \CzinftySpace[B^n(\gamma_0)],
    \end{equation*}
    so that Assumption \ref{Assumption::Subellip::Proof::Apriori} implies
    \begin{equation*}
        \sum_{\degWd(\alpha)\leq \kappa}\BLpNorm{W^{\alpha} f}{2}
        \leq C_1' \Real \Ltip*{f}{\sigma \opL f} + C_2' \BLpNorm{f}{2},\quad \forall f\in \CzinftySpace[B^n(\gamma_0)].
    \end{equation*}
    \cite[Lemma 8.3.3(ii)]{StreetMaximalSubellipticity} shows
    \begin{equation*}
        \sum_{\degWd(\alpha)\leq \kappa-1} \BLpNorm{W^{\alpha}f}{2}\lesssim \sum_{\degWd(\alpha)\leq \kappa}\BHsNorm{W^{\alpha}f}{-\epsilon_0}[\Rn],\quad \forall f\in \CzinftySpace[B^n(\gamma_0)],
    \end{equation*}
        \begin{equation*}
        \sum_{\degWd(\alpha)\leq \kappa-1} \BHsNorm{W^{\alpha}f}{\epsilon_0}[\Rn]\lesssim \sum_{\degWd(\alpha)\leq \kappa}\BLpNorm{W^{\alpha}f}{2},\quad \forall f\in \CzinftySpace[B^n(\gamma_0)].
    \end{equation*}
    With the above estimates in hand, the result now follows from standard a priori techniques as described above.
\end{proof}

Proposition \ref{Prop::Subellip::Proof::Apriori} immediately gives the kinds of hypoelliptic estimates we require, at the unit scale,
as the next corollary shows.

\begin{corollary}\label{Cor::MaxSub::UnitScale::UnitScaleHypoEstimates}
    Suppose 
    \(u\in \Distributions[(-1/2,1/2)\times \Bno]\)
    satisfies \(\left( \partial_t+\opL \right)u=0\).
    Then there is a constant \(C_1\geq 1\) such that
    \begin{equation}\label{Eqn::MaxSub::UnitScale::SubellipticLtCor}
        \LpNorm{\partial_t u}{2}[(-1/4,1/4)\times B^n(1/2)][dt\times \sigma dx]
        \leq C_1 \LpNorm{ u}{2}[(-1/2,1/2)\times \Bno][dt\times \sigma dx],
    \end{equation}
    and for every \(\alpha\), \(\exists C_\alpha\geq 1\),
    \begin{equation}\label{Eqn::MaxSub::UnitScale::SubellipticLinftyCor}
        \sup_{\substack{t\in (-1/4,1/4) \\ x\in B^n(1/2)}} \left| W^{\alpha} u(t,x) \right|
        \leq C_\alpha \LpNorm{ u}{2}[(-1/2,1/2)\times \Bno][dt\times \sigma dx],
    \end{equation}
    where if the right-hand side of either \eqref{Eqn::MaxSub::UnitScale::SubellipticLtCor} or 
    \eqref{Eqn::MaxSub::UnitScale::SubellipticLinftyCor} is finite, so is the corresponding left-hand side.
    Here, \(C_1\) and \(C_\alpha\) can depend on all the same quantities as \(C_{s,\phi_1,\phi_2}\) does in 
    Proposition \ref{Prop::Subellip::Proof::Apriori}, except for \(\phi_1\) and \(\phi_2\),
    and where \(s\) is chosen to be \(2\kappa\) for \(C_1\) and chosen to depend on \(\alpha\) and \(n\) for \(C_\alpha\).
\end{corollary}
\begin{proof}
    Fix \(\phi_1,\phi_2,\phi_3\in \CzinftySpace[(-1/2,1/2)\times \Bno]\) with \(\phi_1\prec \phi_2\prec \phi_3\),
    \(\phi_1=1\) on \((-3/8,3/8)\times B^n(3/4)\), and \(0\leq \phi_2\leq 1\).
    Using Proposition \ref{Prop::Subellip::Proof::Apriori} with \(u\) replaced by \(\phi_3 u\in \Distributions[\R\times \Bno]\)
    and this choice of \(\phi_1\), \(\phi_2\)
    we have, \(\forall s\in \R\),
    \begin{equation}\label{Eqn::MaxSub::UnitScale::UnitScaleHypoEstimates::Tmp1}
        \BHsNorm{\phi_1 u}{s+\epsilon_0}[\Ropn]
        \leq C_{s}
         \BLpNorm{\phi_2 u}{2}[\Ropn]
         \leq C_s\BLpNorm{u}{2}[(-1/2,1/2)\times \Bno].
    \end{equation}

    Applying \eqref{Eqn::MaxSub::UnitScale::UnitScaleHypoEstimates::Tmp1} with \(s=2\kappa\), we see
    \begin{equation*}
        \begin{split}
            &
            \LpNorm{\partial_t  u}{2}[(-1/4,1/4)\times B^n(1/2)][dt\times \sigma dx]
            \lesssim \LpNorm{\partial_t \phi_1 u}{2}[(-1/4,1/4)\times B^n(1/2)][dt\times dx]
            \\&\lesssim \BHsNorm{\phi_1 u}{2\kappa}[\Ropn]
            \lesssim \BLpNorm{u}{2}[(-1/2,1/2)\times \Bno][dt\times dx]
            \lesssim \BLpNorm{u}{2}[(-1/2,1/2)\times \Bno][dt\times \sigma dx],
        \end{split}    
    \end{equation*}
    where in the last first and last estimates we have used \(\sigma\approx 1\),
    establishing \eqref{Eqn::MaxSub::UnitScale::SubellipticLtCor}.

    By the Sobolev embedding theorem, we have for some \(s=s(\alpha,n)\),
    \begin{equation}\label{Eqn::MaxSub::UnitScale::UnitScaleHypoEstimates::Tmp2}
        \sup_{\substack{t\in (-1/4,1/4) \\ x\in B^n(1/2)}} \left| W^{\alpha} u(t,x) \right|
        \lesssim \HsNorm{\phi_1 u}{s}.
    \end{equation}
    Combining \eqref{Eqn::MaxSub::UnitScale::UnitScaleHypoEstimates::Tmp2} with 
    \eqref{Eqn::MaxSub::UnitScale::UnitScaleHypoEstimates::Tmp1}, we have
    \begin{equation*}
        \sup_{\substack{t\in (-1/4,1/4) \\ x\in B^n(1/2)}} \left| W^{\alpha} u(t,x) \right|
        \lesssim \BLpNorm{u}{2}[(-1/2,1/2)\times \Bno]
        \approx \BLpNorm{u}{2}[(-1/2,1/2)\times \Bno][dt\times \sigma dx],
    \end{equation*}
    where in the last estimate we have used \(\sigma\approx 1\),
    establishing \eqref{Eqn::MaxSub::UnitScale::SubellipticLinftyCor}.
\end{proof}

\begin{remark}\label{Rmk::Subellip::Proof::epsilon0NotOptimal}
    The choice of \(\epsilon_0\) in Proposition \ref{Prop::Subellip::Proof::Apriori}
    is not optimal; though finding the optimal \(\epsilon_0\) is irrelevant for the proof of Theorem \ref{Thm::SubMainThm::MainThm}.
    What may be surprising is that Proposition \ref{Prop::Subellip::Proof::Apriori}
    (with a non-sharp \(\epsilon_0>0\)) implies the sharp subelliptic estimates for the operator \(\opL\)
    from Theorem \ref{Thm::SubMainThm::MainThm}. Indeed,
    as Corollary \ref{Cor::Subellip::MaxSub::HeatImpliesMaxSub} shows, \(\opL\) is maximally subelliptic.
    Sharp subelliptic estimates for general maximally subelliptic operators are known;
    see \cite[Corollary 8.2.5]{StreetMaximalSubellipticity} for sharp estimates in \(\LpSpace{p}\)-Sobolev
    spaces, and \cite[Theorem 8.1.1(v) and Section 8.2]{StreetMaximalSubellipticity} for sharp estimates
    in a wide variety of spaces.
    Thus, non-sharp subelliptic estimates can be used as an important step in obtaining sharp subelliptic estimates
    (by way of Gaussian bounds for the corresponding heat operator).
\end{remark}

        \subsubsection{Qualitative consequences of subellipticity}
        We record some standard qualitative consequences of the subelliptic estimate (Proposition \ref{Prop::Subellip::Proof::Apriori})
for operators on manifolds.

Let \(\opL\) be as in Section \ref{Section::Subellip::MainThm};
i.e., \(\opL\) is given by \eqref{Eqn::SubMainThem::FormulaForL}
for some H\"ormander vector fields with formal degrees \(\WWd\).
We assume throughout this section that Assumption \ref{Assump::SubMainTheorem::MaxSubTypeIII}
holds. Recall, we are working on a connected, smooth manifold \(\Manifold\)
with smooth, strictly positive density \(\mu\).

\begin{remark}\label{Rmk::SubProof::QualSubellip::PullBackOp}
    Let \(\Phi:B^n(1)\xrightarrow{\sim} \Phi(B^n(1))\subseteq \Manifold\) be a smooth coordinate chart.
    Define \(\sigma\in \CinftySpace[B^n(1)][(0,\infty)]\) by \(\sigma dx = d\Phi^{*}\mu\).
    Define \(\opL_{\Phi}:=\Phi^{*}\opL \Phi_{*}\); i.e.,
    \begin{equation*}
        \opL_{\Phi} f = \left( \opL (f\circ \Phi^{-1}) \right)\circ \Phi,\quad f\in \CzinftySpace[B^n(1)].
    \end{equation*}
    Let \(W_j^{\Phi}:=\Phi^{*}W_j\). Then,
    \begin{equation}\label{Eqn::SubProof::QualSubellip::opLPhi}
        \opL_\Phi=\sum_{\degWd(\alpha),\degWd(\beta)\leq \kappa} \sigma^{-1}\left( \left( W^{\Phi} \right)^{\alpha} \right)^{*} \sigma\left(  a_{\alpha,\beta}\circ \Phi \right) \left( W^{\Phi} \right)^{\beta}.
    \end{equation}
    Indeed, for \(f,g\in \CzinftySpace[B^n(1)]\), with \(F:=f\circ \Phi^{-1}\) and \(G=g\circ \Phi^{-1}\), we have
    \begin{equation*}
        \begin{split}
            &\Ltip{f}{\sigma \opL_\Phi g}[B^n(1)]
            =\sum_{\degWd(\alpha),\degWd(\beta)\leq \kappa} \Ltip*{\left( W^{\Phi} \right)^{\alpha} f}{\sigma a_{\alpha,\beta}\circ \Phi \left( W^{\Phi} \right)^{\beta}g}[B^n(1)]
            \\&=\sum_{\degWd(\alpha),\degWd(\beta)\leq \kappa} \Ltip*{ \left( W^{\alpha} F \right)\circ \Phi}{a_{\alpha,\beta}\circ \Phi \left( W^{\beta}G\right)\circ \Phi}[B^n(1)][\Phi^{*}\mu]
            \\&=\sum_{\degWd(\alpha),\degWd(\beta)\leq \kappa} \Ltip*{ W^{\alpha} F }{a_{\alpha,\beta}W^{\beta}G}[\Manifold][\mu]
            =\Ltip*{F}{\opL G}[\Manifold][\mu]
            \\&=\Ltip*{f}{\sigma \left( \opL G \right)\circ \Phi}[B^n(1)]
            =\Ltip*{f}{\sigma \left( \opL \left( g\circ \Phi^{-1} \right)\right)\circ \Phi}[B^n(1)].
        \end{split}
    \end{equation*}
\end{remark}

In what follows, we use non-isotropic \(\LpSpace{2}\)-Sobolev spaces \(\HsSpace{s}[\R\times \Manifold]\);
in a local coordinate system the norm is given by \eqref{Eqn::Subellip::Proof::Apriori::HsRopnNorm}.
We only  consider \(\HsNorm{u}{s}[\R\times \Manifold]\) for \(u\) with compact support,
and the equivalence class of this norm is therefore well-defined. 

\begin{proposition}\label{Prop::SubProof::QualSubellip::QualSubellip}
    Fix \(\Omega\Subset \Manifold\) open and relatively compact.
    \(\exists \epsilon=\epsilon(\Omega)>0\), \(\forall s\in \R\),
    \(\forall \phi_1,\phi_2\in \CzinftySpace[\R\times \Omega]\) with \(\phi_1\prec \phi_2\),
    \(\exists C_{\Omega,s,\phi_1,\phi_2}\geq 0\), \(\forall u\in \Distributions[\R\times \Manifold]\),
    \begin{equation*}
        \HsNorm{\phi_1 u}{s+\epsilon}[\R\times \Manifold]
        \leq C_{\Omega,s,\phi_1,\phi_2}
        \left( \HsNorm{\phi_2 \left( \partial_t +\opL \right)u}{s}[\R\times \Manifold]+
        \LpNorm{\phi_2 u}{2}[\R\times \Manifold][dt\times d\mu] \right),
    \end{equation*}
    where if the right-hand side is finite, so is the left-hand side.
\end{proposition}
\begin{proof}
    We will show, for each \(\zeta\in \Manifold\), there exists a neighborhood \(\Omega_\zeta\)
    of \(\zeta\) such that the result holds \(\Omega=\Omega_\zeta\). The full result then follows
    by compactness of \(\overline{\Omega}\) and a simple partition of unity argument, which we leave to the reader.

    Fix \(\zeta\in \Manifold\).
    Let \(\Phi:B^n(2)\xrightarrow{\sim}\Phi(B^n(2))\subseteq \Manifold\) be a smooth coordinate chart
    with \(\Phi(0)=\zeta\). Set \(\Omega_\zeta:=\Phi(B^n(1))\).
    Define \(\sigma\in \CzinftySpace[B^n(2)]\) by \(d\Phi^{*}\mu=\sigma(x)\: dx\).
    Let \(W^\Phi\) and \(\opL_{\Phi}\) be as in Remark \ref{Rmk::SubProof::QualSubellip::PullBackOp}
    with this choice of \(\Phi\). 
    Since \(B^n(1)\Subset B^n(2)\),
    \(\sigma\in \CbinftySpace[B^n(1)][(0,\infty)]\) and is bounded below,
    \(a_{\alpha,\beta}\circ \Phi\in \CbinftySpace[B^n(1)]\), the coefficients
    of \(W_j^{\Phi}\) are in \(\CbinftySpace[B^n(1)][\R]\) for each \(j\),
    and \(W_1^\Phi,\ldots, W_r^\Phi\) satisfy H\"ormander's condition of order \(m\) for some \(m\)
    in the sense that \eqref{Eqn::Subellip::Proof::Apriori::LowerBoundOfDet} holds for some \(m\).

    We claim that Assumption \ref{Assumption::Subellip::Proof::Apriori} holds for \(\opL_{\Phi}\).
    Using that \(\sigma\approx 1\) and Assumption \ref{Assump::SubMainTheorem::MaxSubTypeIII},
    we have, \(\forall f\in \CzinftySpace[B^n(1)]\),
    \begin{equation*}
    \begin{split}
         &\sum_{j=1}^r \BLpNorm{\left( W_j^{\Phi} \right)^{\kappa/\Wd_j} f}{2}[B^n(1)]
         \approx \sum_{j=1}^r \BLpNorm{\left( W_j^{\Phi} \right)^{\kappa/\Wd_j} f}{2}[B^n(1)][\sigma\: dx]
         =\sum_{j=1}^r \BLpNorm{W_j^{\kappa/\Wd_j} f\circ \Phi^{-1}}{2}[\Manifold][\mu]
         \\&\leq C_1 \Real \Ltip*{f\circ \Phi^{-1}}{\opL f\circ \Phi^{-1}}[\Manifold][\measure]
        +C_2 \BLpNorm{f\circ \Phi^{-1}}{2}[\Manifold][\measure]^2
        \\&=C_1 \Ltip{f}{\sigma \opL_{\Phi} f}[B^n(1)] + C_2 \BLpNorm{f}{2}[B^n(1)][\sigma\: dx]^2
        \leq C_1 \Ltip{f}{\sigma \opL_{\Phi} f}[B^n(1)] + C_2' \BLpNorm{f}{2}[B^n(1)]^2.
    \end{split}
    \end{equation*}
    We conclude
    \begin{equation*}
    \begin{split}
         &\sum_{j=1}^r \BLpNorm{\left( W_j^{\Phi} \right)^{\kappa/\Wd_j} f}{2}[B^n(1)]
         \leq C_1'' \Ltip{f}{\sigma \opL_{\Phi} f}[B^n(1)] + C_2'' \BLpNorm{f}{2}[B^n(1)]^2,
    \end{split}
    \end{equation*}
    establishing Assumption \ref{Assumption::Subellip::Proof::Apriori}.

    Let \(\phi_1,\phi_2\in \CzinftySpace[\R\times \Omega_\zeta]\) with \(\phi_1\prec \phi_2\).
    Set \(\psi_j(t,x)=\phi_j(t,\Phi(x))\in \CzinftySpace[\R\times B^n(1)]\) so that  \(\psi_1\prec \psi_2\).
    With \(\epsilon_0>0\) as in Proposition \ref{Prop::Subellip::Proof::Apriori}
    that proposition shows for \(u\in \Distributions[\R\times\Manifold]\),
    \begin{equation*}
        \begin{split}
            &\BHsNorm{\psi_1 \Phi^{*} u}{s+\epsilon_0}[\Ropn]
            \lesssim
            \BHsNorm{\psi_2 \left( \partial_t+\opL_{\Phi} \right)\Phi^{*} u}{s}[\Ropn] + \BLpNorm{\psi_2 \Phi^{*} u}{2}[\Ropn]
            \\&\lesssim 
            \BHsNorm{\psi_2 \left( \partial_t+\opL_{\Phi} \right)\Phi^{*} u}{s}[\Ropn] + \BLpNorm{\psi_2 \Phi^{*} u}{2}[\Ropn][dt\times \sigma\: dx].
        \end{split}
    \end{equation*}
    Changing variables shows
        \begin{equation*}
        \HsNorm{\phi_1 u}{s+\epsilon}[\R\times \Manifold]
        \lesssim 
        \HsNorm{\phi_2 \left( \partial_t +\opL \right)u}{s}[\R\times \Manifold]+
        \LpNorm{\phi_2 u}{2}[\R\times \Manifold][dt\times d\mu],
    \end{equation*}
    completing the proof.
\end{proof}

\begin{corollary}\label{Cor::SubProof::QualSubellip::EstCjNormTwoVars}
    Fix \(\Omega\Subset \Manifold\) open and relatively compact. \(\forall L\in \N\),
    \(\exists N\in \N\), \(\forall \phi_1,\phi_2\in \CzinftySpace[\R\times\Omega]\) with \(\phi_1\prec\phi_2\),
    \(\exists C\geq 0\), \(\forall u\in \Distributions[\R\times \Manifold]\),
    \begin{equation*}
        \CjNorm{\phi_1 u}{L} \leq C\sum_{j=0}^N \BLpNorm{\phi_2 \left( \partial_t+\opL \right)^j u}{2}[\R\times \Manifold][dt\times d\mu],
    \end{equation*}
    where if the right-hand side is finite, then \(\phi_1 u\in \CjSpace{L}[\R\times \Manifold]\).
\end{corollary}
\begin{proof}
    Let \(\epsilon=\epsilon(\Omega)>0\) be as in Proposition \ref{Prop::SubProof::QualSubellip::QualSubellip}.
    By the Sobolev embedding theorem, for \(s=s(L,\Omega)>0\) sufficiently large,
    we have \(\CjNorm{\phi_1 u}{L}\lesssim \HsNorm{\phi_1 u}{s}[\R\times \Manifold]\).
    Pick \(N\in \N\) so that \(N\epsilon\geq s\). 
    Let \(\psi_1,\psi_2,\ldots,\psi_N\in \CzinftySpace[\R\times \Omega]\) be such that
    \(\phi_1\prec \psi_1\prec\psi_2\prec\cdots \prec\psi_N\prec  \phi_2\).
    We have, by \(N\)
    applications of Proposition \ref{Prop::SubProof::QualSubellip::QualSubellip},
    \begin{equation*}
    \begin{split}
         &\CjNorm{\phi_1 u}{L}\lesssim \BHsNorm{\phi_1 u}{s}[\R\times \Manifold]
         \lesssim \BHsNorm{\psi_1 \left( \partial_t+\opL \right)u}{s-\epsilon}[\R\times \Manifold] + \BLpNorm{\psi_1 u}{2}
         \\&\lesssim \BHsNorm{\psi_2 \left( \partial_t+\opL \right)^2 u}{s-\epsilon}[\R\times \Manifold]+ \BLpNorm{\psi_2 (\partial_t+\opL) u}{2} + \BLpNorm{\psi_1 u}{2}
         \\&\lesssim \cdots
         \\&\lesssim \BHsNorm{\psi_N \left( \partial_t+\opL \right)^N u}{s-N\epsilon}[\R\times \Manifold]+\sum_{j=0}^{N-1} \BLpNorm{\psi_{j+1} (\partial_t+\opL)^j u}{2} 
         \\&\lesssim \sum_{j=0}^N \BLpNorm{\phi_2 \left( \partial_t+\opL \right)^j u}{2},
    \end{split}
    \end{equation*}
    where the final estimate used \(N\epsilon\geq s\), completing the proof.
\end{proof}

\begin{corollary}\label{Cor::SubProof::QualSubellip::EstCjNormOneVar}
    Fix \(\Omega\Subset \Manifold\) open and relatively compact.
    \(\forall L\in \N\), \(\exists N\in \N\), \(\forall \psi_1,\psi_2\in \CzinftySpace[\Omega]\)
    with \(\psi_1\prec \psi_2\), \(\exists C\geq 0\), \(\forall f\in \Distributions[\Manifold]\),
    \begin{equation}\label{Eqn::SubProof::QualSubellip::EstCjNormOneVar::Estimate}
        \CjNorm{\psi_1 f}{L} \leq C \sum_{j=0}^N \BLpNorm{\psi_2 \opL^j f}{2}[\Manifold][\mu],
    \end{equation}
    where if the right-hand side is finite, then \(\psi_1 f\in \CjSpace{L}[\Manifold]\).
\end{corollary}
\begin{proof}
    Let \(u\in \Distributions[\R\times \Manifold]\) be defined by \(u(t,\zeta)=f(\zeta)\).
    Let \(\gamma_1,\gamma_2\in \CzinftySpace[\R]\) satisfy \(\gamma_1\prec \gamma_2\)
    and \(\gamma_1=1\) on a neighborhood of \(0\).
    Let \(\phi_j(t,\zeta)=\gamma_j(t)\psi_j(\zeta)\).
    From here, the result follows from Corollary \ref{Cor::SubProof::QualSubellip::EstCjNormTwoVars}
    applied with this choice of \(u\), \(\phi_1\), and \(\phi_2\).
\end{proof}

\begin{corollary}\label{Cor::SubProof::QualSubellip::DomainEmbedsInSmooth}
    Suppose \((A,\Domain[A])\) is an unbounded operator on \(\LpSpace{2}[\Manifold][\mu]\)
    with \(A=-\opL\big|_{\Domain[A]}\). Then, \(\Domain[A^\infty]\subseteq \CinftySpace[\Manifold]\)
    and the inclusion \(\Domain[A^\infty]\hookrightarrow \CinftySpace[\Manifold]\) is continuous.
\end{corollary}
\begin{proof}
    Suppose \(f\in \Domain[A^\infty]\). Then \(\opL^j f= \left( -A \right)^j f\in \LpSpace{2}\), \(\forall j\).
    Corollary \ref{Cor::SubProof::QualSubellip::EstCjNormOneVar} shows
    \(f\in \CinftySpace[\Manifold]\). Moreover,
    for \(\phi\in \CzinftySpace[\Manifold]\), \eqref{Eqn::SubProof::QualSubellip::EstCjNormOneVar::Estimate}
    shows \(\forall L\in \N\), \(\exists N\in \N\), 
    \begin{equation*}
        \CjNorm{\phi f}{L}\lesssim \sum_{j=0}^N \BLpNorm{A^j f}{2}.
    \end{equation*}
    It follows that the inclusion \(\Domain[A^\infty]\hookrightarrow \CinftySpace[\Manifold]\) is continuous.
\end{proof}

\begin{proposition}\label{Prop::SubProof::QualSubellip::ExistsHeatKernel}
    Let \(T(t)\) be a strongly continuous semi-group on \(\LpSpace{2}[\Manifold][\mu]\) with
    generator \((A,\Domain[A])\). Suppose \(\CzinftySpace[\Manifold]\subseteq \Domain[A]\)
    and \(A=-\opL\big|_{\Domain[A]}\). Then, there exists a unique \(K_t(x,y)\in \CinftySpace[(0,\infty)\times \Manifold\times \Manifold]\)
    such that \(K_t(x,\cdot)\in \LpSpace{2}[\Manifold][\mu]\), \(\forall t,x\) and
    \begin{equation*}
        T(t) f(x) = \int K_t(x,y)f(y)\: d\mu(y), \quad \forall f\in \LpSpace{2}[\Manifold][\mu].
    \end{equation*}
\end{proposition}

To prove Proposition \ref{Prop::SubProof::QualSubellip::ExistsHeatKernel} we use the next two lemmas.
For the first, we write 
\(\CinftySpacetx[(0,\infty)\times \Manifold][\LpSpacewy{2}[\Manifold][\mu]]\) for the space
of those functions \(K_t(x,y)\) which are smooth as functions \((t,x)\mapsto K_t(x,\cdot)\), taking
values in \(\LpSpace{2}[\Manifold][\mu]\) in the
\(y\)-variable, where \(\LpSpace{2}[\Manifold][\mu]\) is given the weak topology.\footnote{In fact,
    the weak topology on \(\LpSpace{2}[\Manifold][\mu]\) can be replaced with the strong topology,
    since a \(\CinftySpace\) function taking values in a Banach space given the weak topology is
    \(\CinftySpace\) with the Banach space given the strong topology. We do not use this fact.}

\begin{lemma}\label{Lemma::SubProof::QualSubellip::ExistsHeatKernelWeak}
    Let \(T(t)\) be a strongly continuous semi-group on \(\LpSpace{2}[\Manifold][\mu]\) with
    generator \((A,\Domain[A])\). Suppose  \(A=-\opL\big|_{\Domain[A]}\).
    Then, there exists a unique \(K_t(x,y)\in \CinftySpacetx[(0,\infty)\times \Manifold][\LpSpacewy{2}[\Manifold][\mu]]\),
    such that
    \begin{equation*}
        T(t) f(x) = \int K_t(x,y)f(y)\: d\mu(y), \quad \forall f\in \LpSpace{2}[\Manifold][\mu].
    \end{equation*}
    Moreover, if \(\phi\in \CzinftySpace[(0,\infty)\times \Manifold]\), and \(\OpP\) is any partial differential
    operator with smooth coefficients in the \((t,x)\) variables, then
    \begin{equation}\label{Eqn::SubProof::QualSubellip::ExistsHeatKernelWeak::BoundDerivs}
        \sup_{t,x} \BLpNorm{ \OpP \phi(t,x) K_t(x,\cdot)}{2}[\Manifold][\mu]<\infty.
    \end{equation}
\end{lemma}
\begin{proof}
    For \(v\in \Domain[A]\), \(\partial_t T(t) v=AT(t)v=-\opL T(t)v\) (see \cite[Chapter 2, Lemma 1.3(ii)]{EngelNagelAShortCourseOnOperatorSemigroups}),
    and therefore \((\partial_t+\opL)T(t)v=0\) in the sense of distributions.
    
    Fix \(\phi_1\in \CzinftySpace[(0,\infty)\times \Manifold]\) and let \(\phi_2\in \CzinftySpace[\R\times \Manifold]\) be such that
    \(\phi_1\prec \phi_2\).
    Take \(a>0\) such that \(\phi_2\in \CzinftySpace[(0,a)\times \Manifold]\).
    Using that \((\partial_t+\opL)^j T(t)v=0\) for \(j\geq 1\),
    Corollary \ref{Cor::SubProof::QualSubellip::EstCjNormTwoVars} implies, \(\forall L\in \N\),
    \begin{equation}\label{Eqn::SubProof::QualSubellip::ExistsHeatKernelWeak::FirstBound}
        \CjNorm{\phi_1 T(t)v}{L} \lesssim \LpNorm{\phi_2 T(t)v}{2}[\R\times \Manifold][dt\times d\mu]\lesssim \int_0^a \LpNorm{T(t)v}{2}[\Manifold][\mu]\: dt
        \lesssim \LpNorm{v}{2}[\Manifold][\mu],
    \end{equation}
    where the \(\CjSpace{L}\) norm is in both variables in \(\R\times \Manifold\).
    Since \(\Domain[A]\subseteq \LpSpace{2}[\Manifold][\mu]\)
    is dense \cite[Chapter 2, Theroem 1.4]{EngelNagelAShortCourseOnOperatorSemigroups},
    \eqref{Eqn::SubProof::QualSubellip::ExistsHeatKernelWeak::FirstBound} shows
    \(\phi_1 T(\cdot):\LpSpace{2}[\Manifold][\mu]\rightarrow \CjSpace{\infty}[\R\times \Manifold]\)
    is continuous.  Since \(\phi_1\in \CzinftySpace[(0,\infty)\times \Manifold]\) was arbitrary, this implies
    \(T(\cdot):\LpSpace{2}[\Manifold][\mu]\rightarrow \CjSpace{\infty}[(0,\infty)\times \Manifold]\)
    is continuous.

    By the Riesz representation theorem, for each fixed \((t,x)\in (0,\infty)\times \Manifold\),
    the map \(g\mapsto T(t)g(x)\) is given by integration against a unique
    \(K_t(x,\cdot)\in \LpSpace{2}[\Manifold][\mu]\).  For fixed \(g\),
    the map \((t,x)\mapsto T(t)g(x)\) is in \(\CinftySpace[(0,\infty)\times \Manifold]\),
    which shows \(K_t(x,y)\in \CinftySpacetx[(0,\infty)\times \Manifold][\LpSpacewy{2}[\Manifold][\mu]]\).
    Finally, using \eqref{Eqn::SubProof::QualSubellip::ExistsHeatKernelWeak::FirstBound}
    with \(\phi\prec \phi_1\),
    we have for \(L=L(\OpP)\in \N\) large,
    \begin{equation*}
    \begin{split}
         &\sup_{t,x} \BLpNorm{ \OpP \phi(t,x) K_t(x,\cdot)}{2}[\Manifold][\mu]
         \lesssim \sup_{\substack{v\in \Domain[A]\\ \LpNorm{v}{2}=1}} \CjNorm{\phi_1 T(t)v}{L}
         \lesssim 1,
    \end{split}
    \end{equation*}
    establishing \eqref{Eqn::SubProof::QualSubellip::ExistsHeatKernelWeak::BoundDerivs}.
\end{proof}

The next lemma is standard.

\begin{lemma}\label{Lemma::SubProof::QualSubellip::SobolevTrick}
    Let \(M(t,x,y)\in \LpSpace{2}[\R\times \R^n\times \R^n]\).
    Suppose \(\forall j,\alpha,\beta\),
    \begin{equation}\label{Eqn::SubProof::QualSubellip::SobolevTrick::Assump}
        \partial_t^j M(t,x,y), \partial_x^{\alpha}M_t(x,y), \partial_y^\beta M_t(x,y)\in \LpSpace{2}[\R\times \R^n\times \R^n].
    \end{equation}
    Then, \(M(t,x,y)\in \CinftySpace[\R\times \R^n\times \R^n]\).
\end{lemma}
\begin{proof}
    Let \((\tau,\xi,\eta)\) be dual to \((t,x,y)\).
    \eqref{Eqn::SubProof::QualSubellip::SobolevTrick::Assump} shows, \(\forall m\in \N\),
    \begin{equation*}
        \tau^{2m} \Mh(\tau,\xi,\eta), |\xi|^{2m} \Mh(\tau,\xi,\eta), |\eta|^{2m}\Mh(\tau,\xi,\eta)\in \LpSpace{2}[\R\times \R^n\times \R^n].
    \end{equation*}
    It follows that \(M\in \HsSpace{2m}[\R\times \R^n\times \R^n]\), \(\forall m\in \N\),
    where \(\HsSpace{s}\) denotes the standard \(\LpSpace{2}\)-Sobolev space of order \(s\).
    From here, the result follows from the Sobolev Embedding Theorem.
\end{proof}

\begin{proof}[Proof of Proposition \ref{Prop::SubProof::QualSubellip::ExistsHeatKernel}]
    Let \(K_t(x,y)\) be the function from Lemma \ref{Lemma::SubProof::QualSubellip::ExistsHeatKernelWeak}.
    Proposition \ref{Prop::Subellip::MaxSub::SymmetricAssumps} shows
    that Lemma \ref{Lemma::SubProof::QualSubellip::ExistsHeatKernelWeak}
    also applies to \(T(t)^{*}\); so  there is a unique
    \(L_t(x,y)\in \CinftySpacetx[(0,\infty)\times \Manifold][\LpSpacewy{2}[\Manifold][\mu]]\)
    such that
    \begin{equation*}
        T(t)^{*} f(x) = \int L_t(x,y)f(y)\: d\mu(y), \quad \forall f\in \LpSpace{2}[\Manifold][\mu],
    \end{equation*}
    satisfying \eqref{Eqn::SubProof::QualSubellip::ExistsHeatKernelWeak::BoundDerivs}.
    We have \(L_t(x,y)=\overline{K_t(y,x)}\), and 
    we conclude that if \(\OpP_{t,x}\) is any partial differential operator with smooth coefficients
    in the \(t,x\) variables
    and \(\psi\in \CzinftySpace[(0,\infty)\times \Manifold\times \Manifold]\),
    then
    \begin{equation}\label{Eqn::SubProof::QualSubellip::ExistsHeatKernel::SumNorms}
        \BLpNorm{ \OpP_{t,x} \psi(t,x,y)K_t(x,y)}{2}[(0,\infty)\times \Manifold\times \Manifold]
        +\BLpNorm{ \OpP_{t,y} \psi(t,x,y) K_t(x,y)}{2}[(0,\infty)\times \Manifold\times \Manifold]<\infty.
    \end{equation}
    
    Fix \((t_0,x_0,y_0)\in (0,\infty)\times \Manifold\times \Manifold\).
    Let \(U,V\subseteq \Manifold\) be small neighborhoods of \(x_0\) and \(y_0\), respectively,
    which are diffeomorphic to connected open subsets of \(\R^n\).
    Let \(\psi\in \CzinftySpace[(0,\infty)\times U\times V]\)
    equal \(1\) on a neighborhood of \((t_0,x_0,y_0)\). 
    \eqref{Eqn::SubProof::QualSubellip::ExistsHeatKernel::SumNorms}
    shows that Lemma \ref{Lemma::SubProof::QualSubellip::SobolevTrick} applies
    to give \(\psi(t,x,y) K_t(x,y)\in \CinftySpace[(0,\infty)\times \Manifold\times \Manifold]\),
    and therefore \(K_t(x,y)\) is smooth near \((t_0,x_0,y_0)\). As 
    \((t_0,x_0,y_0)\in (0,\infty)\times \Manifold\times \Manifold\)
    was arbitrary, it follows that \(K_t(x,y)\in \CinftySpace[(0,\infty)\times \Manifold\times \Manifold]\), completing the proof.


\end{proof}

        \subsubsection{Scaling}\label{Section::MaxSub::Scaling}
        The results at the unit scale in Corollary \ref{Cor::MaxSub::UnitScale::UnitScaleHypoEstimates}
imply results at every scale by a general scaling procedure introduced by
Nagel, Stein, and Wainger \cite{NagelSteinWaingerBallsAndMetricsDefinedByVectorFieldsIBasicProperties},
and was later worked on by several authors, see for example
\cite{TaoWrightLpImprovingBoundsForAveragesAlongCurves,StreetMultiParameterCarnotCaratheodoryBalls,MontanariMorbidelliNonsmoothHormanderVectorFieldsAndTheirControlBalls,StovallStreetI,StovallStreetII,StovallStreetIII,StreetYaoImprovingRegularity}.
We state a version of this scaling.
For many more details and a further discussion, we refer the reader to  \cite[Sections 3.3 and 3.5]{StreetMaximalSubellipticity}.

We take the same setting as in the start of Section \ref{Section::Subellip::MainThm}; so that we have H\"ormander vector fields with formal degrees
\(\WWd=\left\{ \left( W_1,\Wd_1 \right),\ldots, \left( W_r,\Wd_r \right) \right\}\)
on a connected, smooth manifold \(\Manifold\) of dimension \(n\) with the corresponding metric \(\metric\). \(\mu\)
is a smooth, strictly positive density on \(\Manifold\).

\begin{theorem}[{See  \cite[Sections 3.3 and 3.5]{StreetMaximalSubellipticity}}]\label{Thm::MaxSub::Scaling::NSWScaling}
    Let \(\Compact_1\Subset \Omega \Subset \Manifold\) where \(\Compact_1\)
    is compact and \(\Omega\) is open and relatively compact. There exists \(\delta_1\in (0,1]\), \(\xi_0\in (0,1)\)
    such that the following holds
    \begin{enumerate}[label=(\alph*),series=scalingenumeration]
        \item\label{Item::MaxSub::Scaling::Doubling}  \(\exists D_1\geq 1\), \(\forall x\in \Compact_1\), \(\forall \delta\in (0,\delta_1]\), \(\measure\left( \MetricBall{x}{2\delta} \right)\leq D_1 \measure\left( \MetricBall{x}{\delta} \right)<\infty\).
    \end{enumerate}
    For \(x\in \Compact_1\) and \(\delta\in (0,\delta_1]\), there exists a smooth map \(\Phi_{x,\delta}:\Bno\rightarrow \Omega\)
    such that:
    \begin{enumerate}[resume*=scalingenumeration]
        \item \(\Phi_{x,\delta}(0)=x\).
        \item \(\Phi_{x,\delta}(\Bno)\subseteq \Omega\) is open and \(\Phi_{x,\delta}:\Bno\rightarrow \Phi_{x,\delta}(\Bno)\) is a smooth diffeomorphism.
        \item\label{Item::MaxSub::Scaling::Containments} \(\MetricBall{x}{\xi_0\delta}\subseteq \Phi_{x,\delta}(B^n(1/2))\subseteq \Phi_{x,\delta}(B^n(1))\subseteq \MetricBall{x_0}{\delta/2}\).
    \end{enumerate}
    For \(x\in \Compact_1\) and \(\delta\in (0,\delta_1]\), let \(W_j^{x,\delta}:=\Phi_{x,\delta}^{*}\delta^{\Wd_j}W_j\). Then,
    \begin{enumerate}[resume*=scalingenumeration]
        \item\label{Item::MaxSub::Scaling::UniformHormander} \(W_1^{x,\delta},\ldots, W_r^{x,\delta}\) are smooth H\"ormander vector fields on \(\Bno\), uniformly for \(x\in \Compact_1\) and \(\delta\in (0,\delta_1]\).
            I.e., the conditions from Section \ref{Section::Subellip::Proof::UnitScale} hold uniformly for in \(x\in \Compact_1\) and \(\delta\in (0,\delta_1]\).
            In particular, 
            \begin{enumerate}[label=(\alph{enumi}.\arabic*)]
                \item If we treat \(W_j^{x,\delta}(u)\) as a vector in \(\Rn\) (by using the standard basis \(\partial_{u_1},\ldots, \partial_{u_n}\)),
                    then 
                    \begin{equation*}
                        \left\{ W_j^{x,\delta} : x\in \Compact_1, \delta\in (0,\delta_1] \right\}\subset \CbinftySpace[B^n(1)][\Rn]
                    \end{equation*}
                    is a bounded set.
                \item \(\exists m\in \Nplus\), independent of \(x\) and \(\delta\), such that \(W_1^{x,\delta},\ldots, W_r^{x,\delta}\)
                    satisfy H\"ormander's condition of order \(m\) on \(B^n(1)\), and if \(\tau_{x,\delta}\)
                    is defined as in \eqref{Eqn::Subellip::Proof::Apriori::LowerBoundOfDet},
                    then
                    \begin{equation*}
                        \inf_{\substack{x\in \Compact_1\\\delta\in (0,\delta_1]}} \tau_{x,\delta}>0.
                    \end{equation*}
            \end{enumerate}
        \item\label{Item::MaxSub::Scaling::PullBackMeasure} \(\Phi_{x,\delta}^{*}\: d\measure = \measure(\MetricBall{x}{\delta}) \sigma_{x,\delta}(v)\: dv\), where \(\sigma_{x,\delta}\in \CbinftySpace[\Bno]\)
            is 
            a positive function satisfying 
            \begin{equation*}
                \inf_{x\in \Compact_1}\inf_{ \delta\in (0,\delta_1]} \inf_{ v\in B^n(1)} \sigma_{x,\delta}(v)>0,
            \end{equation*}
            and \(\left\{ \sigma_{x,\delta} :  x\in \Compact_1, \delta\in (0,\delta_1]\right\}\subset \CbinftySpace[B^n(1)]\)
            is a bounded set.
    \end{enumerate}
\end{theorem}
\begin{proof}[Comments on the proof]
    The statement here is similar to the more general  \cite[Theorem 3.5.1]{StreetMaximalSubellipticity}; however we have used
    \(\Phi_{x,\xi_3\delta/2}\)  from that theorem (where \(\xi_3\) is in that theorem), in place of \(\Phi_{x,\delta}\). This establishes \ref{Item::MaxSub::Scaling::Containments}
    using  \cite[Theorem 3.5.1(b),(e)]{StreetMaximalSubellipticity}.
\end{proof}

        \subsubsection{Completion of the proof of Theorem \texorpdfstring{\ref{Thm::SubMainThm::MainThm}}{\ref*{Thm::SubMainThm::MainThm}}}
        \label{Section::MaxSub::CompleteProof}
        The proof of Theorem \ref{Thm::SubMainThm::MainThm}
is mostly a matter of using
Corollary \ref{Cor::MaxSub::UnitScale::UnitScaleHypoEstimates}
and Theorem \ref{Thm::MaxSub::Scaling::NSWScaling}
to verify the assumptions of Theorem \ref{Thm::Results::MainThm}.
Fix a semi-group \(T(t)\) with generator \((A,\Domain[A])\) as in Theorem \ref{Thm::SubMainThm::MainThm}.

By replacing \(\opL\) with \(\opL+\omega_0\) (and \(T(t)\) with \(e^{-\omega_0 t}T(t)\)), it suffices to prove the result
with \(\omega_0=0\). Here, we have used that the constant \(C\) in Theorem \ref{Thm::SubMainThm::MainThm}
can depend on \(\omega_0\), and we have also used that \(\left( \metric[x][y]+t^{1/2\kappa} \right)\wedge \delta_0\leq \delta_0\leq 1\)
(so that \(\left( \left( \metric[x][y]+t^{1/2\kappa} \right)\wedge \delta_0 \right)^{-2\kappa j}\) is increasing in \(j\)).
This is in contrast to Theorem \ref{Thm::Results::MainThm}, where the constants did not depend on \(\omega_0\),
and \(\omega_0\) was instead kept track of in the right-hand side of \eqref{Eqn::Results::MainGaussianBounds}.

Proposition \ref{Prop::SubProof::QualSubellip::ExistsHeatKernel} establishes the existence 
of a unique smooth function \(K_t(x,y)\in \CinftySpace[(0,\infty)\times \Manifold\times\Manifold]\)
satisfying \eqref{Eqn::SubMainThm::DefineKt}, and therefore, our goal is to establish \eqref{Eqn::SubMainThm::MainBound}.

Fix \(\Compact\Subset \Manifold\) compact, for which we want to prove Theorem \ref{Thm::SubMainThm::MainThm}
and fix \(\psi\in \CzinftySpace[\Manifold]\) with \(\psi=1\) on a neighborhood of \(\Compact\).
For each \(\alpha\) and 
\(\beta\),
we will apply Theorem \ref{Thm::Results::MainThm} with
\(p_0=2\),
\(\NSubsetx=\NSubsety=\Compact\), \(\omega_0=0\), \(X=\psi W^{\alpha}\), \(Y=\psi W^{\beta}\),
\(S_X(\delta)=\epsilon_\alpha \delta^{\degWd(\alpha)}\), and \(S_Y(\delta)=\epsilon_\beta \delta^{\degWd(\beta)}\),
where \(\epsilon_\alpha,\epsilon_\beta>0\) are small numbers to be chosen later
and \(\delta_0\in (0,1]\) is a small number to be chosen later.
In this application of Theorem \ref{Thm::Results::MainThm}, admissible constants will not depend on \(\alpha\) or \(\beta\),
which shows that \(c>0\) does not depend on \(\alpha\) or \(\beta\) (see 
 Remarks \ref{Rmk::Results::cIsAdmissible} and \ref{Rmk::Subellip::MaxSub::ChoicesInApplication}).

 \begin{remark}\label{Rmk::SubProof::Completion::SameTopology}
   Throughout, we use the fact that the topology on \(\Manifold\) induced by \(\metric\)
   is the same as the topology on \(\Manifold\) as a manifold. See
   \cite[Lemma 3.1.7]{StreetMaximalSubellipticity}.
 \end{remark}

 We turn to verifying the assumptions of Theorem \ref{Thm::Results::MainThm}
 with the above choices.

 \begin{lemma}\label{Lemma::SubProof::Completion::XContinuous}
    For every \(\alpha\), \(\psi W^{\alpha}:\Domain[A^\infty]\rightarrow \LpSpace{2}[\Manifold][\mu]\)
    is continuous.
 \end{lemma}
 \begin{proof}
    Corollary \ref{Cor::SubProof::QualSubellip::DomainEmbedsInSmooth} shows
    \(\Domain[A^\infty]\hookrightarrow \CinftySpace[\Manifold]\) is continuous, and clearly
    \(\psi W^{\alpha}:\CinftySpace[\Manifold]\rightarrow \LpSpace{2}[\Manifold][\mu]\) is continuous.
 \end{proof}

 Fix \(\delta_2>0\) so small
 \begin{equation}\label{Eqn::SubProof::CompleteProof::DefineCompact1}
    \Compact_1:=\overline{\bigcup_{x\in \Compact} \MetricBall{x}{\delta_2}}
 \end{equation}
 is compact. This is always possible, since \(\Compact\) is compact 
 and the metric topology induced by \(\rho\) equals the usual topology on \(\Manifold\)
 as a manifold--see
 Remark \ref{Rmk::SubProof::Completion::SameTopology}.
 
 Fix \(\Omega\Subset \Manifold\) open and relatively compact with \(\Compact_1\Subset \Omega\)
 and \(\psi\in \CzinftySpace[\Omega]\).
 Let 
 \(\delta_1\in (0,1]\) and \(\Phi_{x,\delta}\) be as in
 Theorem \ref{Thm::MaxSub::Scaling::NSWScaling} 
 with this choice of \(\Compact_1\) and \(\Omega\)
 and set \(\delta_0:=\min\{\delta_1,\delta_2\}\in (0,1]\).
 
 \begin{lemma}\label{Lemma::SubProof::CompleteProof::SimplierContainment}
    For \(x_0\in \Compact\) and \(\delta\in (0,\delta_0]\), we have \(\MetricBall{x_0}{\delta}\subseteq \Compact_1\).
 \end{lemma}
 \begin{proof}
    Since \(\delta_0\leq \delta_2\), this follows immediately from the definition of \(\Compact_1\);
    see \eqref{Eqn::SubProof::CompleteProof::DefineCompact1}.
 \end{proof}

 \begin{remark}\label{Rmk::SubProof::CompleteProof::SimplierAssumptions}
    We use Lemma \ref{Lemma::SubProof::CompleteProof::SimplierContainment}
    to simplify Assumptions \ref{Assumption::LocalDoubling} and \ref{Assumption::HypoellipticityII};
    see Remark \ref{Rmk::Assumption::RemoveQuantifers}.
    In fact, to establish Assumption \ref{Assumption::LocalDoubling},
    Lemma \ref{Lemma::SubProof::CompleteProof::SimplierContainment}
    shows that it suffices to prove
    \(\mu(\MetricBall{x}{2\delta})\leq D_1\measure(\MetricBall{x}{\delta})<\infty\),
    \(\forall x\in \Compact_1\), \(\delta\in (0,\delta_0]\).
    Similarly, to establish Assumption \ref{Assumption::HypoellipticityII},
    it suffices to establish \eqref{Eqn::Assump::Hyp::HypoellipticityIIBound}
    \(\forall x\in \Compact_1\), \(\delta\in (0,\delta_0)\)
    (and the same for \(A\) replaced with \(A^{*}\)).
 \end{remark}

 \begin{lemma}\label{Lemma::SubProof::CompleteProof::CzinfyInDomainAInfty}
    \(\CzinftySpace[\Manifold]\subseteq \Domain[A^\infty]\).
 \end{lemma}
 \begin{proof}
    We prove by induction \(\CzinftySpace[\Manifold]\subseteq \Domain[A^j]\) for all \(j\geq 1\).
    The base case, \(\CzinftySpace[\Manifold]\subseteq \Domain[A]\), is assumed in Theorem \ref{Thm::SubMainThm::MainThm}.
    Suppose, for induction, \(\CzinftySpace[\Manifold]\subseteq \Domain[A^j]\) for some \(j\geq 1\).
    Then, since \(A\CzinftySpace[\Manifold]=-\opL\CzinftySpace[\Manifold]\subseteq \CzinftySpace[\Manifold]\subseteq \Domain[A^j]\),
    we see \(\CzinftySpace[\Manifold]\subseteq \Domain[A^{j+1}]\), completing the inductive step and the proof.
 \end{proof}

 \begin{lemma}\label{Lemma::SubProof::CompleteProof::Assumption1Holds}
    Assumption \ref{Assumption::Locality} holds.
    In fact, \(\forall x_0\in \Manifold\), \(\forall \delta>0\), the \(\LpSpace{2}[\Manifold][\mu]\)
    closure of 
    \begin{equation}\label{Eqn::SubProof::CompleteProof::Assumption1Holds::Set}
        \left\{ \phi\in \DAinfty : A^j\phi\big|_{\MetricBall{x_0}{\delta}}=0, \forall j\geq 0 \right\}
    \end{equation}
    contains \(\LpSpace{2}[\Manifold\setminus \MetricBallClosure{x_0}{\delta}][\measure]\).
    Here, \(\MetricBallClosure{x_0}{\delta}\) is the closure of \(\MetricBall{x_0}{\delta}\)
    with respect to the metric \(\metric\).
    The same result holds for \((A^{*},\Domain[A^{*}])\) in place of \((A,\Domain[A])\).
 \end{lemma}
 \begin{proof}
    Since the metric topology induced by \(\metric\) equals the topology on \(\Manifold\)
    as a manifold (see
 Remark \ref{Rmk::SubProof::Completion::SameTopology}),
    \(\MetricBallClosure{x_0}{\delta}\) is closed as a subset of the manifold.

    We claim \(\CzinftySpace[\Manifold\setminus \MetricBallClosure{x_0}{\delta}]\)
    is a subset of \eqref{Eqn::SubProof::CompleteProof::Assumption1Holds::Set}.
    Indeed, if \(\phi\in \CzinftySpace[\Manifold\setminus \MetricBallClosure{x_0}{\delta}]\),
    then \(\phi\in \Domain[A^\infty]\) by Lemma \ref{Lemma::SubProof::CompleteProof::CzinfyInDomainAInfty},
    and \(A^j\phi\big|_{\MetricBall{x_0}{\delta}} = (-\opL)^j\phi\big|_{\MetricBall{x_0}{\delta}}=0\).
    We conclude \(\phi\) is in \eqref{Eqn::SubProof::CompleteProof::Assumption1Holds::Set} as desired.

    Since the \(\LpSpace{2}[\Manifold][\mu]\) closure of \(\CzinftySpace[\Manifold\setminus \MetricBallClosure{x_0}{\delta}]\)
    contains \(\LpSpace{2}[\Manifold\setminus \MetricBallClosure{x_0}{\delta}][\measure]\),
    the result for \((A,\Domain[A])\) follows.

    Since our assumptions are symmetric in \(A\) and \(A^{*}\) (see Proposition \ref{Prop::Subellip::MaxSub::SymmetricAssumps}),
    the same result holds for \(A^{*}\) in place of \(A\).
 \end{proof}

 \begin{lemma}\label{Lemma::SubProof::Completion::PositiveMeasure}
    Assumption \ref{Assumption::LocalPositivity} holds.
    In fact, \(\forall \EmptySet\ne U\subseteq \Manifold\) open, \(\mu(U)>0\).
 \end{lemma}
 \begin{proof}
    This follows from the fact that \(\mu\) is a strictly positive density. See, also, Remark \ref{Rmk::SubProof::Completion::SameTopology}.
 \end{proof}

 \begin{lemma}\label{Lemma::SubProof::CompleteProof::Assumption3Holds}
    Assumption \ref{Assumption::LocalDoubling} holds.
    In fact, the following stronger statement is true.
    \(\exists D_1\geq 1\), \(\forall x\in \Compact_1\), \(\forall \delta\in (0,\delta_0]\),
    \begin{equation}\label{Eqn:::SubProof::CompleteProof::Doubling::MainDoubling}
        \mu\left( \MetricBall{x}{2\delta} \right)\leq D_1 \mu\left( \MetricBall{x}{\delta} \right).
    \end{equation}
 \end{lemma}
 \begin{proof}
    \eqref{Eqn:::SubProof::CompleteProof::Doubling::MainDoubling}
    is Theorem \ref{Thm::MaxSub::Scaling::NSWScaling} \ref{Item::MaxSub::Scaling::Doubling}.
    That \eqref{Eqn:::SubProof::CompleteProof::Doubling::MainDoubling} implies Assumption \ref{Assumption::LocalDoubling} follows from
   Remark \ref{Rmk::SubProof::CompleteProof::SimplierAssumptions}.
 \end{proof}

 \begin{lemma}\label{Lemma::SubProof::CompleteProof::HeatFuncsAreSmooth}
   \(\forall (x,\delta)\in \Manifold\times (0,\infty)\) and \(I\subseteq \R\) an open interval,
   if \(u:I\rightarrow \Domain[A^\infty]\) is an \((A,x,\delta)\)-heat function, then
   \(u\big|_{I\times \MetricBall{x}{\delta}}\in \CinftySpace[I\times \MetricBall{x}{\delta}]\).
 \end{lemma}
 \begin{proof}
   Fix \((t_0,y_0)\in I\times \MetricBall{x}{\delta}\) and let \(\phi_1,\phi_2\in \CzinftySpace[I\times \MetricBall{x}{\delta}]\)
   be such that \(\phi_1\prec \phi_2\) and \(\phi_1=1\) on a neighborhood of \((t_0,y_0)\).
   Consider, \(\phi_2(\delta^{-2\kappa}t,y)(\partial_t+\opL)u(\delta^{-2\kappa}t,y)=\phi_2(\delta^{-2\kappa}t,y)(\partial_t-A)u(\delta^{-2\kappa}t,y)=0\),
   and \(\phi_2(\delta^{-2\kappa}t,y)u(\delta^{-2\kappa}t,y)\in \LpSpace{2}[I\times \Manifold]\) by the definition of heat functions
   (see Definition \ref{Defn::Assump::Hypo::HeatFunction} \ref{Item::Assump::Hypo::HeatFunction::LocalSmooth}).
   Proposition \ref{Prop::SubProof::QualSubellip::QualSubellip} implies
   \(\phi_1(\delta^{-2\kappa}t,y)u(\delta^{-2\kappa}t,y)\in \HsSpace{s}[\R\times \Manifold]\), \(\forall s\in \R\).
   The Sobolev imbedding theorem then implies \(\phi_1(t,y)u(t,y)\in \CinftySpace[\R\times \Manifold]\).
   This shows \(u(t,y)\) is smooth on a neighborhood of \((t_0,y_0)\).
   As \((t_0,y_0)\in I\times \MetricBall{x}{\delta}\) was arbitrary, this completes the proof.
 \end{proof}

 \begin{proposition}\label{Prop::SubProof::CompleteProof::ScaledEstimates}
   Let \(a_1:=\min\{\xi_0,1/4\}\in (0,1/4]\), where \(\xi_0\in (0,1)\) is as in Theorem \ref{Thm::MaxSub::Scaling::NSWScaling}.
   \(\exists C_1\geq 0\), \(\forall \alpha\), \(\exists \epsilon_\alpha'> 0\),
   \(\forall (x,\delta)\in \Compact_1\times (0,\delta_0]\), \(\forall u:(-1/2,1/2)\rightarrow \Domain[A^\infty]\)
   an \((A,x,\delta)\)-heat function, we have
   \begin{enumerate}[(i)]
      \item\label{Item::SubProof::CompleteProof::ScaledEstimates::Smooth} \(\psi W^{\alpha}u(t,x)\big|_{(-1/2,1/2)\times \MetricBall{x}{\delta}}
      \in \CinftySpace[(-1/2,1/2)\times\MetricBall{x}{\delta}]\).
      \item\label{Item::SubProof::CompleteProof::ScaledEstimates::Walpha} \begin{equation*}
         \sup_{\substack{t\in (-a_1,a_1)\\ y\in \MetricBall{x}{a_1\delta}}} \left| \epsilon_\alpha' \psi \delta^{\degWd(\alpha)} W^{\alpha} u(t,y) \right|
         \leq \measure\left( \MetricBall{x}{\delta} \right)^{-1/2} \LpNorm{u}{2}[(-1/2,1/2)\times \MetricBall{x}{\delta/2}].
      \end{equation*}
      \item\label{Item::SubProof::CompleteProof::ScaledEstimates::partialt} \(
         \LpNorm{\partial_t u}{2}[(-a_1,a_1)\times \MetricBall{x}{a_1\delta}]
         \leq C_1 \LpNorm{u}{2}[(-1/2,1/2)\times \MetricBall{x}{\delta/2}].
      \)
   \end{enumerate}
 \end{proposition}

 We separate the proof of Proposition \ref{Prop::SubProof::CompleteProof::ScaledEstimates}
 into several steps.

 \begin{proof}[Proof of Proposition \ref{Prop::SubProof::CompleteProof::ScaledEstimates} \ref{Item::SubProof::CompleteProof::ScaledEstimates::Smooth}]
   This follows immediately from Lemma \ref{Lemma::SubProof::CompleteProof::HeatFuncsAreSmooth}.
 \end{proof}

 The proofs of Proposition \ref{Prop::SubProof::CompleteProof::ScaledEstimates}
\ref{Item::SubProof::CompleteProof::ScaledEstimates::Walpha} and \ref{Item::SubProof::CompleteProof::ScaledEstimates::partialt}
require some initial setup.
For \(x\in \Compact\), \(\delta\in (0,\delta_0]\), let
\(W_j^{x,\delta}:=\Phi_{x,\delta}^{*} \delta^{\Wd_j}W_j\) be as in Theorem \ref{Thm::MaxSub::Scaling::NSWScaling},
and let \(\sigma_{x,\delta}\) be as in Theorem \ref{Thm::MaxSub::Scaling::NSWScaling} \ref{Item::MaxSub::Scaling::PullBackMeasure}
so that \(\sigma_{x,\delta}\) is uniformly bounded below and
\(\left\{ \sigma_{x,\delta} : x\in \Compact, \delta\in (0,\delta_0] \right\}\subset \CbinftySpace[B^n(1)]\) is a bounded set.
Set \(a_{\alpha,\beta}^{x,\delta}:=\delta^{2\kappa-\degWd(\alpha)-\degWd(\beta)}a_{\alpha,\beta}\circ \Phi_{x,\delta}\in \CinftySpace[B^n(1)]\).

Let \(\delta^{\Wd}W=\left( \delta^{\Wd_1}W_1,\ldots, \delta^{\Wd_r}W_r\right)\).
Using \eqref{Eqn::SubMainThem::FormulaForL}, we have
\begin{equation}\label{Eqn::SubProof::CompleteProof::deltaopL}
   \delta^{2\kappa}\opL = \sum_{\degWd(\alpha),\degWd(\beta)\leq \kappa} \left( \left( \delta^{\Wd}W \right)^{\alpha} \right)^{*} \delta^{2\kappa-\degWd(\alpha)-\degWd(\beta)} a_{\alpha,\beta} \left( \delta^{\Wd}W \right)^{\beta}.
\end{equation}
Let \(\opLxdelta:=\Phi_{x,\delta}^{*} \delta^{2\kappa}\opL \left( \Phi_{x,\delta} \right)_{*}\) (see Remark \ref{Rmk::SubProof::QualSubellip::PullBackOp}).
In light of \eqref{Eqn::SubProof::QualSubellip::opLPhi}, using \eqref{Eqn::SubProof::CompleteProof::deltaopL}, 
the above definitions, and Remark \ref{Rmk::SubProof::QualSubellip::PullBackOp}, we have
\begin{equation}\label{Eqn::SubProof::CompleteProof::opLxdelta}
\begin{split}
    \opLxdelta
    &=\sum_{\degWd(\alpha),\degWd(\beta)\leq \kappa} \measure(\MetricBall{x}{\delta})^{-1} \sigma_{x,\delta}^{-1} \left( \left( W^{x,\delta} \right)^{\alpha} \right)^{*} \measure(\MetricBall{x}{\delta}) \sigma_{x,\delta} a_{\alpha,\beta}^{x,\delta} \left( W^{x,\delta} \right)^{\beta}
    \\&=\sum_{\degWd(\alpha),\degWd(\beta)\leq \kappa}  \sigma_{x,\delta}^{-1} \left( \left( W^{x,\delta} \right)^{\alpha} \right)^{*}  \sigma_{x,\delta} a_{\alpha,\beta}^{x,\delta} \left( W^{x,\delta} \right)^{\beta},
\end{split}
\end{equation}
where we have used \(\measure(\MetricBall{x}{\delta})\) is a constant since \(x\) is fixed.

To prove  Proposition \ref{Prop::SubProof::CompleteProof::ScaledEstimates}
\ref{Item::SubProof::CompleteProof::ScaledEstimates::Walpha} and \ref{Item::SubProof::CompleteProof::ScaledEstimates::partialt}
we will apply 
Corollary \ref{Cor::MaxSub::UnitScale::UnitScaleHypoEstimates}
with \(\opL\) replaced by \(\opLxdelta\); note that
\eqref{Eqn::SubProof::CompleteProof::opLxdelta} shows
\(\opLxdelta\) is of the form \eqref{Eqn::Subellip::Proof::Apriori::opLFormula}.
Importantly, we need that the estimates obtained from 
Corollary \ref{Cor::MaxSub::UnitScale::UnitScaleHypoEstimates}
are uniform in \((x,\delta)\).
Thus, we need to establish that the assumptions of 
Proposition \ref{Prop::Subellip::Proof::Apriori}
hold uniformly in \((x,\delta)\); see the statement of Proposition \ref{Prop::Subellip::Proof::Apriori} where the dependence
of the constant \(C_{s,\phi_1,\phi_2}\) is spelled out explicitly.

\begin{lemma}\label{Lemma::SubProof::CompleteProof::aAlphaBetaxdeltaBounded}
   \(\left\{ a_{\alpha,\beta}^{x,\delta} : (x,\delta)\in \Compact_1\times (0,\delta_0], \degWd(\alpha),\degWd(\beta)\leq \kappa \right\}\subset \CbinftySpace[B^n(1)]\)
   is a bounded set.
\end{lemma}
\begin{proof}
   Since \(\delta_0\in (0,1)\), it suffices to show
   \begin{equation}\label{Eqn::SubProof::CompleteProof::aAlphaBetaxdeltaBounded::BoundCjNorm}
      \sup_{\substack{x\in \Compact_1\\ \delta\in (0,\delta_0]
      }} \BCjNorm{a_{\alpha,\beta}\circ \Phi_{x,\delta}}{L}[B^n(1)]<\infty,\quad \forall L\in \N.
   \end{equation}
   In light of Theorem \ref{Thm::MaxSub::Scaling::NSWScaling} \ref{Item::MaxSub::Scaling::UniformHormander},
   \eqref{Eqn::SubProof::CompleteProof::aAlphaBetaxdeltaBounded::BoundCjNorm} is equivalent to
   \begin{equation}\label{Eqn::SubProof::CompleteProof::aAlphaBetaxdeltaBounded::BoundCjWNorm}
      \sup_{\substack{x\in \Compact_1\\ \delta\in (0,\delta_0]
      }}\sum_{|\gamma|\leq N} \sup_{u\in B^n(1)} \left| \left( W^{x,\delta} \right)^{\gamma} a_{\alpha,\beta}\circ \Phi_{x,\delta} (u) \right| <\infty,\quad \forall N\in \N.
   \end{equation}
   Since \(W_j^{x,\delta}:=\Phi_{x,\delta}^{*} \delta^{\Wd_j}W_j\), we have
   \begin{equation}\label{Eqn::SubProof::CompleteProof::aAlphaBetaxdeltaBounded::Tmp1}
      \left( W^{x,\delta} \right)^{\gamma} a_{\alpha,\beta}\circ \Phi_{x,\delta}
      = \left( \delta^{\degWd(\gamma)} W^{\gamma} a_{\alpha,\beta} \right)\circ \Phi_{x,\delta}.
   \end{equation}
   Using \eqref{Eqn::SubProof::CompleteProof::aAlphaBetaxdeltaBounded::Tmp1} and the fact that \(\delta_0\in (0,1)\),
   we have
   \eqref{Eqn::SubProof::CompleteProof::aAlphaBetaxdeltaBounded::BoundCjWNorm} 
   is implied by
   \begin{equation}\label{Eqn::SubProof::CompleteProof::aAlphaBetaxdeltaBounded::Tmp2}
      \sup_{\substack{x\in \Compact_1\\ \delta\in (0,\delta_0]
      }}\sum_{|\gamma|\leq N} \sup_{y\in \Phi_{x,\delta}(B^n(1))} \left| W^{\gamma} a_{\alpha,\beta}(y) \right| <\infty,\quad \forall N\in \N.
   \end{equation}
   But \(\Phi_{x,\delta}(B^n(1))\subseteq \Omega\)
   (see Theorem \ref{Thm::MaxSub::Scaling::NSWScaling})
   and \(\overline{\Omega}\Subset \Manifold\) is compact,
   so we have
   \begin{equation*}
   \begin{split}
       &\sup_{\substack{x\in \Compact_1\\ \delta\in (0,\delta_0]
      }}\sum_{|\gamma|\leq N} \sup_{y\in \Phi_{x,\delta}(B^n(1))} \left| W^{\gamma} a_{\alpha,\beta}(y) \right|
      \leq \sum_{|\gamma|\leq N} \sup_{y\in \overline{\Omega}} \left| W^{\gamma} a_{\alpha,\beta}(y) \right|
      <\infty,\quad \forall N\in \N,
   \end{split}
   \end{equation*}
   establishing \eqref{Eqn::SubProof::CompleteProof::aAlphaBetaxdeltaBounded::Tmp2} and 
   completing the proof.
\end{proof}

\begin{lemma}\label{Lemma::SubProof::CompleteProof::EqualL2Norms}
   \(\forall x\in \Compact_1\), \(\forall \delta\in (0,\delta_0]\), \(\forall f\in \LpSpace{2}[B^n(1)]\),
   \begin{equation}\label{Eqn::SubProof::CompleteProof::EqualL2Norms}
      \measure\left( \MetricBall{x}{\delta} \right)^{1/2} \BLpNorm{f}{2}[B^n(1)]
      \approx \measure\left( \MetricBall{x}{\delta} \right)^{1/2} \BLpNorm{f}{2}[B^n(1)][\sigma_{x,\delta}(v)\: dv]
      =\BLpNorm{ f\circ \Phi_{x,\delta}^{-1}}{2}[\Phi_{x,\delta}(B^n(1))][\measure],
   \end{equation}
   where the implicit constants do not depend on \(x\), \(\delta\), or \(f\).
\end{lemma}
\begin{proof}
   By Theorem \ref{Thm::MaxSub::Scaling::NSWScaling} \ref{Item::MaxSub::Scaling::PullBackMeasure},
   \(\sigma_{x,\delta}\approx 1\), and the \(\approx\) in \eqref{Eqn::SubProof::CompleteProof::EqualL2Norms} follows.
   The equality in \eqref{Eqn::SubProof::CompleteProof::EqualL2Norms}
   follows from \(\Phi_{x,\delta}^{*}\: d\measure = \measure(\MetricBall{x}{\delta}) \sigma_{x,\delta}(v)\: dv\)
   (see Theorem \ref{Thm::MaxSub::Scaling::NSWScaling} \ref{Item::MaxSub::Scaling::PullBackMeasure}).
\end{proof}

\begin{lemma}\label{Lemma::SubProof::CompleteProof::CheckScaledMaximalSub}
   Assumption \ref{Assumption::Subellip::Proof::Apriori} holds uniformly in \(x\) and \(\delta\).
   More precisely, \(\exists \Ct_1,\Ct_2\geq 0\), \(\forall f\in \CzinftySpace[B^n(1)]\),
   \(\forall x\in \Compact_1\), \(\forall \delta\in (0,\delta_0]\),
   \begin{equation*}
      \sum_{j=1}^r \BLpNorm{ \left( W_j^{x,\delta} \right)^{\kappa/\Wd_j} f}{2}[B^n(1)]^2
      \leq
      \Ct_1 \Real \Ltip*{f}{\sigma_{x,\delta}\opLxdelta f}[B^n(1)]
      +\Ct_2 \BLpNorm{f}{2}[B^n(1)]^2.
   \end{equation*}
\end{lemma}
\begin{proof}
   Let \(f\in \CzinftySpace[B^n(1)]\), \(x\in \Compact\), and \(\delta\in (0,\delta_0]\). 
   Note that \(f\circ \Phi_{x,\delta}\in \CzinftySpace[\Omega]\) (see Theorem \ref{Thm::MaxSub::Scaling::NSWScaling}).
   Applying Assumption \ref{Assump::SubMainTheorem::MaxSubTypeIII} with \(f\) replaced by
   \(f\circ \Phi_{x,\delta}\), we see
   \begin{equation}\label{Eqn::SubProof::CompleteProof::CheckScaledMaximalSub::Tmp1}
         \sum_{j=1}^r \BLpNorm{ W_j^{\kappa/\Wd_j} f\circ \Phi_{x,\delta}}{2}[\Manifold][\measure]^2
        \leq C_1 \Real \Ltip*{f\circ \Phi_{x,\delta}}{\opL f\circ \Phi_{x,\delta}}[\Manifold][\measure]
        +C_2 \BLpNorm{f\circ \Phi_{x,\delta}}{2}[\Manifold][\measure]^2.
   \end{equation}
   Multiplying both sides of \eqref{Eqn::SubProof::CompleteProof::CheckScaledMaximalSub::Tmp1}
   by \(\delta^{2\kappa}\) and using \(\delta\in (0,1]\), we see
   \begin{equation}\label{Eqn::SubProof::CompleteProof::CheckScaledMaximalSub::Tmp2}
         \sum_{j=1}^r \BLpNorm{ \left(\delta^{\Wd_j} W_j \right)^{\kappa/\Wd_j} f\circ \Phi_{x,\delta}}{2}[\Manifold][\measure]^2
        \leq C_1 \Real \Ltip*{f\circ \Phi_{x,\delta}}{\delta^{2\kappa}\opL f\circ \Phi_{x,\delta}}[\Manifold][\measure]
        +C_2 \BLpNorm{f\circ \Phi_{x,\delta}}{2}[\Manifold][\measure]^2.
   \end{equation}
   Using \(\delta^{\Wd_j} W_j f\circ \Phi_{x,\delta}= \left( W_j^{x,\delta} f\right)\circ \Phi_{x,\delta}\),
   changing variables \(u=\Phi_{x,\delta}(v)\), using \(\Phi_{x,\delta}^{*}\: d\measure = \measure(\MetricBall{x}{\delta}) \sigma_{x,\delta}(v)\: dv\)
   (see Theorem \ref{Thm::MaxSub::Scaling::NSWScaling} \ref{Item::MaxSub::Scaling::PullBackMeasure}), 
    Remark \ref{Rmk::SubProof::QualSubellip::PullBackOp} with \(\sigma\) replaced by
   \(\measure\left( \MetricBall{x}{\delta} \right)\sigma_{x,\delta}\), and \eqref{Eqn::SubProof::CompleteProof::deltaopL},
   \eqref{Eqn::SubProof::CompleteProof::CheckScaledMaximalSub::Tmp2} shows
   \begin{equation}\label{Eqn::SubProof::CompleteProof::CheckScaledMaximalSub::Tmp3}
        \begin{split}
         &\sum_{j=1}^r \measure\left( \MetricBall{x}{\delta} \right)\BLpNorm{ \left( W_j^{x,\delta} \right)^{\kappa/\Wd_j} f}{2}[B^n(1)][\sigma_{x,\delta}\: dv]^2
         \\&\leq C_1 \Real \measure\left( \MetricBall{x}{\delta} \right)\Ltip*{f}{\sigma_{x,\delta}\opLxdelta f}[B^n(1)]
         +C_2 \measure\left( \MetricBall{x}{\delta} \right)\BLpNorm{f}{2}[B^n(1)][\sigma_{x,\delta}\: dv]^2.
        \end{split}
   \end{equation}
   By Lemma \ref{Lemma::SubProof::CompleteProof::EqualL2Norms},
   \eqref{Eqn::SubProof::CompleteProof::CheckScaledMaximalSub::Tmp3} implies
   \begin{equation}\label{Eqn::SubProof::CompleteProof::CheckScaledMaximalSub::Tmp4}
        \begin{split}
         &\sum_{j=1}^r \measure\left( \MetricBall{x}{\delta} \right)\BLpNorm{ \left( W_j^{x,\delta} \right)^{\kappa/\Wd_j} f}{2}[B^n(1)]^2
         \\&\leq \Ct_1 \Real \measure\left( \MetricBall{x}{\delta} \right)\Ltip*{f}{\sigma_{x,\delta}\opLxdelta f}[B^n(1)]
         +\Ct_2 \measure\left( \MetricBall{x}{\delta} \right)\BLpNorm{f}{2}[B^n(1)]^2.
        \end{split}
   \end{equation}
   Dividing \(\measure\left( \MetricBall{x}{\delta} \right)\) from both sides of 
   \eqref{Eqn::SubProof::CompleteProof::CheckScaledMaximalSub::Tmp4} completes the proof;
   here we have used \(\measure\left( \MetricBall{x}{\delta} \right)> 0\)
   by Lemma \ref{Lemma::SubProof::Completion::PositiveMeasure}.
\end{proof}

\begin{lemma}\label{Lemma::SubProof::CompleteProof::EstiamtesAtScale}
   There exists constants \(C_1\geq 0\) and \(\forall \alpha\), \(C_\alpha\geq 0\) such that 
   \(\forall x\in \Compact_1\), \(\forall \delta\in (0,\delta_0]\)
   \(\forall \uh(t,v)\in \Distributions[(-1/2,1/2)\times \Bno]\) satisfying \(\left( \partial_t+\opLxdelta \right)\uh=0\),
   we have
   \begin{equation}\label{Eqn::SubProof::CompleteProof::EstiamtesAtScale::partialtEstimate}
      \BLpNorm{\partial_t \uh}{2}[(-1/4,1/4)\times \Phi_{x,\delta}^{-1}\left( \MetricBall{x}{\xi_0 \delta} \right)][dt\times \sigma_{x,\delta}\: dv]
      \leq C_1 \BLpNorm{\uh}{2}[(-1/2,1/2)\times B^n(1)][dt\times \sigma_{x,\delta}\: dv],
   \end{equation}
   and
   \begin{equation}\label{Eqn::SubProof::CompleteProof::EstiamtesAtScale::WalphaEstimate}
      \sup_{\substack{v\in \Phi^{-1}_{x,\delta}\left( \MetricBall{x}{\xi_0\delta} \right) \\ t\in (-1/4,1/4)}}
      \left| \left( W^{x,\delta} \right)^{\alpha}\uh(t,v) \right|
      \leq C_\alpha \BLpNorm{\uh}{2}[(-1/2,1/2)\times B^n(1)][dt\times \sigma_{x,\delta}\: dv],
   \end{equation}
   where if the right-hand side of one of the above equations is finite, so is the corresponding left-hand side.
\end{lemma}
\begin{proof}
   We have shown that the assumptions of Corollary \ref{Cor::MaxSub::UnitScale::UnitScaleHypoEstimates}
   hold uniformly for \(x\in \Compact_1\) and \(\delta\in (0,\delta_0]\)
   (see Lemmas \ref{Lemma::SubProof::CompleteProof::aAlphaBetaxdeltaBounded} and \ref{Lemma::SubProof::CompleteProof::CheckScaledMaximalSub}
   and Theorem \ref{Thm::MaxSub::Scaling::NSWScaling}).
   Thus the constants from Corollary \ref{Cor::MaxSub::UnitScale::UnitScaleHypoEstimates}
   can be chosen independent of \(x\in \Compact_1\) and \(\delta\in (0,\delta_0]\)
   (see the statement of Proposition \ref{Prop::Subellip::Proof::Apriori} for a description
   of what these constants depend on).
   Corollary \ref{Cor::MaxSub::UnitScale::UnitScaleHypoEstimates}
   then applies 
   to establish \eqref{Eqn::SubProof::CompleteProof::EstiamtesAtScale::partialtEstimate}
   and \eqref{Eqn::SubProof::CompleteProof::EstiamtesAtScale::WalphaEstimate}
   but with \(\Phi_{x,\delta}^{-1}\left( \MetricBall{x}{\xi_0 \delta} \right)\) on the left-hand side
   of both equations replaced by \(B^n(1/2)\).

   Theorem \ref{Thm::MaxSub::Scaling::NSWScaling} \ref{Item::MaxSub::Scaling::Containments}
   shows \(\Phi_{x,\delta}^{-1}\left( \MetricBall{x}{\xi_0 \delta} \right)\subseteq B^n(1/2)\),
   which completes the proof.
\end{proof}

\begin{proof}[Proof of Proposition \ref{Prop::SubProof::CompleteProof::ScaledEstimates}
\ref{Item::SubProof::CompleteProof::ScaledEstimates::Walpha} and \ref{Item::SubProof::CompleteProof::ScaledEstimates::partialt}]
   For \(x\in \Compact_1\), \(\delta\in (0,\delta_0]\),
   let \(u:(-1/2,1/2)\rightarrow \Domain[A^\infty]\) be an \((A,x,\delta)\)-heat function
   as in the statement of the proposition. Set \(\uh(t,v):=u(t,\Phi_{x,\delta}(v))\).

   By the definition of \((A,x,\delta)\)-heat functions (see Definition \ref{Defn::Assump::Hypo::HeatFunction} \ref{Item::Assump::Hypo::HeatFunction::Derivatives}),
   and using \(A=-\opL\big|_{\Domain[A]}\), we have
   \begin{equation*}
      \left( \left( \partial_t+\delta^{2\kappa}\opL \right) u \right)\big|_{(-1/2,1/2)\times \MetricBall{x}{\delta}}=0.
   \end{equation*}
   Since \(\opLxdelta:=\Phi_{x,\delta}^{*} \delta^{2\kappa}\opL \left( \Phi_{x,\delta} \right)_{*}\),
   and the range of \(\Phi_{x,\delta}\) is contained in \(\MetricBall{x}{\delta}\)
   (see Theorem \ref{Thm::MaxSub::Scaling::NSWScaling} \ref{Item::MaxSub::Scaling::Containments})
   we have \((\partial_t+\opLxdelta)\uh=0\).
   Lemma \ref{Lemma::SubProof::CompleteProof::EstiamtesAtScale} applies to show
   \eqref{Eqn::SubProof::CompleteProof::EstiamtesAtScale::partialtEstimate}
   and \eqref{Eqn::SubProof::CompleteProof::EstiamtesAtScale::WalphaEstimate}
   hold.

   Multiplying both sides of \eqref{Eqn::SubProof::CompleteProof::EstiamtesAtScale::partialtEstimate}
   by \(\measure(\MetricBall{x}{\delta})^{1/2}\)
   and applying Lemma \ref{Lemma::SubProof::CompleteProof::EqualL2Norms} with \(f(\cdot)=\uh(t,\cdot)=u(t,\Phi_{x,\delta}(\cdot))\),
   we obtain
   \begin{equation}\label{Eqn::SubProof::CompleteProof::FinishProofOfScaledEst::partialtEstimate}
      \BLpNorm{\partial_t u}{2}[(-1/4,1/4)\times  \MetricBall{x}{\xi_0 \delta} ][dt\times d\measure]
      \leq C_1 \BLpNorm{u}{2}[(-1/2,1/2)\times \Phi_{x,\delta}\left( B^n(1) \right)][dt\times d\measure].
   \end{equation}
   Similarly from \eqref{Eqn::SubProof::CompleteProof::EstiamtesAtScale::WalphaEstimate}
   and using \(\left( W^{x,\delta} \right)^{\alpha}\uh(t,v)= \left( \delta^{\degWd(\alpha)}W^{\alpha}u \right)(t,\Phi_{x,\delta}(v))\),
   we obtain
   \begin{equation}\label{Eqn::SubProof::CompleteProof::FinishProofOfScaledEst::WalphaEstimate}
      \sup_{\substack{y\in \MetricBall{x}{\xi_0\delta}\\ t\in (-1/4,1/4)}}
      \left| \delta^{\degWd(\alpha)}W^{\alpha}u(t,y) \right|
      \leq C_\alpha \BLpNorm{u}{2}[(-1/2,1/2)\times \Phi_{x,\delta}\left( B^n(1) \right)][dt\times d\measure].
   \end{equation}
   Finally, using \(a_1=\min\{\xi_0,1/4\}\) and \(\Phi_{x,\delta}(B^n(1))\subseteq \MetricBall{x}{\delta/2}\)
   (see Theorem \ref{Thm::MaxSub::Scaling::NSWScaling} \ref{Item::MaxSub::Scaling::Containments}),
   \eqref{Eqn::SubProof::CompleteProof::FinishProofOfScaledEst::partialtEstimate}
   and \eqref{Eqn::SubProof::CompleteProof::FinishProofOfScaledEst::WalphaEstimate} imply
      \begin{equation}\label{Eqn::SubProof::CompleteProof::FinishProofOfScaledEst::partialtEstimate::2}
      \BLpNorm{\partial_t u}{2}[(-a_1,a_1)\times  \MetricBall{x}{a_1 \delta} ][dt\times d\measure]
      \leq C_1 \BLpNorm{u}{2}[(-1/2,1/2)\times \MetricBall{x}{\delta/2}][dt\times d\measure],
   \end{equation}
   \begin{equation}\label{Eqn::SubProof::CompleteProof::FinishProofOfScaledEst::WalphaEstimate::2}
      \sup_{\substack{y\in \MetricBall{x}{a_1\delta}\\ t\in (-a_1,a_1)}}
      \left| \delta^{\degWd(\alpha)}W^{\alpha}u(t,y) \right|
      \leq C_\alpha \measure\left( \MetricBall{x}{\delta} \right)^{-1/2}  \BLpNorm{u}{2}[(-1/2,1/2)\times \MetricBall{x}{\delta/2}][dt\times d\measure].
   \end{equation}
   \ref{Item::SubProof::CompleteProof::ScaledEstimates::partialt}
   follows from \eqref{Eqn::SubProof::CompleteProof::FinishProofOfScaledEst::partialtEstimate::2},
   and with \(\epsilon_\alpha'=1/C_\alpha\), \ref{Item::SubProof::CompleteProof::ScaledEstimates::Walpha}
   follows from \eqref{Eqn::SubProof::CompleteProof::FinishProofOfScaledEst::WalphaEstimate::2}.
\end{proof}

\begin{lemma}\label{Lemma::SubProof::CompleteProof::Assumptions4And6Hold}
   Assumptions \ref{Assumption::HypoellipticityI} and \ref{Assumption::HypoellipticityII}
   hold with \(a_1=a_2\) where \(a_1\) is as in Proposition \ref{Prop::SubProof::CompleteProof::ScaledEstimates},
   where in Assumptions \ref{Assumption::HypoellipticityI} we take \(X=\psi W^{\alpha}\),  \(Y=\psi W^{\beta}\),
    \(S_X(\delta)=\epsilon_\alpha \psi \delta^{\degWd(\alpha) }W^{\alpha}\), and
   \(S_Y(\delta)=\epsilon_\beta \psi \delta^{\degWd(\beta) }W^{\beta}\),
   where \(\epsilon_\alpha>0\) is small described in the proof.
\end{lemma}
\begin{proof}
   Recall, \(\omega_0=0\).  We use 
   Remark \ref{Rmk::SubProof::CompleteProof::SimplierAssumptions}
   to see that to establish Assumption \ref{Assumption::HypoellipticityII},
   it suffices to prove 
   \eqref{Eqn::Assump::Hyp::HypoellipticityIIBound}
   for \(x\in \Compact_1\)
   and \(\delta\in (0,\delta_0]\).  Moreover, since \(\NSubsetx=\Compact\subseteq \Compact_1\),
   it also suffices to establish Assumption \ref{Assumption::HypoellipticityI}
   for \(x\in \Compact_1\)
   and \(\delta\in (0,\delta_0]\).

   Proposition \ref{Prop::Subellip::MaxSub::SymmetricAssumps} shows that our assumptions
   are symmetric \((A,\Domain[A])\) and \((A^{*},\Domain[A^*])\).
   In particular, Proposition \ref{Prop::SubProof::CompleteProof::ScaledEstimates}
   holds with \(A\) replaced by \(A^{*}\) throughout.
   Let \(\epsilon_\alpha''\) be \(\epsilon_\alpha\) from Proposition \ref{Prop::SubProof::CompleteProof::ScaledEstimates}
   when \(A\) is replaced by \(A^{*}\). We take \(\epsilon_\alpha:=\min\left\{ \epsilon_\alpha',\epsilon_\alpha'' \right\}\).

   Keeping this symmetry in \(A\) and \(A^{*}\) in mind,
   Assumption \ref{Assumption::HypoellipticityI} \ref{Item::Assumption::HypoellipticityI::Qualitative}
   follows from  Proposition  \ref{Prop::SubProof::CompleteProof::ScaledEstimates} \ref{Item::SubProof::CompleteProof::ScaledEstimates::Smooth}.
   Assumption \ref{Assumption::HypoellipticityI} \ref{Item::Assumption::HypoellipticityI::Quantitative}
   follows from Proposition  \ref{Prop::SubProof::CompleteProof::ScaledEstimates} \ref{Item::SubProof::CompleteProof::ScaledEstimates::Walpha}.
   For Assumption \ref{Assumption::HypoellipticityI} \ref{Item::Assumption::HypoellipticityI::CommuteDerivs}
   recall that
   \(A=-\opL\big|_{\Domain[A]}\), and therefore all operators in Assumption \ref{Assumption::HypoellipticityI} \ref{Item::Assumption::HypoellipticityI::CommuteDerivs}
   can be viewed as derivatives in the sense of distributions. Since \(\psi W^{\alpha}\) commutes with \(\partial_t\),
   Assumption \ref{Assumption::HypoellipticityI} \ref{Item::Assumption::HypoellipticityI::CommuteDerivs} follows
   (see, also, Remark \ref{Rmk::Assumption::Hypo::XandpartialtCommute}).

   Assumption \ref{Assumption::HypoellipticityII} follows from Proposition  \ref{Prop::SubProof::CompleteProof::ScaledEstimates} \ref{Item::SubProof::CompleteProof::ScaledEstimates::partialt}.
\end{proof}

\begin{lemma}\label{Lemma::SubProof::CompleteProof::Assumptions5Holds}
   Assumption \ref{Assumption::DoublingOfSXandSY} holds.
   Namely, let \(S_X(\delta):=\epsilon_{\alpha}\delta^{\degWd(\alpha)}\), 
   where \(\epsilon_{\alpha}>0\) is as in Lemma \ref{Lemma::SubProof::CompleteProof::Assumptions4And6Hold}.
   Then,
   \(S_X(2\delta)\leq 2^{\degWd(\alpha)}S_X(\delta)\).
\end{lemma}
\begin{proof}
   This is clear.
\end{proof}

\begin{proof}[Completion of the proof of Theorem \ref{Thm::SubMainThm::MainThm}]
   Recall, we have reduced to the case \(\omega_0=0\).

   Proposition \ref{Prop::SubProof::QualSubellip::ExistsHeatKernel} establishes the existence 
   of a unique smooth function \(K_t(x,y)\in \CinftySpace[(0,\infty)\times \Manifold\times\Manifold]\)
   satisfying \eqref{Eqn::SubMainThm::DefineKt}, and therefore, all that remains is to establish \eqref{Eqn::SubMainThm::MainBound}.
   We do this by applying Theorem \ref{Thm::Results::MainThm}
   with the choices described in Remark \ref{Rmk::Subellip::MaxSub::ChoicesInApplication}.

   In the above lemmas (especially Lemmas 
   \ref{Lemma::SubProof::Completion::XContinuous},
   \ref{Lemma::SubProof::CompleteProof::Assumption1Holds},
   \ref{Lemma::SubProof::Completion::PositiveMeasure},
   \ref{Lemma::SubProof::CompleteProof::Assumption3Holds},
   \ref{Lemma::SubProof::CompleteProof::Assumptions4And6Hold},
   and \ref{Lemma::SubProof::CompleteProof::Assumptions5Holds})
   show that the assumptions of 
   Theorem \ref{Thm::Results::MainThm} hold with these choices.
   Importantly, \(a_1\) as chosen in Proposition \ref{Prop::SubProof::CompleteProof::ScaledEstimates}
   does not depend on \(\alpha\) or \(\beta\).
   Thus,  as described in Remark \ref{Rmk::Results::cIsAdmissible},
   \(c>0\) from Theorem \ref{Thm::Results::MainThm} does not depend on \(\alpha\) or \(\beta\).

   From here, and using the fact that \(\psi=1\) on a neighborhood of \(\Compact\),
   \eqref{Eqn::SubMainThm::MainBound} follows from Theorem \ref{Thm::Results::MainThm}.
\end{proof}

    \subsection{Boundary value problems}\label{Section::MaxSub::BoundaryValueProblems}
    Theorem \ref{Thm::SubMainThm::MainThm} and Proposition \ref{Prop::SubEllip::Examples::ConditionMConclusions} 
are interior results which apply on manifolds without boundary.
One could hope instead for results concerning homogeneous boundary value problems and this was one of the main
original motivations for the results in this paper.

Roughly speaking, 
the assumptions of 
Proposition \ref{Prop::SubEllip::Examples::ConditionMConclusions}
can be stated in terms of the quadratic form
\begin{equation*}
    \FormQ[f][g]=\Ltip{f}{\opL g}[\Manifold][\mu].
\end{equation*}
where \(\FormQ\) is given the a priori domain \(\CzinftySpace[\Manifold]\).
Note that we have
\begin{equation}\label{Eqn::SubEllip::BoundaryValues:FormDefn}
    \FormQ[f][g]
    =\sum_{\degWd(\alpha),\degWd(\beta)\leq \kappa}\Ltip{W^{\alpha} f}{a_{\alpha,\beta}W^{\beta}g}[\Manifold][\mu].
\end{equation}
By changing the domain of \(\FormQ\),
using the formula \eqref{Eqn::SubEllip::BoundaryValues:FormDefn} (and by possibly changing the choice of \(a_{\alpha,\beta}\), keeping the same
underlying operator on the interior intact\footnote{
    The coefficients \(a_{\alpha,\beta}\) may not be uniquely determined by \(\opL\). For an extreme example,
    note that \(0\partial_x \partial_y + 0\partial_y \partial_x = 0 = \partial_x \partial_y - \partial_y\partial_x\). Different choices of \(a_{\alpha,\beta}\)
    in the definition of \(\FormQ\) in \eqref{Eqn::SubEllip::BoundaryValues:FormDefn} can lead to different boundary conditions.
    See \cite[Example 7.11]{FollandIntroductionToPartialDifferentialEquations} for this idea in action.
}), such forms give rise to operators \((A,\Domain[A])\)
which correspond to different homogeneous boundary value problems. See
\cite[Chapter 7, Section C]{FollandIntroductionToPartialDifferentialEquations}
for a friendly introduction.

Due to its abstract nature,
Theorem \ref{Thm::Results::MainThm} does not require any changes to apply such homogeneous boundary value problems.
However, to prove a result like Theorem \ref{Thm::SubMainThm::MainThm},
we require generalizations of Theorem \ref{Thm::MaxSub::Scaling::NSWScaling}
and Proposition \ref{Prop::Subellip::Proof::Apriori}
which work for boundary value problems.
These involve some added difficulties because one cannot as easily reduce the problems to standard
distribution theory: one needs to work carefully with the boundary and \(\Domain[A]\).

A generalization of Theorem \ref{Thm::MaxSub::Scaling::NSWScaling} for boundary value problems
can be found in the preprint \cite{StreetCarnotCaratheodoryBallsOnManifoldsWithBoundary},
while a version of Proposition \ref{Prop::Subellip::Proof::Apriori} for boundary value problems
is contained in the forthcoming paper \cite{StreetAPrioriEstimatesForMaximallySubellipticQuadraticForms}.
In a forthcoming paper, we will combine these with the main results of this paper to address
heat operators corresponding to maximally subelliptic boundary value problems.

\section{The Proof}
In this section, we prove Theorems \ref{Thm::Results::DefineOperator} and \ref{Thm::Results::MainThm}.
We being with three reductions.
\begin{itemize}
    \item By replacing \(T(t)\) with \(e^{-\omega_0 t}T(t)\), we may assume \(\omega_0=0\); we henceforth do this.
    \item Because the definition of admissible constants is uniform across \(x\in \NSubsetx\) and \(y\in \NSubsety\),
        it suffices to prove the results in the special case \(\NSubsetx=\{x_0\}\) and \(\NSubsety=\{y_0\}\); we henceforth make this replacement.
    \item Because our assumptions are symmetric in \(T(t)\) and \(T(t)^{*}\) (see Remark \ref{Rmk::Assumption::Symmetric} for details), we will often state and prove lemmas
        only for \(T(t)\), and later also use the conclusion for \(T(t)^{*}\).
\end{itemize}

    \subsection{Gevrey regularity and exponential bounds}
    The exponential decay in the Gaussian bounds \eqref{Eqn::Results::MainGaussianBounds} is introduced by showing that
certain Gevrey functions automatically have such an exponential decay; this section is devoted to showing
this fact. The main proof we present (Lemma \ref{Lemma::Proof::Gevrey::NoDerivs}) was done by Jerison and S\'anchez-Calle when \(\gamma=2\) (see  \cite[Lemma 2]{JerisonSanchezCalleEstimatesForTheHeatKernelForASumOfSquaresOfVectorFields})
and by Hebeisch for general \(\gamma>1\) (see  \cite[Lemma 6]{HebischSharpPointwiseEstimateForTheKernelsOfTheSemigroupGeneratedBySumsOfEvenPowersOfVectorFieldsOnHomogeneousGroups}).
We use the convention \(0^0=1\).

\begin{proposition}\label{Prop::Proof::Gevrey::MainProp}
    Let \(I\subseteq \R\) be an open interval containing \(0\) and suppose \(f\in \CinftySpace[I]\) satisfies
    \(f(t)=0\) for \(t<0\) and \(\exists \gamma>1\), \(C\geq 0\), \(R>0\), such that
    \begin{equation*}
        \left| \partial_t^l f(t) \right|\leq C R^l (l!)^{\gamma}, \quad \forall t\in I\cap (0,\infty),\: l\in \N.
    \end{equation*}
    Then, \(\forall t\in I\), \(l\in \N\),
    \begin{equation}\label{Eqn::Proof::Gevrey::MainBound}
        \left| \partial_t^l f(t) \right| \leq C R^l 2^{l\gamma} l^{l\gamma} e^{\gamma-1}\exp\left\{ -(\gamma-1) e^{-1}\left( (Re^{\gamma}) t  \right)^{-1/(\gamma-1)}  \right\}.
    \end{equation}
\end{proposition}

We turn to the proof of Proposition \ref{Prop::Proof::Gevrey::MainProp}; for it we need two lemmas. In the first,
we establish a similar result for \(l=0\).

\begin{lemma}\label{Lemma::Proof::Gevrey::NoDerivs}
    Let \(I\subseteq \R\) be an open interval containing \(0\) and suppose \(f\in \CinftySpace[I]\) satisfies
    \(f(t)=0\) for \(t<0\) and \(\exists \gamma>1\), \(C\geq 0\), \(R>0\), such that
    \begin{equation}\label{Eqn::Proof::Gevrey::MainGevAssumpInLemma}
        \left| \partial_t^l f(t) \right|\leq C R^l (l!)^{\gamma}, \quad \forall t\in I\cap (0,\infty),\: l\in \N.
    \end{equation}
    Then, \(\forall t\in I\)
    \begin{equation}\label{Eqn::Proof::Gevrey::MainBound::Withl0}
        \left| f(t) \right| \leq C e^{\gamma-1} \exp\left\{ -(\gamma-1)e^{-1}(Rt)^{-1/(\gamma-1)} \right\}.
    \end{equation}
\end{lemma}
\begin{proof}
    \eqref{Eqn::Proof::Gevrey::MainBound::Withl0} is trivial for \(t\leq 0\) (since \(f(t)=0\)), so we assume \(t>0\).
    If \(0\leq e^{-1}\left( Rt \right)^{-1/(\gamma-1)}<1\), then the right-hand side of \eqref{Eqn::Proof::Gevrey::MainBound::Withl0}
    is \(\geq C\) and \eqref{Eqn::Proof::Gevrey::MainBound::Withl0} follows from the case \(l=0\) of \eqref{Eqn::Proof::Gevrey::MainGevAssumpInLemma}.
    Otherwise, pick \(l\in \N\) such that \(l+1\leq e^{-1}(Rt)^{-1/(\gamma-1)}<l+2\), so that \((Rt)^{1/(\gamma-1)}\leq 1/e(l+1)\).
    We have,
    \begin{equation*}
    \begin{split}
         |f(t)|\leq &\int_0^t \frac{(t-s)^l}{l!} \left| f^{(l+1)}(s) \right|\: ds \leq C R^{l+1} \left( \left( l+1 \right)! \right)^{\gamma}\frac{t^{l+1}}{\left( l+1 \right)!}
         \\&=CR^{l+1} \left( \left( l+1 \right)! \right)^{\gamma-1} t^{l+1}
         \leq C\left[ \left( Rt \right)^{1/(\gamma-1)} (l+1) \right]^{(\gamma-1)(l+1)}
         \\&\leq C e^{-(\gamma-1)((l+2)-1)}
         \leq C \exp\left\{  -(\gamma-1) \left( e^{-1} \left( Rt \right)^{-1/(\gamma-1)}-1 \right) \right\}.
    \end{split}
    \end{equation*}
\end{proof}

\begin{lemma}\label{Lemma::Proof::Gevrey::ReduceDerivs}
    Fix \(\gamma,R>0\) and \(I\subseteq \R\) an open interval. Suppose \(f\in \CinftySpace[I]\) satisfies
    \begin{equation}\label{Eqn::Proof::Gevrey::ReduceDerivs::Assump1}
        \left| \partial_t^l f(t) \right|\leq C R^l (l!)^{\gamma}.
    \end{equation}
    Then, \(\forall l_0\in \N\), \(f^{(l_0)}(t)=\partial_t^{l_0} f(t)\) satisfies
    \begin{equation*}
        \left| \partial_t^l f^{(l_0)}(t) \right|
        \leq C R^{l_0} 2^{l_0 \gamma} l_0^{l_0\gamma} \left( R e^{\gamma} \right)^l  \left( l! \right)^{\gamma}, \quad \forall l\in \N.
    \end{equation*}
\end{lemma}
\begin{proof}
    The result is trivial when \(l_0=0\), so we assume \(l_0\geq 1\). We claim
    \begin{equation}\label{Eqn::Proof::Gevrey::ReduceDerivs::ToShow1}
        (l+l_0)(l+l_0-1)\cdots(l+1)\leq 2^{l_0} l_0^{l_0}e^{l}, \quad \forall l\geq 0.
    \end{equation}
    Indeed, for \(l_0\geq l\), the left hand side of \eqref{Eqn::Proof::Gevrey::ReduceDerivs::ToShow1} is \(\leq (2l_0)^{l_0}\) and \eqref{Eqn::Proof::Gevrey::ReduceDerivs::ToShow1} follows.
    If \(l\geq l_0\), the left hand side of \eqref{Eqn::Proof::Gevrey::ReduceDerivs::ToShow1} is \(\leq 2^{l_0}l^{l_0}\leq 2^{l_0}(l_0!) e^{l}\leq 2^{l_0} l_0^{l_0}e^{l}\),
    where we have used \(l^{l_0}/l_0!\leq e^{l}\) (since \(l^{l_0}/l_0!\) is a term in the power series of \(e^l\)).

    Using \eqref{Eqn::Proof::Gevrey::ReduceDerivs::Assump1} and \eqref{Eqn::Proof::Gevrey::ReduceDerivs::ToShow1}, we have
    \begin{equation*}
    \begin{split}
         &\left| \partial_t^l f^{(l_0)}(t) \right|
         \leq C R^{l} R^{l_0} \left( (l+l_0)! \right)^{\gamma}
         =C R^{l}R^{l_0} \left( (l+l_0)(l+l_0-1)\cdots(l+1) \right)^{\gamma} (l!)^{\gamma}
         \\&\leq C R^{l} R^{l_0} (2^{l_0}l_0^{l_0}e^{l})^{\gamma} \left(l! \right)^{\gamma}
         =C R^{l_0} 2^{l_0\gamma} l_0^{l_0 \gamma} (Re^{\gamma})^l (l!)^{\gamma}.
    \end{split}
    \end{equation*}
\end{proof}

\begin{proof}[Proof of Proposition \ref{Prop::Proof::Gevrey::MainProp}]
    Lemma \ref{Lemma::Proof::Gevrey::ReduceDerivs} shows that we may apply Lemma \ref{Lemma::Proof::Gevrey::NoDerivs}
    to \(f^{(l_0)}\), for every \(l_0\in \N\), with \(C\) replaced by \(CR^{l_0}2^{l_0\gamma}l_0^{l_0\gamma}\),
    and \(R\) replaced by \(Re^{\gamma}\). The result now follows from Lemma \ref{Lemma::Proof::Gevrey::NoDerivs}.
\end{proof}

    \subsection{Lemmas about doubling}
    \begin{lemma}\label{Lemma::ProofDoubling::AllDoubling}
    \(\forall \delta>0\), \(\measure(\MetricBall{x_0}{(2\delta)\wedge \delta_0})\leq D_1\measure(\MetricBall{x_0}{\delta\wedge \delta_0})\).
\end{lemma}
\begin{proof}
    If \(2\delta<\delta_0\), this follows from Assumption \ref{Assumption::LocalDoubling}. 
    If \(\delta_0=\infty\), we are done.
    Otherwise, for \(2\delta\geq\delta_0\),
    it suffices to show 
    \begin{equation}\label{Eqn::ProofDoubling::AllDoubling::Tmp1}
        \measure(\MetricBall{x_0}{\delta_0})\leq D_1 \measure(\MetricBall{x_0}{\delta_0/2}).
    \end{equation}
    To prove \eqref{Eqn::ProofDoubling::AllDoubling::Tmp1}, we take \(\delta_1<\delta_0\). Assumption \ref{Assumption::LocalDoubling}
    shows \(\measure(\MetricBall{x_0}{\delta_1})\leq D_1 \measure(\MetricBall{x_0}{\delta_1/2})\). Taking the limit \(\delta_1\uparrow \delta_0\)
    completes the proof.
\end{proof}

The next lemma is standard (see, for example,  \cite[page 32]{SteinHarmonicAnalysis}), but we include the proof for completeness.

\begin{lemma}\label{Lemma::Doubling::Covering}
    There exists an admissible constant \(N\in \N\) such that the following holds.
    \(\forall \delta_1\in (0,\delta_0)\), \(\delta_2\in (0,(\delta_0-\delta_1)/2)\),
    \(\MetricBall{x_0}{\delta_1}\) can be covered by a collection \(\left\{ \MetricBall{x_j}{a_2\delta_2} \right\}_{j=1,2,\ldots}\),
    where \(x_j\in \MetricBall{x_0}{\delta_1}\), and no point lies in more than \(N\)
    of the balls \(\MetricBall{x_j}{\delta_2}\).
\end{lemma}
\begin{proof}
    Let \(\MetricBall{x_j}{a_2\delta_2/2}\), \(j=1,2,\ldots\) be a maximal disjoint subcollection of
    \begin{equation*}
        \left\{ \MetricBall{x}{a_2\delta_2/2} : x\in \MetricBall{x_0}{\delta_1} \right\}.
    \end{equation*}
    Note that
    \begin{equation*}
        \bigcup_{j\geq 1} \MetricBall{x_j}{a_2\delta_2}\supseteq \MetricBall{x_0}{\delta_1},
    \end{equation*}
    since if \(y\in \MetricBall{x_0}{\delta_1}\setminus \bigcup_{j\geq 1} \MetricBall{x_j}{a_2\delta_2}\),
    then \(\MetricBall{y}{a_2 \delta_2/2}\) is disjoint from the collection, contradicting maximality.

    Now suppose \(y\in \MetricBall{x_{i_1}}{\delta_2}\cap \cdots \cap \MetricBall{x_{i_N}}{\delta_2}\); we wish to show \(N\lesssim 1\).
    We have,
    \begin{equation*}
        \MetricBall{y}{2\delta_2}\supseteq \MetricBall{x_{i_l}}{\delta_2}\supseteq \MetricBall{x_{i_l}}{a_2 \delta_2/2}.
    \end{equation*}
    Because the balls \(\MetricBall{x_{i_l}}{a_2 \delta_2/2}\) are disjoint, we have
    \begin{equation}\label{Eqn::ProofDoubling::CoveringLemma::Tmp1}
        \measure (\MetricBall{y}{2\delta_2})\geq \sum_{l=1}^N \measure(\MetricBall{x_{i_l}}{a_2\delta_2/2}).
    \end{equation}
    Using repeated applications of Assumption \ref{Assumption::LocalDoubling}, we have
    \begin{equation}\label{Eqn::ProofDoubling::CoveringLemma::Tmp2}
        \measure(\MetricBall{x_{i_l}}{a_2\delta_2/2}) \approx \measure(\MetricBall{x_{i_l}}{2\delta_2})\geq \measure(\MetricBall{y}{\delta_2})\approx \measure(\MetricBall{y}{2\delta_2}).
    \end{equation}
    Combining \eqref{Eqn::ProofDoubling::CoveringLemma::Tmp1} and \eqref{Eqn::ProofDoubling::CoveringLemma::Tmp2}, we see
    \begin{equation*}
        \measure (\MetricBall{y}{2\delta_2}) \gtrsim N \measure (\MetricBall{y}{2\delta_2}).
    \end{equation*}
    Since \(0<\measure (\MetricBall{y}{2\delta_2}) <\infty\) (see Assumptions \ref{Assumption::LocalPositivity} and \ref{Assumption::LocalDoubling}),
    it follows \(N\lesssim 1\), completing the proof.
\end{proof}

\begin{lemma}\label{Lemma::ProofDoubling::ExpBeatsDoubling}
    Suppose \(c,\gamma>0\) and \(D\geq 1\) are given, and \(F(\delta):(0,\infty)\rightarrow (0,\infty)\)
    is a non-decreasing function satisfying \(F(2\delta)\leq DF(\delta)\).
    Then, \(\exists C=C(c,\gamma,D)\geq 1\), depending on nothing else, such that
    \begin{equation*}
        \sup_{\delta_1,\delta_2>0} \frac{F(\delta_1+\delta_2)}{F(\delta_1)} \exp\left( -c\left( \frac{\delta_2}{\delta_1} \right)^{\gamma} \right)\leq C.
    \end{equation*}
\end{lemma}
\begin{proof}
    If \(\delta_2\leq \delta_1\), we have \(F(\delta_1+\delta_2)/F(\delta_1)\leq F(2\delta_1)/F(\delta_1)\leq D\) and the result follows.
    Now let \(\delta_2>\delta_1\). 
    Fix \(j\in \N\) with \(2^{j}\delta_1<\delta_2\leq 2^{j+1}\delta_1\). Then,
    \begin{equation*}
        \frac{F(\delta_1+\delta_2)}{F(\delta_1)} \leq \frac{F((2^{j+1}+1)\delta_1)}{F(\delta_1)}\leq \frac{F(2^{j+2}\delta_1)}{F(\delta_1)}
        \leq D^{j+2} = D^2 \left( 2^{j} \right)^{\log_2(D)}\leq D^2 \left( \frac{\delta_2}{\delta_1} \right)^{\log_2(D)}.
    \end{equation*}
    We conclude,
    \begin{equation*}
        \sup_{\delta_2>\delta_1>0} \frac{F(\delta_1+\delta_2)}{F(\delta_1)} \exp\left( -c\left( \frac{\delta_2}{\delta_1} \right)^{\gamma} \right)
        \leq \sup_{\delta_2>\delta_1>0} D^2 \left( \frac{\delta_2}{\delta_1} \right)^{\log_2(D)}\exp\left( -c\left( \frac{\delta_2}{\delta_1} \right)^{\gamma} \right).
    \end{equation*}
    The right-hand side of the above equation is finite and depends only on \(c\), \(D\), and \(\gamma\), completing the proof.
\end{proof}

    \subsection{Heat functions}
    In this section, we prove properties about \((A,x,\delta)\)-heat functions as described in Definition \ref{Defn::Assump::Hypo::HeatFunction}.
They arise through the following lemma.

\begin{lemma}\label{Lemma::HeatFuncs::Characterize}
    Fix \(\delta>0\).
    \begin{enumerate}[(i)]
        \item\label{Item::HeatFuncs::Characterize::NonNegativeReals} Let \(\phi \in \DAinfty\), and \(u(t):=T(\delta^{2\kappa}t)\phi\). Then, \(u\in \CinftySpace*[\lbrack 0,\infty)][\DAinfty]\) 
            and satisfies \(\partial_t^j u(t) = \left( \delta^{2\kappa} A \right)^j u(t)\), 
            \(\forall j\in \N\).
        \item\label{Item::HeatFuncs::Characterize::PositiveReals} Let \(\phi \in \DAinfty\), and \(u(t):=T(\delta^{2\kappa}t)\phi\). Then, \(u:(0,\infty)\rightarrow \DAinfty\) is an \((A,x,\delta)\)-heat function,
            for every \(x\in \MetricSpace\).
        \item\label{Item::HeatFuncs::Characterize::AllReals} Fix \(x\in \MetricSpace\) and suppose \(\phi\in \DAinfty\) is such that \(A^j\phi\big|_{\MetricBall{x}{\delta}}=0\), \(\forall j\geq 0\).
            Define \(u:\R\rightarrow \DAinfty\) by
            \begin{equation*}
                u(t):=\begin{cases}
                    T(\delta^{2\kappa}t) \phi, &t\geq 0,\\
                    0, &t<0.
                \end{cases}
            \end{equation*}
            Then, \(u\) is an \((A,x,\delta)\)-heat function.
    \end{enumerate}
\end{lemma}
\begin{proof}
    \ref{Item::HeatFuncs::Characterize::NonNegativeReals}: 
    By replacing \(t\) with \(\delta^{-2\kappa}t\), it suffices to consider the case \(\delta=1\).
    By  \cite[Chapter 2, Lemma 1.3]{EngelNagelAShortCourseOnOperatorSemigroups}
    we have \(u\in C^1\left( \lbrack 0,\infty); \LpSpace{p_0}[\MetricSpace][\measure] \right)\) and satisfies
    \begin{equation*}
        \partial_t u(t) = \partial_t T(t)\phi = AT(t)\phi = T(t)A\phi.
    \end{equation*}
    Since \(A\phi\) is of the same form as \(\phi\), iterating this 
    \(u(t)\in \CinftySpace*[\lbrack 0,\infty )][\LpSpace{p_0}[\MetricSpace][\measure]]\)
    and satisfies \(\partial_t^j u(t) = A^j u(t)\).
    Finally, since \(Au(t)=AT(t)\phi = T(t)A\phi\) is of the same form as \(u(t)\), we see
    \(u(t)\in \CinftySpace*[\lbrack 0,\infty )][\DAinfty]\).

    \ref{Item::HeatFuncs::Characterize::PositiveReals} follows from \ref{Item::HeatFuncs::Characterize::NonNegativeReals}, by restricting
    to \((0,\infty)\).

    For \ref{Item::HeatFuncs::Characterize::AllReals}, 
    we verify the conditions of Definition \ref{Defn::Assump::Hypo::HeatFunction} hold with \(t_0=0\).
    Define 
    \(v_1\in \CinftySpace*[(-\infty,0,\rbrack][\DAinfty]\)
    by \(v_1(t)=0\), \(\forall t\),
    and
     \(v_2\in \CinftySpace*[\lbrack 0,\infty)][\DAinfty]\) 
     by \(v_2(t)=T(\delta^{2\kappa}t)\phi\) (see \ref{Item::HeatFuncs::Characterize::NonNegativeReals}).
    By \ref{Item::HeatFuncs::Characterize::NonNegativeReals} 
    \(v_2\) satisfies \(\partial_t^j v_2(t) =(\delta^{2\kappa}A)^j v_2(t)\), \(\forall j\). Clearly, 
     \(\partial_t^j v_1(t)=0=(\delta^{2\kappa}A)^jv_1(t)\), \(\forall j\).
     Since 
     \begin{equation*}
        u(t):=\begin{cases}
           v_2(t), &t\geq 0,\\
           v_1(t), &t<0,
       \end{cases}
     \end{equation*}
    to complete the proof it suffices to show \(v_1^{(j)}(0)\big|_{\MetricBall{x}{\delta}}=v_2^{(j)}(0)\big|_{\MetricBall{x}{\delta}}\), \(\forall j\).
    I.e., we wish to show \(v_2^{(j)}(0)\big|_{\MetricBall{x}{\delta}}=0\), \(\forall j\in \N\). But, by  \cite[Chapter 2, Lemma 1.3]{EngelNagelAShortCourseOnOperatorSemigroups},
    \begin{equation*}
        v_2^{(j)}(0)\big|_{\MetricBall{x}{\delta}} = T(0)(\delta^{2\kappa}A)^j \phi\big|_{\MetricBall{x}{\delta}}=(\delta^{2\kappa}A)^j \phi \big|_{\MetricBall{x}{\delta}} =0,
    \end{equation*}
    completing the proof.
\end{proof}


A main result of this section is the following regularity for heat functions.

\begin{proposition}\label{Prop::Proof::HeatFuncs::MainExpBound}
    There exist admissible constants \(c_1>0\) and \(B_1\geq 1\) such that the following holds.
    Let \(\delta_1\in (0,\delta_0)\) and \(u:(-1,1)\rightarrow \DAinfty\) be an \((A,x_0,\delta_1)\)-heat function
    satisfying \(u\big|_{(-1,0)\times \MetricBall{x_0}{\delta_1}}=0\). Then,
    \(Xu\big|_{(-a_1,a_1)\times \MetricBall{x_0}{a_1\delta_1}} \in\CinftySpace[(-a_1,a_1)][\CSpace{\MetricBall{x_0}{a_1\delta_1}}]\)
    and, for every \(l\in \N\), \(\forall t\in (-a_1,a_1)\)
    \begin{equation}\label{Eqn::Proof::HeatFunc::MainExpBound}
\begin{split}
            &\sup_{x\in \MetricBall{x_0}{a_1\delta_1}} \left| \partial_t^l X u(t,x) \right|
            \\&\lesssim S_X(\delta_1)^{-1} \measure(\MetricBall{x_0}{\delta_1})^{-1/p_0} B_1^l l^{l\gamma} \exp\left( -c_1 t^{-1/(2\kappa-1)} \right)\LpNorm{u}{p_0}[(-1,1)\times \MetricBall{x_0}{\delta_1}][dt\times d\measure].
\end{split}    \end{equation}
\end{proposition}

The main thing Proposition \ref{Prop::Proof::HeatFuncs::MainExpBound} gives is that improves
\eqref{Eqn::Assump::Hyp::HypoellipticityIBound} to include the exponential \(\exp\left( -c_1 t^{-1/(2\kappa-1)}\right)\); this is how
we introduce the exponential in the Gaussian bounds. 
We turn to proving Proposition \ref{Prop::Proof::HeatFuncs::MainExpBound}; we require several preliminary results.
The goal is to apply Proposition \ref{Prop::Proof::Gevrey::MainProp}, and therefore we need to prove a certain Gevrey regularity for \(Xu(t,x)\) in the \(t\)-variable.

\begin{lemma}\label{Lemma::Proof::HeatFuncs::EstimateUnscaledt}
    \(\forall x\in \MetricBall{x_0}{\delta_0}\), \(\forall \delta\in (0,\delta_0-\metric[x_0][x])\), \(\forall \epsilon\in (0,1)\),
    \(\forall u:(-\epsilon^{-2\kappa}, \epsilon^{-2\kappa})\rightarrow \DAinfty\) an \((A,x,\delta)\)-heat function,
    \(\forall k\in [1, (2\epsilon)^{-1}]\),
    \begin{equation}\label{Eqn::Proof::HeatFunc::EsimateUnscaledt}
    \begin{split}
            &\LpNorm{\partial_t u}{p_0}[(-\epsilon^{-2\kappa}+(k\epsilon)\epsilon^{-2\kappa}, \epsilon^{-2\kappa}-(k\epsilon)\epsilon^{-2\kappa})\times \MetricBall{x}{a_2\delta}]
            \\&\lesssim 
            \LpNorm{u}{p_0}[(-\epsilon^{-2\kappa}+((k-1)\epsilon)\epsilon^{-2\kappa}, \epsilon^{-2\kappa}-((k-1)\epsilon)\epsilon^{-2\kappa})\times \MetricBall{x}{\delta}].
    \end{split}    
    \end{equation}
\end{lemma}
\begin{proof}
    Note that
    \begin{equation*}
        \begin{split}
        &(-\epsilon^{-2\kappa}+((k-1)\epsilon)\epsilon^{-2\kappa}, \epsilon^{-2\kappa}-((k-1)\epsilon)\epsilon^{-2\kappa})
        \setminus (-\epsilon^{-2\kappa}+(k\epsilon)\epsilon^{-2\kappa}, \epsilon^{-2\kappa}-(k\epsilon)\epsilon^{-2\kappa})
        \\& = (-\epsilon^{-2\kappa}+((k-1)\epsilon)\epsilon^{-2\kappa},-\epsilon^{-2\kappa}+(k\epsilon)\epsilon^{-2\kappa}]
        \bigcup
        [\epsilon^{-2\kappa}-(k\epsilon)\epsilon^{-2\kappa}, \epsilon^{-2\kappa}-((k-1)\epsilon)\epsilon^{-2\kappa})
    \end{split}
    \end{equation*}
    which are two intervals of length \(\epsilon^{-2\kappa+1}\geq \epsilon^{-1}> 1\). 
    Thus, we may cover \((-\epsilon^{-2\kappa}+(k\epsilon)\epsilon^{-2\kappa}, \epsilon^{-2\kappa}-(k\epsilon)\epsilon^{-2\kappa})\)
    by a finite collection of intervals of the form \((t_1-a_2, t_i+a_2)\) with \(t_i\in (-\epsilon^{-2\kappa}+(k\epsilon)\epsilon^{-2\kappa}, \epsilon^{-2\kappa}-(k\epsilon)\epsilon^{-2\kappa})\)
    (and therefore \((t_1-1,t_i+1)\subset (-\epsilon^{-2\kappa}+((k-1)\epsilon)\epsilon^{-2\kappa}, \epsilon^{-2\kappa}-((k-1)\epsilon)\epsilon^{-2\kappa})\)),
    and such that no point lies in more than \(2\ceil{a_2^{-1}}\) of the intervals \((t_i-1,t_i+1)\).

    Assumption \ref{Assumption::HypoellipticityII} gives
    \begin{equation*}
        \LpNorm{\partial_t u}{p_0}[(t_i-a_2,t_i+a_2)\times \MetricBall{x}{a_2\delta}]
        \leq C_1 \LpNorm{u}{p_0}[(t_i-1,t_i+1)\times \MetricBall{x}{\delta}].
    \end{equation*}
    Combining this with the above, we have
    \begin{equation*}
    \begin{split}
         &\LpNorm{\partial_t u}{p_0}[(-\epsilon^{-2\kappa}+(k\epsilon)\epsilon^{-2\kappa}, \epsilon^{-2\kappa}-(k\epsilon)\epsilon^{-2\kappa})\times \MetricBall{x}{a_2\delta}]^{p_0}
         \\&\leq \sum_{i} \LpNorm{\partial_t u}{p_0}[(t_i-a_2,t_i+a_2)\times \MetricBall{x}{a_2\delta}]^{p_0}
         \\&\leq \sum_{i} C_1^{p_0}\LpNorm{u}{p_0}[(t_i-1,t_i+1)\times \MetricBall{x}{\delta}]^{p_0}
         \\&\leq 2C_1^{p_0}\ceil{a_2^{-1}} \LpNorm{u}{p_0}[(-\epsilon^{-2\kappa}+((k-1)\epsilon)\epsilon^{-2\kappa}, \epsilon^{-2\kappa}-((k-1)\epsilon)\epsilon^{-2\kappa})\times \MetricBall{x}{\delta}]^{p_0}.
    \end{split}
    \end{equation*}
    The result follows.
\end{proof}

\begin{lemma}\label{Lemma::Proof::HeatFuncs::Inducel::Tmp}
    \(\forall \delta_1\in (0,\delta_0)\),
    \(\forall x\in \MetricBall{x_0}{\delta_1}\), \(\forall \delta\in (0,\delta_1-\metric[x_0][x])\), \(\forall k\in [1, \delta_1/2\delta]\),
    \(\forall u:(-1,1)\rightarrow \DAinfty\) an \((A,x_0,\delta_1)\)-heat function,
    \begin{equation*}
        \LpNorm{\partial_t u}{p_0}[(-1+k\delta/\delta_1, 1-k\delta/\delta_1)\times \MetricBall{x}{a_2\delta}]
        \lesssim \left( \delta/\delta_1 \right)^{-2\kappa} \LpNorm{u}{p_0}[(-1+(k-1)\delta/\delta_1, 1-(k-1)\delta/\delta_1)\times \MetricBall{x}{\delta}].
    \end{equation*}
\end{lemma}
\begin{proof}
    Set \(\epsilon:=\delta/\delta_1\in (0,1)\) and \(\uh(t):=u(\epsilon^{2\kappa}t)\).
    Then, \(\uh:(-\epsilon^{-2\kappa}, \epsilon^{-2\kappa})\rightarrow \DAinfty\),
    \(\partial_t^j \uh(t,x)\big|_{(-\epsilon^{-2\kappa},\epsilon^{2\kappa})\times \MetricBall{x_0}{\delta_1}} = \left( \delta^{2\kappa}A \right)^j \uh(t,x)\big|_{(-\epsilon^{-2\kappa},\epsilon^{2\kappa})\times \MetricBall{x_0}{\delta_1}}\), \(\forall j\),
    and \(\uh\big|_{(-\epsilon^{-2\kappa},\epsilon^{2\kappa})\times \MetricBall{x_0}{\delta_1}}\in \CinftySpace[(-\epsilon^{-2\kappa},\epsilon^{-2\kappa})][\LpSpace{p_0}[\MetricBall{x_0}{\delta_1}]]\).
    Since \(\MetricBall{x}{\delta}\subseteq \MetricBall{x_0}{\delta_1}\), it follows that \(\uh\) is an \((A,x,\delta)\)-heat function.
    Lemma \ref{Lemma::Proof::HeatFuncs::EstimateUnscaledt} shows
    \eqref{Eqn::Proof::HeatFunc::EsimateUnscaledt} holds with \(u\) replaced by \(\uh\).
    Using that \(\uh(t,x)=u(\epsilon^{2\kappa}t,x)\), the result follows by changing variables in \(t\) and using \(\epsilon=\delta/\delta_1\).
\end{proof}

\begin{lemma}\label{Lemma::Proof::HeatFuncs::Inducel::Tmp2}
    Let \(\delta_1\in (0,\delta_0)\) and let \(u:(-1,1)\rightarrow \DAinfty\) be an \((A,x_0,\delta_1)\)-heat function.
    Then, \(\forall \delta\in (0,\delta_1)\), \(\forall k\in [2,\delta_1/2\delta]\),
    \begin{equation*}
        \LpNorm{\partial_t u}{p_0}[(-1+k\delta/\delta_1, 1-k\delta/\delta_1)\times \MetricBall{x_0}{\delta_1-k\delta}]
        \lesssim \left( \delta/\delta_1 \right)^{-2\kappa} \LpNorm{u}{p_0}[(-1+(k-1)\delta/\delta_1, 1-(k-1)\delta/\delta_1)\times \MetricBall{x_0}{\delta_1-(k-1)\delta}].
    \end{equation*}
\end{lemma}
\begin{proof}
    Let \(\deltah_1=\delta_1-k\delta\), and \(\deltah_2=\delta\). Note that \(\deltah_2\leq (\delta_1-\deltah_1)/2<(\delta_0-\delta_1)/2\).
    By Lemma \ref{Lemma::Doubling::Covering} with \(\delta_1\) and \(\delta_2\) replaced with \(\deltah_1\) and \(\deltah_2\),
    there exists
    an admissible constant \(N\in \N\), and \(x_1,x_2,\ldots \in \MetricBall{x_0}{\delta_1-k\delta}\), such that
    \(\bigcup_j \MetricBall{x_j}{a_2\delta}\supseteq \MetricBall{x_0}{\delta_1-k\delta}\) and no point lies in more than \(N\)
    of the balls \(\MetricBall{x_j}{\delta}\).
    The triangle inequality implies \(\MetricBall{x_j}{\delta}\subseteq \MetricBall{x_0}{\delta_1-(k-1)\delta}\).
    
    Also, \(\delta\leq k\delta =\delta_1-(\delta_1-k\delta) <\delta_1-\metric[x_0][x_j]\).  Lemma \ref{Lemma::Proof::HeatFuncs::Inducel::Tmp}
    shows for some admissible constant \(C\geq 1\),
    \begin{equation*}
    \begin{split}
         &\LpNorm{\partial_t u}{p_0}[(-1+k\delta/\delta_1,1-k\delta/\delta_1)\times \MetricBall{x_0}{\delta_1-k\delta}]^{p_0}
         \leq \sum_{j\geq 1} \LpNorm{\partial_t u}{p_0}[(-1+k\delta/\delta_1,1-k\delta/\delta_1)\times \MetricBall{x_j}{a_2 \delta}]^{p_0}
         \\&\leq C^{p_0} \left( \delta/\delta_1 \right)^{-2\kappa p_0} \sum_{j\geq 1} \LpNorm{u}{p_0}[(-1+(k-1)\delta/\delta_1,1-(k-1)\delta/\delta_1)\times \MetricBall{x_j}{\delta}]^{p_0}
         \\&\leq C^{p_0} N \left( \delta/\delta_1 \right)^{-2\kappa p_0}\LpNorm{u}{p_0}[(-1+(k-1)\delta/\delta_1,1-(k-1)\delta/\delta_1)\times \MetricBall{x_0}{\delta_1-(k-1)\delta}]^{p_0}.
    \end{split}
    \end{equation*}
    The result follows.
\end{proof}

\begin{lemma}\label{Lemma::Proof::HeatFuncs::Inducel::Firstl}
    Let \(\delta_1\in (0,\delta_0)\) and \(u:(-1,1)\rightarrow \DAinfty\) be an \((A,x_0,\delta_1)\)-heat function. Then,
    for \(l\in \left\{ 1,2,3,\ldots \right\}\), \(k\in \left\{ 2,3,\ldots, l+1 \right\}\),
    \begin{equation*}
    \begin{split}
            &\LpNorm{\partial_t u}{p_0}[(-1+k/2(l+1), 1-k/2(l+1))\times \MetricBall{x_0}{\delta_1(1-k/2(l+1))}]
            \\&\lesssim l^{2\kappa}\LpNorm{u}{p_0}[(-1+(k-1)/2(l+1), 1-(k-1)/2(l+1))\times \MetricBall{x_0}{\delta_1(1-(k-1)/2(l+1))}] 
    \end{split}    
    \end{equation*}
\end{lemma}
\begin{proof}
    This follows from Lemma \ref{Lemma::Proof::HeatFuncs::Inducel::Tmp2} with \(\delta=\delta_1/2(l+1)\).
\end{proof}

\begin{lemma}\label{Lemma::Proof::HeatFuncs::DerivOfHeatIsHeat}
    Let \(\delta>0\), \(x\in \MetricSpace\), \(I\subseteq \R\) a connected open interval, and \(u:I\rightarrow \DAinfty\) an \((A,x,\delta)\)-heat function.
    \begin{enumerate}[(i)]
        \item\label{Item::Proof::HeatFuncs::DerivOfHeatIsHeat::ApplyA} \(\delta^{2\kappa}Au:I\rightarrow \DAinfty\) is an \((A,x,\delta)\)-heat function.
        \item\label{Item::Proof::HeatFuncs::DerivOfHeatIsHeat::Deriv} \(\forall j\in \N\), there exists an \((A,x,\delta)\)-heat function \(v_j:I\rightarrow \DAinfty\)
            such that \(\partial_t^j\left( u(t)\big|_{\MetricBall{x}{\delta}} \right)=v_j(t)\big|_{\MetricBall{x}{\delta}}\).
    \end{enumerate}
\end{lemma}
\begin{proof}
    \ref{Item::Proof::HeatFuncs::DerivOfHeatIsHeat::ApplyA}:  Using Definition \ref{Defn::Assump::Hypo::HeatFunction} \ref{Item::Assump::Hypo::HeatFunction::Derivatives},
    we have
    \begin{equation*}
        \partial_t^j \left( \delta^{2\kappa} A u(t)\big|_{\MetricBall{x}{\delta}}  \right) = \partial_t^{j+1}\left(  u(t)\big|_{\MetricBall{x}{\delta}} \right)
        =\left( \left( \delta^{2\kappa} A \right)^{j+1} u(t)\big|_{\MetricBall{x}{\delta}}\right)=\left( \left( \delta^{2\kappa} A \right)^{j} \delta^{2\kappa} Au(t)\big|_{\MetricBall{x}{\delta}}\right).
    \end{equation*}
    Clearly, \(\delta^{2\kappa}A u\) satisfies Definition \ref{Defn::Assump::Hypo::HeatFunction} \ref{Item::Assump::Hypo::HeatFunction::SmoothExceptOnePoint}
    with the same \(t_0\) as \(u\).
    From here, \ref{Item::Proof::HeatFuncs::DerivOfHeatIsHeat::ApplyA} follows.

    \ref{Item::Proof::HeatFuncs::DerivOfHeatIsHeat::Deriv}: Take \(v_j(t)=\left( \delta^{2\kappa} A\right)^j u(t)\).  \(v_j\) is an \((A,x,\delta)\)-heat
    function by \ref{Item::Proof::HeatFuncs::DerivOfHeatIsHeat::ApplyA}, and \(\partial_t^j\left( u(t)\big|_{\MetricBall{x}{\delta}} \right)=v_j(t)\big|_{\MetricBall{x}{\delta}}\)
    by Definition \ref{Defn::Assump::Hypo::HeatFunction} \ref{Item::Assump::Hypo::HeatFunction::Derivatives}.
\end{proof}

\begin{lemma}\label{Lemma::Proof::HeatFuncs::Inducel::GetFactorial}
    There exists an admissible constant \(R_1\geq 1\) such that the following holds.
    Let \(\delta_1\in (0,\delta_0)\) and \(u:(-1,1)\rightarrow \DAinfty\) be an \((A,x_0,\delta_1)\)-heat function.
    Then, \(\forall l\in \N\),
    \begin{equation*}
        \LpNorm{\partial_t^l u}{p_0}[(-1/2,1/2)\times \MetricBall{x_0}{\delta_1/2}]
        \lesssim R_1^l \left( l! \right)^{2\kappa} \LpNorm{u}{p_0}[(-1,1)\times \MetricBall{x_0}{\delta_1}].
    \end{equation*}
\end{lemma}
\begin{proof}
    The result is trivial for \(l=0\), so we assume \(l\geq 1\).

    Lemma \ref{Lemma::Proof::HeatFuncs::DerivOfHeatIsHeat} \ref{Item::Proof::HeatFuncs::DerivOfHeatIsHeat::Deriv}
    shows that for every \(k=1,2,\ldots,l\), \(\partial_t^{k-1}\left( u(t)\big|_{\MetricBall{x_0}{\delta_1}} \right)\)
    agrees with an \((A,x,\delta_1)\) heat function. We may therefore apply Lemma \ref{Lemma::Proof::HeatFuncs::Inducel::Firstl}
    to see that there exists an admissible constant \(B_0\geq 1\) with
    \begin{equation}\label{Eqn::Proof::HeatFuncs::Inducel::GetFactorial::Tmp1}
        \begin{split}
            &\LpNorm{\partial_t^k u}{p_0}[(-1+(k+1)/2(l+1), 1-(k+1)/2(l+1))\times \MetricBall{x_0}{\delta_1(1-(k+1)/2(l+1)}]
            \\&\leq l^{2\kappa} B_0\LpNorm{\partial_t^{k-1} u}{p_0}[(-1+k)/2(l+1), 1-k/2(l+1))\times \MetricBall{x_0}{\delta_1(1-k/2(l+1)}].
        \end{split}
    \end{equation}
    Applying \eqref{Eqn::Proof::HeatFuncs::Inducel::GetFactorial::Tmp1} \(l\) times and using Stirling's formula
    gives
    \begin{equation*}
        \begin{split}
        &\LpNorm{\partial_t^l u}{p_0}[(-1+1/2, 1-1/2)\times \MetricBall{x_0}{\delta_1/2}]
        \leq B_0^l l^{2\kappa l}
        \LpNorm{u}{p_0}[(-1,1)\times \MetricBall{x_0}{\delta_1}]
         \lesssim (B_0 e^{2\kappa})^l (l!)^{2\kappa}\LpNorm{u}{p_0}[(-1,1)\times \MetricBall{x_0}{\delta_1}].
        \end{split}
    \end{equation*}
\end{proof}

\begin{lemma}\label{Lemma::Proof::HeatFuncs::SupBoundWithoutExponential}
    There exists an admissible constant \(R_1\geq 1\) such that the following holds.
    Let \(\delta_1\in (0,\delta_0)\) and \(u:(-1,1)\rightarrow \DAinfty\) be an \((A,x_0,\delta_1)\)-heat function.
    Then,
    \begin{equation}\label{Eqn::Proof::HeatFuncs::SupBoundWithoutExponential::Qualitative}
        Xu(t)\big|_{(-a_1,a_1)\times \MetricBall{x_0}{a_1\delta_1}} \in \CinftySpace[(-a_1,a_1)][\CSpace{\MetricBall{x_0}{a_1\delta_1}}]
    \end{equation}
    and \(\forall l\in \N\),
    \begin{equation}\label{Eqn::Proof::HeatFuncs::SupBoundWithoutExponential::Quantitative}
    \begin{split}
         \sup_{\substack{t\in (-a_1,a_1)\\ x\in \MetricBall{x_0}{a_1\delta_1}}} \left| \partial_t^l Xu(t,x) \right|
         \lesssim S_X(\delta_1)^{-1} \measure\left( \MetricBall{x_0}{\delta_1} \right)^{-1/p_0} R_1^l (l!)^{2\kappa}
         \LpNorm{u}{p_0}[(-1,1)\times \MetricBall{x_0}{\delta_1}].
    \end{split}
    \end{equation}
\end{lemma}
\begin{proof}
    \eqref{Eqn::Proof::HeatFuncs::SupBoundWithoutExponential::Qualitative} follows from Assumption \ref{Assumption::HypoellipticityI} \ref{Item::Assumption::HypoellipticityI::Qualitative}.

    Repeated applications of Lemma \ref{Lemma::Proof::HeatFuncs::DerivOfHeatIsHeat} \ref{Item::Proof::HeatFuncs::DerivOfHeatIsHeat::ApplyA}
    show \(\left( \delta_1^{2\kappa }A \right)^j u:I\rightarrow \DAinfty\) is an \((A,x_0,\delta_1)\)-heat function,
    \(\forall j\in \N\).
    We claim, for \(0\leq j\leq l\),
    \begin{equation}\label{Eqn::Proof::HeatFuncs::SupBoundWithoutExponential::Tmp1}
         \partial_t^{l} X u(t,x)\big|_{(-a_1,a_1)\times \MetricBall{x_0}{a_1\delta_1}} =\partial_t^{l-j} X \left( \delta_1^{2\kappa} A \right)^j u(t,x)\big|_{(-a_1,a_1)\times \MetricBall{x_0}{a_1\delta_1}}.
    \end{equation}
    We establish \eqref{Eqn::Proof::HeatFuncs::SupBoundWithoutExponential::Tmp1} by induction on \(j\).
    The base case, \(j=0\), is trivial.  We assume the result for \(0\leq j<l-1\).
    Using the inductive hypothesis, that \(\left( \delta_1^{2\kappa }A \right)^j u:I\rightarrow \DAinfty\) is an \((A,x,\delta_1)\)-heat function,
    and Assumption \ref{Assumption::HypoellipticityI} \ref{Item::Assumption::HypoellipticityI::CommuteDerivs},
    we see
    \begin{equation*}
    \begin{split}
         \partial_t^{l} X u(t,x)\big|_{(-a_1,a_1)\times \MetricBall{x_0}{a_1\delta_1}} 
         &=\partial_t^{l-j} X \left( \delta_1^{2\kappa} A \right)^j u(t,x)\big|_{(-a_1,a_1)\times \MetricBall{x_0}{a_1\delta_1}}
         \\&=\partial_t^{l-j-1} X \left( \delta_1^{2\kappa} A \right)^{j+1} u(t,x)\big|_{(-a_1,a_1)\times \MetricBall{x_0}{a_1\delta_1}}.
    \end{split}
    \end{equation*}
    This completes the inductive step and proof of \eqref{Eqn::Proof::HeatFuncs::SupBoundWithoutExponential::Tmp1}.

    By taking \(j=l\) in \eqref{Eqn::Proof::HeatFuncs::SupBoundWithoutExponential::Tmp1}, we have
    \begin{equation}\label{Eqn::Proof::HeatFuncs::SupBoundWithoutExponential::Tmp2}
        \sup_{\substack{t\in (-a_1,a_1)\\ x\in \MetricBall{x_0}{a_1\delta_1}}} \left| \partial_t^l Xu(t,x) \right|
        =\sup_{\substack{t\in (-a_1,a_1)\\ x\in \MetricBall{x_0}{a_1\delta_1}}} \left| X \left( \delta_1^{2\kappa} A \right)^l  u(t,x) \right|
    \end{equation}
    Since \(\left( \delta_1^{2\kappa }A \right)^l u:I\rightarrow \DAinfty\) is an \((A,x_0,\delta_1)\)-heat function,
    Assumption \ref{Assumption::HypoellipticityI} \ref{Item::Assumption::HypoellipticityI::Quantitative} shows
    \begin{equation}\label{Eqn::Proof::HeatFuncs::SupBoundWithoutExponential::Tmp3}
        \sup_{\substack{t\in (-a_1,a_1)\\ x\in \MetricBall{x_0}{a_1\delta_1}}} \left| X \left( \delta_1^{2\kappa} A \right)^l  u(t,x) \right|
        \leq 
        S_X(\delta_1)^{-1} \measure(\MetricBall{x_0}{\delta_1})^{-1/p_0} \LpNorm{(\delta_1^{2\kappa}A)^l u}{p_0}[(-1/2,1/2)\times \MetricBall{x_0}{\delta_1/2}].
    \end{equation}
    Definition \ref{Defn::Assump::Hypo::HeatFunction} \ref{Item::Assump::Hypo::HeatFunction::Derivatives} shows
    \begin{equation}\label{Eqn::Proof::HeatFuncs::SupBoundWithoutExponential::Tmp4}
    \begin{split}
            &S_X(\delta_1)^{-1} \measure(\MetricBall{x_0}{\delta_1})^{-1/p_0} \LpNorm{(\delta_1^{2\kappa}A)^l u}{p_0}[(-1/2,1/2)\times \MetricBall{x_0}{\delta_1/2}]
            \\&=S_X(\delta_1)^{-1} \measure(\MetricBall{x_0}{\delta_1})^{-1/p_0} \LpNorm{\partial_t^l u}{p_0}[(-1/2,1/2)\times \MetricBall{x_0}{\delta_1/2}].
    \end{split}    \end{equation}
    Finally, Lemma \ref{Lemma::Proof::HeatFuncs::Inducel::GetFactorial} shows
    \begin{equation}\label{Eqn::Proof::HeatFuncs::SupBoundWithoutExponential::Tmp5}
    \begin{split}
            &S_X(\delta_1)^{-1} \measure(\MetricBall{x_0}{\delta_1})^{-1/p_0} \LpNorm{\partial_t^l u}{p_0}[(-1/2,1/2)\times \MetricBall{x_0}{\delta_1/2}]
            \\&\lesssim S_X(\delta_1)^{-1} \measure(\MetricBall{x_0}{\delta_1})^{-1/p_0}R_1^{l} (l!)^{2\kappa}\LpNorm{u}{p_0}[(-1,1)\times \MetricBall{x_0}{\delta_1}].
    \end{split}    \end{equation}
    
    Combining \eqref{Eqn::Proof::HeatFuncs::SupBoundWithoutExponential::Tmp2},
    \eqref{Eqn::Proof::HeatFuncs::SupBoundWithoutExponential::Tmp3},
    \eqref{Eqn::Proof::HeatFuncs::SupBoundWithoutExponential::Tmp4},
    and \eqref{Eqn::Proof::HeatFuncs::SupBoundWithoutExponential::Tmp5} completes the proof.
\end{proof}

\begin{proof}[Proof of Proposition \ref{Prop::Proof::HeatFuncs::MainExpBound}]
  In light of Lemma \ref{Lemma::Proof::HeatFuncs::SupBoundWithoutExponential}, 
  there is an admissible constant \(A_0\geq 1\) such that
  the conditions of
  Proposition \ref{Prop::Proof::Gevrey::MainProp} hold
  for \(Xu(t,x)\), \(t\in (-a_1,a_1)\), \(x\in \MetricBall{x_0}{\delta_1}\),
  with \(C\), \(R\), and \(\gamma\) replaced by \(A_0 S_X(\delta_1)^{-1} \measure\left( \MetricBall{x_0}{\delta_1} \right)^{-1/p_0}\LpNorm{u}{p_0}[(-1,1)\times \MetricBall{x_0}{\delta_1}]\), \(R_1\), and \(2\kappa\), respectively.
  The result follows.
\end{proof}

    \subsection{Operator bounds}\label{Section::Proof::OperatorBounds}
    In this section, we study the operator \(\OpQx(t)\) from Theorem \ref{Thm::Results::DefineOperator}.
In particular, we provide estimates on its integral kernel and prove Theorem \ref{Thm::Results::DefineOperator}.

\begin{notation}
    When given a function \(L(t,x,y)\) of three variables \(t\in I\subseteq \R\), \(x\in U\subseteq \MetricSpace\), and \(y\in V\subseteq \MetricSpace\),
    we write \(L\in \CinftySpacet[I][\CSpacex{U}[\LpSpacewy{p}[V]]]\) to mean that \(L\) is a \(\CinftySpace\) function in the \(t\)-variable,
    taking values in the space of functions which are continuous functions in \(U\) (in the \(x\)-variable) taking values in \(\LpSpace{p}[V]\), where \(\LpSpace{p}[V]\)
    is given the weak topology. We similarly define \(\CinftySpacet[I][\CSpacey{V}[\LpSpacewx{p}[U]]]\) with the roles of \(x\) and \(y\) reversed.
\end{notation}

\begin{definition}
    For \(t\geq 0\), \(x_1\in \MetricSpace\), \(\delta>0\), we define \(\OpQx[x_1][\delta](t):\DAinfty\rightarrow \LpSpace{p_0}[\MetricBall{x_1}{\delta}]\) by
    \begin{equation*}
        \OpQx[x_1][\delta](t) \phi:= XT(t)\phi\big|_{\MetricBall{x_1}{\delta}}.
    \end{equation*}
    We similarly define \(\OpQy[y_1][\delta](t):\DAsinfty\rightarrow \LpSpace{p_0'}[\MetricBall{y_1}{\delta}]\)
    by
    \begin{equation*}
        \OpQy[y_1][\delta](t) \phi:= YT(t)^{*}\phi\big|_{\MetricBall{y_1}{\delta}}.
    \end{equation*}
\end{definition}

\begin{lemma}\label{Lemma::Proof::OperatorBounds::ReiszRepn}
    Let \(I\subseteq \R\) be an open interval, \(N\) a metric space, \(\BanachSpaceY\) a Banach space,
    \(\BanachSpace\subseteq \BanachSpaceY\) a closed subspace,
    and \(D\subseteq \BanachSpaceY\) a subspace such that \(\overline{D}\supseteq \BanachSpace\).
    For \(t\in I\), let \(Q(t):D \rightarrow \CSpace{N}\) be a linear map
    such that
    \begin{enumerate}[(i)]
        \item\label{Item::Proof::OperatorBounds::ReiszRepn::QtvSmooth} \(Q(t)v(x)\in \CinftySpace[I][\CSpace{N}]\), \(\forall v\in D\).
        \item\label{Item::Proof::OperatorBounds::ReiszRepn::AprioriBounds} \(\forall l\in \N\), \(\exists C_l\geq 0\), \(\forall v\in D\),
            \begin{equation}\label{Eqn::Proof::OperatorBounds::ReiszRepn::APrioriBound}
                \sup_{\substack{x\in N \\ t\in I}} \left| Q(t)v(x) \right| \leq C_l \BanachSpaceNorm{v}.
            \end{equation}
    \end{enumerate}
    Then, for each \(t\in I\), \(Q(t)\) extends to a unique continuous map \(\BanachSpace\rightarrow \CSpace{N}\),
    which we also denote by \(Q(t)\).
    There exists a unique \(L(t,x)\in \CinftySpacet[I][\CSpacex{N}[\BanachSpaceDualwstar]]\),
    where \(\BanachSpaceDualwstar\) denotes the dual of \(\BanachSpace\) given the weak-\(*\) topology,\footnote{In what
    follows, \(\BanachSpace\) will be taken to be either \(\LpSpace{p_0}\) or \(\LpSpace{p_0'}\) and is therefore
    reflexive. Thus, the weak-\(*\) topology agrees with the weak topology.}
    satisfying
    \(Q(t) v(x) = \BanachSpacePairing{L(t,x)}{v}\), \(\forall v\in \BanachSpace\), and
    \begin{equation}\label{Eqn::Proof::OperatorBounds::ReiszRepn::APosterioriBound}
        \sup_{\substack{x\in N \\ t\in I}}  \BanachSpaceDualNorm{\partial_t^l L(t,x)} \leq C_l,\quad \forall l\in \N.
    \end{equation}
\end{lemma}
\begin{proof}
    That \(Q(t)\) extends to a unique continuous map \(\overline{D}\rightarrow \CSpace{N}\), \(\forall t\in I\),
    follows immediately from \eqref{Eqn::Proof::OperatorBounds::ReiszRepn::APrioriBound}.
    Moreover \ref{Item::Proof::OperatorBounds::ReiszRepn::QtvSmooth} and \ref{Item::Proof::OperatorBounds::ReiszRepn::AprioriBounds}
    hold
    with \(D\) replaced by \(\overline{D}\) (with the same choice of \(C_l\) in \ref{Item::Proof::OperatorBounds::ReiszRepn::AprioriBounds}).
    By restricting to \(\BanachSpace\), we see \(Q(t)\) extends to a unique continuous map \(\BanachSpace\rightarrow \CSpace{N}\), \(\forall t\in I\),
    and \ref{Item::Proof::OperatorBounds::ReiszRepn::QtvSmooth} and \ref{Item::Proof::OperatorBounds::ReiszRepn::AprioriBounds}
    hold
    with \(D\) replaced by \(\BanachSpace\) (with the same choice of \(C_l\) in \ref{Item::Proof::OperatorBounds::ReiszRepn::AprioriBounds}).

    
    For each \((t,x)\in I\times N\), \(v\mapsto Q(t)v(x)\) is a continuous linear functional
    on \(\BanachSpace\). Thus, there exists a unique \(L(t,x)\in \BanachSpaceDual\)
    with \(Q(t) v(x) = \BanachSpacePairing{L(t,x)}{v}\), \(\forall v\in \BanachSpace\)
    (where \(\BanachSpacePairing{\cdot}{\cdot}\) denotes the pairing of \(\BanachSpace\)
    and \(\BanachSpaceDual\)).
    Since \(Q(t)v(x)\in \CinftySpace[I][\CSpacex{N}]\), \(\forall v\in \BanachSpace\),
    we conclude \(L(t,x)\in \CinftySpacet[I][\CSpacex{N}[\BanachSpaceDualwstar]]\).

    Finally, \eqref{Eqn::Proof::OperatorBounds::ReiszRepn::APosterioriBound}
    follows from \eqref{Eqn::Proof::OperatorBounds::ReiszRepn::APrioriBound} by taking the supremum
    of \eqref{Eqn::Proof::OperatorBounds::ReiszRepn::APrioriBound}
    over
    \(\BanachSpaceNorm{v}=1\).
\end{proof}

\begin{lemma}\label{Lemma::Proof::OperatorBounds::OnDiagonalOperators}
    Let \(\delta_2\in (0,\delta_0)\) and \(t_0\geq 2\delta_2^{2\kappa}\). Then, for \(t\in (t_0/2-a_1\delta_2^{2\kappa}, t_0/2+a_1\delta_2^{2\kappa})\),
    \(\OpQx[x_0][a_1\delta_2](t)\), initially defined as an operator \(\DAinfty\rightarrow \LpSpace{p_0}[\MetricBall{x_0}{a_1\delta_2}][\measure]\),
    extends to a continuous operator \(\LpSpace{p_0}[\MetricSpace][\measure]\rightarrow \CSpace{\MetricBall{x_0}{a_1\delta_2}}\).
    There exists a unique function
    \begin{equation*}
        XK_t(x,y)\in \CinftySpacet*[(t_0/2-a_1\delta_2^{2\kappa},t_0/2+a_1\delta_2^{2\kappa})][\CSpacex*{\MetricBall{x_0}{a_1\delta_2}}[\LpSpacewy{p_0'}[\MetricSpace][\measure]]]
    \end{equation*}
    such that
    \begin{equation*}
        \OpQx[x_0][a_1\delta_2](t)\phi(x) = \int XK_t(x,y)\phi(y)\: dy,\quad \forall \phi\in \LpSpace{p_0}[\MetricSpace][\measure],  x\in \MetricBall{x_0}{a_1\delta_2}.
    \end{equation*}
    This function satisfies, \(\forall l\in \N\),
    \begin{equation*}
        \sup_{\substack{x\in \MetricBall{x_0}{a_1\delta_2}\\ t\in (t_0/2-a_1\delta_2,t_0/2+a_1\delta_2)}} \LpNorm{ (\delta_2^{2\kappa} \partial_t)^l XK_t(x,\cdot)}{p_0'}[\MetricSpace]
        \lesssim_l S_X(\delta_2)^{-1} \measure(\MetricBall{x_0}{\delta_2})^{-1/p_0}.
    \end{equation*}
\end{lemma}
\begin{proof}
    Fix \(\phi\in \DAinfty\) and consider the function \(u:(-1,1)\rightarrow \DAinfty\) defined by
    \(u(t)=T(t_0/2+t\delta_2^{2\kappa})\phi\); Lemma \ref{Lemma::HeatFuncs::Characterize} \ref{Item::HeatFuncs::Characterize::PositiveReals}
    shows \(u\) is an \((A,x_0,\delta_2)\)-heat function.
    Lemma \ref{Lemma::Proof::HeatFuncs::SupBoundWithoutExponential}
    shows \eqref{Eqn::Proof::HeatFuncs::SupBoundWithoutExponential::Qualitative} and \eqref{Eqn::Proof::HeatFuncs::SupBoundWithoutExponential::Quantitative} hold
    (with \(\delta_1\) replaced with \(\delta_2\)).
    Moreover, using \eqref{Eqn::Proof::HeatFuncs::SupBoundWithoutExponential::Quantitative}, 
    we have
    \begin{equation}\label{Eqn::Proof::OperatorBounds::OnDiagonalOperators::Tmp1}
    \begin{split}
         &\sup_{\substack{x\in \MetricBall{x_0}{a_1\delta_2}\\ t\in (t_0/2-a_1\delta_2,t_0/2+a_1\delta_2)}} \left| (\delta_2^{2\kappa} \partial_t)^l\OpQx[x_0][a_1\delta_2](t)\phi(t,x) \right|
         =\sup_{\substack{x\in \MetricBall{x_0}{a_1\delta_2}\\ t\in (-a_1,a_1)}} \left|\partial_t^l Xu(t,x) \right|
         \\&\lesssim_l S_X(\delta_2)^{-1} \measure\left( \MetricBall{x_0}{\delta_2} \right)^{-1/p_0} 
         \LpNorm{u}{p_0}[(-1,1)\times \MetricBall{x_0}{\delta_2}]
         \\&= S_X(\delta_2)^{-1} \measure\left( \MetricBall{x_0}{\delta_2} \right)^{-1/p_0} 
         \LpNorm{T(t_0/2+t\delta_2^{2\kappa})\phi}{p_0}[(-1,1)\times \MetricBall{x_0}{\delta_2}]
         \\&\lesssim S_X(\delta_2)^{-1} \measure\left( \MetricBall{x_0}{\delta_2} \right)^{-1/p_0} 
         \LpNorm{\phi}{p_0}[\MetricSpace].
    \end{split}
    \end{equation}
    Since \(\DAinfty\) is dense in \(\LpSpace{p_0}[\MetricSpace][\mu]\) (combine  \cite[Theorem 1.4 and Proposition 1.8 in Chapter 2]{EngelNagelAShortCourseOnOperatorSemigroups}),
    the conditions from Lemma \ref{Lemma::Proof::OperatorBounds::ReiszRepn}
    are established
    for \(Q(t)=\OpQx[x_0][a_1\delta_2](t)\), \(D=\Domain[A^\infty]\), \(\BanachSpace=\BanachSpaceY=\LpSpace{p_0}[\MetricSpace][\mu]\),
    \(I=(t_0/2-a_1\delta_2^{2\kappa},t_0/2+a_1\delta_2^{2\kappa})\),
    and \(N=\MetricBall{x_0}{a_1\delta_2}\),
    and the result follows from  Lemma \ref{Lemma::Proof::OperatorBounds::ReiszRepn} (using \eqref{Eqn::Proof::OperatorBounds::OnDiagonalOperators::Tmp1} to determine \(C_l\) in that lemma).
\end{proof}

\begin{proof}[Proof of Theorem \ref{Thm::Results::DefineOperator}]
    The result for \(\OpQx(t)\) is an immediate consequence of Lemma \ref{Lemma::Proof::OperatorBounds::OnDiagonalOperators}.
    The result for \(\OpQy(t)\) follows  by symmetry
    (see Remark \ref{Rmk::Assumption::Symmetric}).
\end{proof}

\begin{lemma}\label{Lemma::Proof::OperatorBounds::OffDiagonalOperators}
    Let \(\delta_1\in (0,\delta_0)\). Then, for \(t\in (0,a_1\delta_1^{2\kappa})\),
    \(\OpQx[x_0][a_1\delta_1](t)\) extends to a continuous operator
    \(\LpSpace{p_0}[\MetricSpace\setminus \MetricBallClosure{x_0}{\delta_1}]\rightarrow \CSpace{\MetricBall{x_0}{a_1\delta_1}}\).
    There exists a unique function
    \begin{equation*}
        XK_t(x,y)\in \CinftySpacet*[(0,a_1\delta_1^{2\kappa})][\CSpacex*{\MetricBall{x_1}{a_1\delta_1}}[\LpSpacewy{p_0'}[\MetricSpace\setminus \MetricBallClosure{x_0}{\delta_1}][\measure]]]
    \end{equation*}
    such that for \(t\in (0,a_1\delta_1^{2\kappa})\) and \(x_1\in \MetricBall{x_1}{a_1\delta_1}\),
    \begin{equation*}
        \OpQx[x_0][a_1\delta_1](t)f(x)=\int XK_t(x,y) f(y)\: d\measure(y), \quad \forall f\in \LpSpace{p_0}[\MetricSpace\setminus\MetricBallClosure{x_0}{\delta_1}].
    \end{equation*}
    There exists an admissible constant \(c_1>0\) such that
    this function satisfies, \(\forall j\in \N\),
    \begin{equation*}
        \sup_{x\in \MetricBall{x_0}{a_1\delta_1} \\ t\in (0,a_1 \delta^{2\kappa})} \LpNorm{\partial_t^j XK_t(x,\cdot)}{p_0'}[\MetricSpace\setminus \MetricBallClosure{x_0}{\delta_1}]
        \lesssim_j \delta_1^{-2\kappa j} S_X(\delta_1)^{-1} \measure(\MetricBall{x_0}{\delta_1})^{-1/p_0} \exp\left( -c_1  \delta_1^{2\kappa/(2\kappa-1)} t^{-1/(2\kappa-1)} \right).
    \end{equation*}
\end{lemma}
\begin{proof}
    Let \(\phi\in \DAinfty\) be such that \(A^j\phi\big|_{\MetricBall{x_0}{\delta_1}}=0\), \(\forall j\in \N\); by Assumption
    \ref{Assumption::Locality}, such \(\phi\) can approximate any element of \(\LpSpace{p_0}[\MetricSpace\setminus \MetricBallClosure{x_0}{\delta_1}]\)
    in \(\LpSpace{p_0}[\MetricSpace]\).
    Define \(u:(-1,1)\rightarrow \DAinfty\) by
    \begin{equation*}
        u(t):=
        \begin{cases}
            T(\delta_1^{2\kappa}t) \phi & t\geq 0,\\
            0 & t<0.
        \end{cases}
    \end{equation*}
    Lemma \ref{Lemma::HeatFuncs::Characterize} \ref{Item::HeatFuncs::Characterize::AllReals} shows \(u\) is an \((A,x_0,\delta_1)\)-heat function.
    Proposition \ref{Prop::Proof::HeatFuncs::MainExpBound} gives that \(Xu\big|_{(-a_1,a_1)\times \MetricBall{x_0}{a_1\delta_1}} \in\CinftySpace[(-a_1,a_1)][\CSpace{\MetricBall{x_0}{a_1\delta_1}}]\)
    and using \eqref{Eqn::Proof::HeatFunc::MainExpBound} we have for \(t\in (0,a_1)\) and \(l\in \N\),
    \begin{equation}\label{Eqn::Proof::OperatorBounds::OffDiagonalOperators::Tmp1}
    \begin{split}
         & \sup_{x\in \MetricBall{x_0}{\delta_1}} \left| \partial_t^l \OpQx[x_0][a_1\delta_1](\delta_1^{2\kappa}t)\phi(x)  \right|
         = \sup_{x\in \MetricBall{x_0}{\delta_1}} \left| \partial_t^l X u(t,x) \right|
         \\&\lesssim_l S_X(\delta_1)^{-1} \measure(\MetricBall{x_0}{\delta_1})^{-1/p_0} \exp\left( -c_1 t^{-1/(2\kappa-1)} \right)\LpNorm{u}{p_0}[(-1,1)\times \MetricBall{x_0}{\delta_1}]
         \\&=S_X(\delta_1)^{-1} \measure(\MetricBall{x_0}{\delta_1})^{-1/p_0} \exp\left( -c_1 t^{-1/(2\kappa-1)} \right)\LpNorm{T(\delta_1^{2\kappa}t)\phi}{p_0}[(-1,1)\times \MetricBall{x_0}{\delta_1}]
         \\&\lesssim S_X(\delta_1)^{-1} \measure(\MetricBall{x_0}{\delta_1})^{-1/p_0} \exp\left( -c_1 t^{-1/(2\kappa-1)} \right)\LpNorm{\phi}{p_0}[\MetricSpace].
    \end{split}
    \end{equation}
    We have established the conditions of 
    Lemma \ref{Lemma::Proof::OperatorBounds::ReiszRepn}
    with \(Q(t)=\OpQx[x_0][a_1\delta_1](\delta_1^{2\kappa}t)\),
    \(D=
    \left\{\phi\in \DAinfty: A^j\phi\big|_{\MetricBall{x_0}{\delta_1}}=0,\forall j\in \N \right\}\),
    \(\BanachSpace=\LpSpace{p_0}[\MetricSpace\setminus \MetricBallClosure{x_0}{\delta_1}]\),
    \(\BanachSpaceY=\LpSpace{p_0}[\MetricSpace]\),
    \(I=(0,a_1)\), and \(N=\MetricBall{x_1}{a_1\delta_1}\).
    The result now follows from
    Lemma \ref{Lemma::Proof::OperatorBounds::ReiszRepn} 
    with \(t\) replaced by \(\delta^{-2\kappa}t\)
    (using \eqref{Eqn::Proof::OperatorBounds::OffDiagonalOperators::Tmp1} to determine \(C_l\) in that lemma).
\end{proof}

\begin{remark}\label{Lemma::Proof::OperatorBounds::XKtWellDefined}
    We have used the same name for the function \(XK_t(x,y)\) in Lemmas \ref{Lemma::Proof::OperatorBounds::OnDiagonalOperators}
    and \ref{Lemma::Proof::OperatorBounds::OffDiagonalOperators}, however there is no ambiguity since these functions agree on their common domain
    due to their uniqueness.
\end{remark}

    \subsection{Pointwise Bounds}
    In this section, we complete the proof of Theorem \ref{Thm::Results::MainThm}, by proving pointwise bounds on the kernel
\(X\Yb K_t(x,y)\). 
In Section \ref{Section::Proof::OperatorBounds}, we established properties of the function \(XK_t(x,y)\) which is, by definition,
the integral kernel of \(\OpQx(t)\) (which is the restriction of \(XT(t)\) near \(x_0\)).
As described in Remark \ref{Rmk::Assumption::Symmetric}, our assumptions are symmetric in \(x\) and \(y\), and so we obtain
analogous estimates for \(\OpQy(t)\) (which is the restriction of \(YT(t)^{*}\) near \(y_0\)).
We denote by \(\Yb K_t(x,y)\) the complex conjugate of integral kernel of \(\OpQy(t)\) acting in the \(x\)-variable, so that if \(f\in \LpSpace{p_0'}[\MetricSpace]\),
then
\begin{equation*}
    \OpQy(t) f (y)=\int \overline{\Yb K_t(x,y)}  f(x)\: d\mu(x).
\end{equation*}
By definition, \(XT(t)Y^{*}=\OpQx(t/2) \OpQy(t/2)^{*}\) (see the discussion in Section \ref{Section::Results}).
Thus, we have, for \(x\) near \(x_0\) and \(y\) near \(y_0\),
\begin{equation}\label{Eqn::Proof::PtwiseBounds::XYbKtFormula}
    X\Yb K_t(x,y)=\int XK_{t/2}(x,z) \Yb K_{t/2}(z,y)\: d\measure(z).
\end{equation}
We will use \eqref{Eqn::Proof::PtwiseBounds::XYbKtFormula} to prove the properties described in Theorem \ref{Thm::Results::MainThm}.

\begin{proposition}\label{Prop::Proof::PtwiseBounds::PreliminaryEstimate}
    Fix \(t_0>0\) and let \(\delta_1:=(t_0/2)^{1/2\kappa} \wedge \delta_0\).
    Then,
    \begin{enumerate}[(i)]
        \item\label{Item::Proof::PtwiseBound::SmoothInt} For each fixed \((x,y)\in \MetricBall{x_0}{a_1\delta_1}\times \MetricBall{y_0}{a_1\delta_1}\), the map
            \(t\mapsto X\Yb K_t(x,y)\) is smooth for \(t\in (t_0-2a_1\delta_1^{2\kappa}, t_0+2a_1\delta_1^{2\kappa})\).
        \item\label{Item::Proof::PtwiseBound::ContinuousForFixedy} For each \(j\in \N\), and each fixed \(y\in \MetricBall{y_0}{a_1\delta_1}\), the map
            \((t,x)\mapsto \partial_t^j X\Yb K_t(x,y)\) is continuous for \((t,x)\in (t_0-2a_1\delta_1^{2\kappa}, t_0+2a_1\delta_1^{2\kappa})\times \MetricBall{x_0}{a_1\delta_1}\).
        \item\label{Item::Proof::PtwiseBound::ContinuousForFixedx} \ref{Item::Proof::PtwiseBound::ContinuousForFixedy} remains true with the roles of \(x\) and \(y\) reversed.
    \end{enumerate}
    Also, \(\forall j\in \N\),
    \begin{equation}\label{Eqn::Proof::PtwiseBound::BoundKt}
        \sup_{\substack{t\in (t_0-a_1\delta_1^{2\kappa}, t_0+a_1\delta_1^{2\kappa})\\ x\in \MetricBall{x_0}{a_1\delta_1} \\ y\in \MetricBall{y_0}{a_1\delta_1}}}
        \left| \partial_t^j X\Yb K_t(x,y) \right|
        \lesssim_j S_X(\delta_1)^{-1} S_Y(\delta_1)^{-1} \delta_1^{-2\kappa j}\measure(\MetricBall{x_0}{\delta_1})^{-1/p_0} \measure(\MetricBall{y_0}{\delta_1})^{-1/p_0'}.
    \end{equation}
\end{proposition}

Before we prove Proposition \ref{Prop::Proof::PtwiseBounds::PreliminaryEstimate}, we address one subtlety.
Let \(\BanachSpace\) be a reflexive Banach space and \(\BanachSpaceDual\) its dual, and let
\(\BanachSpacew\) and \(\BanachSpaceDualw\) denote these spaces with the weak topology.
Let \(\BanachSpacePairing{\cdot}{\cdot}\) denote the pairing of \(\BanachSpace\) with \(\BanachSpaceDual\).
Then, the map \((f,g)\mapsto \BanachSpacePairing{f}{g}\), mapping \(\BanachSpacew\times \BanachSpaceDualw\rightarrow \C\)
is separately continuous but not jointly continuous.  This is why Proposition \ref{Prop::Proof::PtwiseBounds::PreliminaryEstimate}
only establishes \(\partial_t^j X\Yb K_t(x,y)\) is separately continuous in \(x,y\) (the pairing of a Banach space with its
dual is in \eqref{Eqn::Proof::PtwiseBounds::XYbKtFormula}).
However, Proposition \ref{Prop::Proof::PtwiseBounds::PreliminaryEstimate} claims
\(\partial_t^j X\Yb K_t(x,y)\) is jointly continuous in \((t,x)\) and jointly continuous in \((t,y)\);
despite the fact that the variable \(t\) is in both terms in right-hand side of \eqref{Eqn::Proof::PtwiseBounds::XYbKtFormula}.
In the next two lemmas, we present one way to see this. Roughly speaking, the key is that everything is
smooth in \(t\), not merely continuous, and one can use this smoothness to establish the joint continuity.

\begin{lemma}\label{Lemma::Proof::PtwiseBounds::SmoothnessToImprovingParing}
    Let \(I\subseteq \R\) be an interval and \(N\) a metric space.
    Suppose \(K(t,x)\in \CinftySpacet[I][\CSpacex{N}[\BanachSpacew]]\)
    and \(L(t)\in \CinftySpace[I][\BanachSpaceDualw]\).  Suppose, \(\forall l\in \N\),
    \begin{equation}\label{Eqn::Proof::PtwiseBounds::SmoothnessToImprovingParing::AprioriBound}
        \sup_{\substack{t\in I \\ x\in N}} \BanachSpaceNorm*{\partial_t^l K(t,x)} + \sup_{t\in I} \BanachSpaceDualNorm*{\partial_t^l L_t}<\infty.
    \end{equation}
    Set \(M(t,x):=\BanachSpacePairing*{K(t,x)}{L(t)}\). Then,
    \begin{enumerate}[(i)]
        \item For each \(x\), \(t\mapsto M(t,x)\) is  smooth on its domain.
        \item For each \(l\in \N\), \((t,x)\mapsto \partial_t^l M(t,x)\) is continuous on its domain.
        \item\label{Item::Proof::PtwiseBounds::SmoothnessToImprovingParing::DerivFormula} \(\partial_{t} M(t,x)=\BanachSpacePairing{\partial_t K(t,x)}{L(t)}+\BanachSpacePairing{K(t,x)}{\partial_t L(t)}\).
    \end{enumerate}
\end{lemma}
\begin{proof}
    We will show \(M(t,x)\) is jointly continuous in \((t,x)\) differentiable in \(t\) and
    \ref{Item::Proof::PtwiseBounds::SmoothnessToImprovingParing::DerivFormula} holds.
    This will complete the proof since \ref{Item::Proof::PtwiseBounds::SmoothnessToImprovingParing::DerivFormula}
    shows
    \(\partial_t M(t,x)\) is a sum of two terms of the same form as \(M(t,x)\); and therefore
    \(\partial_t M(t,x)\) is jointly continuous in \((t,x)\) and differentiable in \(t\), and one obtains
    a formula for \(\partial_t^2M(t,x)\) which sees it as a sum of terms of the same form as \(M(t,x)\), etc. A simple induction then
    iterates this to establish the full result.

    Without loss of generality, we may assume \(I=(-a,a)\) for some \(a>0\).
    For \((t,x)\in I\times N\) fixed, we may write
    \begin{equation*}
        K(t,x)=K(0,x)+\int_0^t \partial_{s_1} K(s_1,x)\: ds_1,
    \end{equation*}
    where the Riemann sums of the integral converge in \(\BanachSpacew\). Thus,
    \begin{equation}\label{Eqn::Proof::PtwiseBounds::SmoothnessToImprovingParing::Tmp1}
        M(t,x)=\BanachSpacePairing{K(0,x)}{L(t)}+\int_0^t \BanachSpacePairing{\partial_{s_1} K(s_1,x)}{L(t)}\: ds_1.
    \end{equation}
    Similarly, \(L(t)=L(0)+\int_0^t \partial_{s_2} L(s_2)\: ds_2\). Plugging this into
    \eqref{Eqn::Proof::PtwiseBounds::SmoothnessToImprovingParing::Tmp1} shows
    \begin{equation}\label{Eqn::Proof::PtwiseBounds::SmoothnessToImprovingParing::Tmp2}
        \begin{split}
            M(t,x)=&\BanachSpacePairing{L(0,x)}{K(0)} + \int_0^t \BanachSpacePairing{K(0,x)}{\partial_{s_2}L(s_2)}\: ds_2
            \\&+\int_0^t \BanachSpacePairing{\partial_{s_1}K(s_1,x)}{L(0)}\: ds_1
            +\int_0^t \int_0^t \BanachSpacePairing{\partial_{s_1}K(s_1,x)}{\partial_{s_2}L(s_2)}\: ds_1\: ds_2.
        \end{split}
    \end{equation}
    We claim that each of the four terms on the right-hand side of \eqref{Eqn::Proof::PtwiseBounds::SmoothnessToImprovingParing::Tmp2}
    are jointly continuous in \((t,x)\).  The first and third terms are clear. We address the fourth, as it is the hardest;
    a similar proof establishes the joint continuity of the second term.

    Fix \((t_0,x_0)\in I\times N\). We wish to show
    \begin{equation}
        \begin{split}
            &\lim_{(t_0,x_0)\rightarrow (t,x)} \int_0^t \int_0^t \BanachSpacePairing{\partial_{s_1}K(s_1,x)}{\partial_{s_2}L(s_2)}\: ds_1\: ds_2
            \\&=\int_0^{t_0} \int_0^{t_0} \BanachSpacePairing{\partial_{s_1}K(s_1,x_0)}{\partial_{s_2}L(s_2)}\: ds_1\: ds_2.
        \end{split}
    \end{equation}
    Set \(C:=\sup_{\substack{t\in I \\ x\in N}} \BanachSpaceNorm*{\partial_t^l K(t,x)} + \sup_{t\in I} \BanachSpaceDualNorm*{\partial_t^l L_t}\);
    by assumption \(C<\infty\).
    We have,
    \begin{equation}\label{Eqn::Proof::PtwiseBounds::SmoothnessToImprovingParing::Tmp3}
    \begin{split}
         &\left| \int_0^t \int_0^t \BanachSpacePairing{\partial_{s_1}K(s_1,x)}{\partial_{s_2}L(s_2)}\: ds_1\: ds_2 - \int_0^{t_0} \int_0^{t_0} \BanachSpacePairing{\partial_{s_1}K(s_1,x_0)}{\partial_{s_2}L(s_2)}\: ds_1\: ds_2 \right|
         \\&\leq 2 a C^2 |t-t_0| + \int_0^{t_0}\int_0^{t_0}\left| \BanachSpacePairing{\partial_{s_1}K(s_1,x)-\partial_{s_1}K(s_1,x_0)}{\partial_{s_2}L(s_2)}  \right|\: ds_1\: ds_2.
    \end{split}
    \end{equation}
    The first term on the right-hand side of \eqref{Eqn::Proof::PtwiseBounds::SmoothnessToImprovingParing::Tmp3} clearly tends
    to \(0\) as \((t,x)\rightarrow (t_0,x_0)\), while the second term tends to \(0\) by \eqref{Eqn::Proof::PtwiseBounds::SmoothnessToImprovingParing::AprioriBound}
    and the dominated convergence theorem.

    This completes the proof that \(M(t,x)\) is jointly continuous in \((t,x)\).
    The formula  \ref{Item::Proof::PtwiseBounds::SmoothnessToImprovingParing::DerivFormula} follows by taking
    \(\partial_t\) of \eqref{Eqn::Proof::PtwiseBounds::SmoothnessToImprovingParing::Tmp2}.
\end{proof}

\begin{lemma}\label{Lemma::Proof::PtwiseBounds::SmoothnessToImprovingParingWithVars}
    Let \(I\subseteq \R\) be an open interval and \(N_1\) and \(N_2\) metric spaces.
    Suppose \(K(t,x)\in \CinftySpacet[I][\CSpacex{N_1}[\BanachSpacew]]\)
    and \(L(t,y)\in \CinftySpacet[I][\CSpacex{N_2}[\BanachSpaceDualw]]\).
    Suppose \(\forall l\in \N\),
    \begin{equation*}
        \sup_{\substack{t\in I \\ x\in N_1}} \BanachSpaceNorm*{\partial_t^l K(t,x)} + \sup_{\substack{t\in I \\ y\in N_1}} \BanachSpaceDualNorm*{\partial_t^l L(t,y)}<\infty.
    \end{equation*}
    Set \(M(t,x,y):=\BanachSpacePairing{K(t,x)}{L(t,y)}\).
    Then,
    \begin{enumerate}[(i)]
        \item\label{Item::Proof::PtwiseBounds::SmoothnessToImprovingParingWithVars::SmoothInt} For each \((x,y)\), \(t\mapsto M(t,x,y)\) is  smooth on its domain.
        \item\label{Item::Proof::PtwiseBounds::SmoothnessToImprovingParingWithVars::ContinuousIntx} For each \(l\in \N\) and \(y\in N_2\), \((t,x)\mapsto \partial_t^l M(t,x,y)\) is continuous on its domain.
        \item\label{Item::Proof::PtwiseBounds::SmoothnessToImprovingParingWithVars::ContinuousInty} For each \(l\in \N\) and \(x\in N_1\), \((t,y)\mapsto \partial_t^l M(t,x,y)\) is continuous on its domain.
        \item\label{Item::Proof::PtwiseBounds::SmoothnessToImprovingParingWithVars::DerivFormula} \(\partial_{t} M(t,x,y)=\BanachSpacePairing{\partial_t K(t,x)}{L(t,y)}+\BanachSpacePairing{K(t,x)}{\partial_t L(t,y)}\).
    \end{enumerate}
\end{lemma}
\begin{proof}
    By fixing \(y\), \ref{Item::Proof::PtwiseBounds::SmoothnessToImprovingParingWithVars::SmoothInt},
    \ref{Item::Proof::PtwiseBounds::SmoothnessToImprovingParingWithVars::ContinuousIntx},
    and \ref{Item::Proof::PtwiseBounds::SmoothnessToImprovingParingWithVars::DerivFormula}
    follow from Lemma \ref{Lemma::Proof::PtwiseBounds::SmoothnessToImprovingParing} applied with \(L(t)=L(t,y)\).
    \ref{Item::Proof::PtwiseBounds::SmoothnessToImprovingParingWithVars::ContinuousInty} follows from
    \ref{Item::Proof::PtwiseBounds::SmoothnessToImprovingParingWithVars::ContinuousIntx} because the assumptions
    are symmetric in \(x\) and \(y\).
\end{proof}

\begin{proof}[Proof of Proposition \ref{Prop::Proof::PtwiseBounds::PreliminaryEstimate}]
    Let \(\delta_2\in (0,\delta_1)\). 
    We first establish the result with \(\delta_2\) in place of \(\delta_1\).
    We use Lemma \ref{Lemma::Proof::OperatorBounds::OnDiagonalOperators} both for
    \(XK_t(x,y)\) and for \(\Yb K_t(x,y)\) (using that our assumptions are symmetric in this replacement; see Remark \ref{Rmk::Assumption::Symmetric}).
    Lemma \ref{Lemma::Proof::OperatorBounds::OnDiagonalOperators} and \eqref{Eqn::Proof::PtwiseBounds::XYbKtFormula}
    show that Lemma \ref{Lemma::Proof::PtwiseBounds::SmoothnessToImprovingParingWithVars} applies
    to establish \ref{Item::Proof::PtwiseBound::SmoothInt}, \ref{Item::Proof::PtwiseBound::ContinuousForFixedy},
    and \ref{Item::Proof::PtwiseBound::ContinuousForFixedx}.
    \eqref{Eqn::Proof::PtwiseBound::BoundKt} follows from Lemma \ref{Lemma::Proof::OperatorBounds::OnDiagonalOperators}
    and H\"older's inequality applied to \eqref{Eqn::Proof::PtwiseBounds::XYbKtFormula}.  Here we have used Lemma \ref{Lemma::Proof::PtwiseBounds::SmoothnessToImprovingParingWithVars} \ref{Item::Proof::PtwiseBounds::SmoothnessToImprovingParingWithVars::DerivFormula}
    when taking \(\partial_t\) derivatives.

    The claim for \(\delta_1\) now follows by taking
    \(\delta_2\uparrow \delta_1\) and using Assumption \ref{Assumption::DoublingOfSXandSY}.
\end{proof}

\begin{proof}[Proof of Theorem \ref{Thm::Results::MainThm}]
    In light of Proposition \ref{Prop::Proof::PtwiseBounds::PreliminaryEstimate}, all that remains to be shown is
    \eqref{Eqn::Results::MainGaussianBounds} (recall, we have reduced to the case \(\omega_0=0\), \(\NSubsetx=\{x_0\}\), and \(\NSubsety=\{y_0\}\)).

    We begin with the \textit{on-diagonal bounds}; namely, we establish \eqref{Eqn::Results::MainGaussianBounds} in the case
    \(t\geq \min \left\{ 2a_1 \delta_0^{2\kappa}, 2^{1-4\kappa} a_1 \metric[x_0][y_0]^{2\kappa} \right\}\).
    Set \(t_0=t\). By Proposition \ref{Prop::Proof::PtwiseBounds::PreliminaryEstimate} with this choice of \(t_0\)
    we have \(\forall j\in \N\), 
    \begin{equation}\label{Eqn::Proof::PtwiseBounds::OnDiag::Tmp1}
    \begin{split}
         \left| \partial_t^j X\Yb K_{t_0}(x_0,y_0) \right|
         \lesssim_j &\left( (t_0/2)^{1/2\kappa}\wedge \delta_0  \right)^{-2\kappa j}
         S_X\left( (t_0/2)^{1/2\kappa} \wedge \delta_0 \right)^{-1}
         S_Y\left( (t_0/2)^{1/2\kappa} \wedge \delta_0 \right)^{-1}
         \\&\times \measure( \MetricBall{x_0}{ (t_0/2)^{1/2\kappa}\wedge \delta_0 })^{-1/p_0}
         \measure( \MetricBall{y_0}{ (t_0/2)^{1/2\kappa}\wedge \delta_0 })^{-1/p_0'}.
    \end{split}
    \end{equation}
    Using that in the case we are considering, we have \((t_0/2)^{1/2\kappa} \wedge \delta_0 \approx \left( t_0^{1/2\kappa}+\metric[x_0][y_0] \right)\wedge \delta_0\), 
    Assumption \ref{Assumption::DoublingOfSXandSY} and Lemma \ref{Lemma::ProofDoubling::AllDoubling} (which also holds with \(x_0\) replaced with \(y_0\)--see Remark \ref{Rmk::Assumption::Symmetric}),
    show that \eqref{Eqn::Proof::PtwiseBounds::OnDiag::Tmp1} implies
    \begin{equation}\label{Eqn::Proof::PtwiseBounds::OnDiag::Tmp2}
        \begin{split}
             \left| \partial_{t}^j X\Yb K_{t_0}(x_0,y_0) \right|
             \lesssim_j &\left( \left(t_0^{1/2\kappa} +\metric[x_0][y_0]\right)\wedge \delta_0  \right)^{-2\kappa j}
             \\&\times S_X\left( \left(t_0^{1/2\kappa} +\metric[x_0][y_0]\right)\wedge \delta_0 \right)^{-1}
             S_Y\left( \left(t_0^{1/2\kappa} +\metric[x_0][y_0]\right) \wedge \delta_0 \right)^{-1}
             \\&\times \measure( \MetricBall{x_0}{ (t_0^{1/2\kappa} +\metric[x_0][y_0])\wedge \delta_0 })^{-1/p_0}
             \measure( \MetricBall{y_0}{ (t_0^{1/2\kappa} +\metric[x_0][y_0])^{1/2\kappa}\wedge \delta_0 })^{-1/p_0'}.
        \end{split}
        \end{equation}
        \eqref{Eqn::Proof::PtwiseBounds::OnDiag::Tmp2} is the same as \eqref{Eqn::Results::MainGaussianBounds},
        except that it lacks the factor
        \begin{equation}\label{Eqn::Proof::PtwiseBounds::OnDiag::Tmp3}
            \exp\left( -c \left( \frac{\left( \metric[x_0][y_0]\wedge \delta_0 \right)^{2\kappa}}{t_0} \right)^{1/(2\kappa-1)} \right).
        \end{equation}
        However, once we choose the admissible constant \(c>0\) in the off-diagonal estimates, below, 
        and using that in the case we're considering \(\left( \metric[x_0][y_0]\wedge \delta_0 \right)^{2\kappa}\lesssim t_0\),
        we have \eqref{Eqn::Proof::PtwiseBounds::OnDiag::Tmp3} \(\approx 1\), and so \eqref{Eqn::Proof::PtwiseBounds::OnDiag::Tmp2} completes the proof of \eqref{Eqn::Results::MainGaussianBounds} in this case.

        Finally, we turn to the \textit{off-diagonal estimates}; namely the case when \(t< \min \left\{ 2^{1-4\kappa} a_1\metric[x_0][y_0]^{2\kappa}, 2a_1\delta_0^{2\kappa} \right\}\);
        set \(t_0=t\).
        Let \(\delta_1:=(\metric[x_0][y_0]/2)\wedge \delta_0\),
        \(\delta_3:=\delta_1/2\) so that
        \(t_0/2<a_1\delta_3^{2\kappa}\),
        and \(\delta_2:=(t_0/2)^{1/2\kappa}<a_1^{1/2\kappa}\delta_0<\delta_0\).
        Using Lemma \ref{Lemma::Proof::OperatorBounds::OffDiagonalOperators} with \(\delta_1\) replaced by \(\delta_3\)
        we have
        \begin{equation}\label{Eqn::Proof::PtwiseBounds::OffDiag::xoff::Tmp1}
            \begin{split}
            & \BLpNorm{\partial_t^j XK_{t_0/2}(x_0,\cdot)}{p_0'}[\MetricSpace\setminus \MetricBallClosure{x_0}{\delta_3}]
            \\&\lesssim_j \delta_3^{-2\kappa j} S_X(\delta_3)^{-1} \measure(\MetricBall{x_0}{\delta_3})^{-1/p_0} \exp\left( -c_2   \delta_3^{2\kappa/(2\kappa-1)} t^{-1/(2\kappa-1)} \right),
            \end{split}
        \end{equation}
         and Lemma \ref{Lemma::Proof::OperatorBounds::OnDiagonalOperators} with this choice of \(\delta_2\), we have
        \begin{equation}\label{Eqn::Proof::PtwiseBounds::OffDiag::xon}
             \BLpNorm{ (\delta_2^{2\kappa} \partial_t)^j XK_{t_0/2}(x_0,\cdot)}{p_0'}[\MetricSpace]
            \lesssim_j S_X(\delta_2)^{-1} \measure(\MetricBall{x_0}{\delta_2})^{-1/p_0}.
        \end{equation}
        Here, \(c_2>0\) is the admissible constant given by \(c_2=c_12^{1/(2\kappa-1)}\) where \(c_1>0\) is the admissible constant from Lemma \ref{Lemma::Proof::OperatorBounds::OffDiagonalOperators}.
        Using \(\delta_3=\delta_1/2<\delta_1\) in \eqref{Eqn::Proof::PtwiseBounds::OffDiag::xoff::Tmp1} (and using Assumption \ref{Assumption::DoublingOfSXandSY} and Lemma \ref{Lemma::ProofDoubling::AllDoubling}), 
        we obtain
        \begin{equation}\label{Eqn::Proof::PtwiseBounds::OffDiag::xoff}
            \begin{split}
            & \BLpNorm{\partial_t^j XK_{t_0/2}(x_0,\cdot)}{p_0'}[\MetricSpace\setminus \MetricBall{x_0}{\delta_1}]
            \leq \BLpNorm{\partial_t^j XK_{t_0/2}(x_0,\cdot)}{p_0'}[\MetricSpace\setminus \MetricBallClosure{x_0}{\delta_3}]
            \\&\lesssim_j \delta_1^{-2\kappa j} S_X(\delta_1)^{-1} \measure(\MetricBall{x_0}{\delta_1})^{-1/p_0} \exp\left( -c_3   \delta_1^{2\kappa/(2\kappa-1)} t_0^{-1/(2\kappa-1)} \right),
            \end{split}
        \end{equation}
        with \(c_3=2^{-2\kappa/(2\kappa-1)}c_2\).
        By using the symmetry of our assumptions in \(x\) and \(y\) (see Remark \ref{Rmk::Assumption::Symmetric}), we also have
        \begin{equation}\label{Eqn::Proof::PtwiseBounds::OffDiag::yoff}
            \begin{split}
            & \BLpNorm{\partial_t^j \Yb K_{t_0/2}(\cdot,y_0)}{p_0}[\MetricSpace\setminus \MetricBall{y_0}{\delta_1}]
            \\&\lesssim_j \delta_1^{-2\kappa j} S_Y(\delta_1)^{-1} \measure(\MetricBall{y_0}{\delta_1})^{-1/p_0'} \exp\left( -c_3   \delta_1^{2\kappa/(2\kappa-1)} t^{-1/(2\kappa-1)} \right),
            \end{split}
        \end{equation}
        \begin{equation}\label{Eqn::Proof::PtwiseBounds::OffDiag::yon}
             \BLpNorm{ (\delta_2^{2\kappa} \partial_t)^j \Yb K_{t_0/2}(\cdot,y_0)}{p_0}[\MetricSpace]
            \lesssim_j S_Y(\delta_2)^{-1} \measure(\MetricBall{x_0}{\delta_2})^{-1/p_0'}.
        \end{equation}
        Since \(\delta_1\leq \metric[x_0][y_0]/2\), we have \(\left( \MetricSpace\setminus \MetricBall{x_0}{\delta_1} \right)\bigcup \left( \MetricSpace\setminus \MetricBall{y_0}{\delta_1} \right)=\MetricSpace\).
        Using this, \eqref{Eqn::Proof::PtwiseBounds::XYbKtFormula}, and repeated applications of
        Lemma \ref{Lemma::Proof::PtwiseBounds::SmoothnessToImprovingParing} \ref{Item::Proof::PtwiseBounds::SmoothnessToImprovingParing::DerivFormula}, we have for \(l\in \N\),
        \begin{equation}\label{Eqn::Proof::PtwiseBounds::OffDiag::CutUpFinalEstimate}
        \begin{split}
             &\left| \partial_t^l X\Yb K_{t_0}(x_0,y_0)  \right|
             = \left| \partial_t^l \int XK_{t/2}(x_0,z) \Yb K_{t/2}(z,y_0)\: d\measure(z)\big|_{t=t_0} \right|
             \\&\lesssim_l \sum_{j=0}^l \int \left| \partial_t^{l-j}  X K_{t_0/2}(x,z) \partial_t^j \Yb K_{t_0/2}(z,y_0)\right|\:d\measure(z)  
             \\&\lesssim_l \sum_{j=0}^l \int_{\MetricSpace\setminus \MetricBall{x_0}{\delta_1}}\left| \partial_t^{l-j}  X K_{t_0/2}(x,z) \partial_t^j \Yb K_{t_0/2}(z,y_0)\right|\:d\measure(z)
             \\&\quad\quad+\sum_{j=0}^l \int_{\MetricSpace\setminus \MetricBall{y_0}{\delta_1}}\left| \partial_t^{j}  X K_{t_0/2}(x,z) \partial_t^{l-j} \Yb K_{t_0/2}(z,y_0)\right|\:d\measure(z)
            \\&\leq \sum_{j=0}^l \BLpNorm{\partial_t^{l-j} XK_{t_0/2}(x_0,\cdot)}{p_0'}[\MetricSpace\setminus \MetricBall{x_0}{\delta_1}] \BLpNorm{\partial_t^j \Yb K_{t_0/2}(\cdot, y_0)}{p_0}[\MetricSpace]
            \\&\quad\quad +\sum_{j=0}^l \BLpNorm{\partial_t^{j} XK_{t_0/2}(x_0,\cdot)}{p_0'}[\MetricSpace] \BLpNorm{\partial_t^{l-j} \Yb K_{t_0/2}(\cdot, y_0)}{p_0}[\MetricSpace\setminus \MetricBall{y_0}{\delta_1}]
        \end{split}
        \end{equation}
        We estimate the two terms on the right-hand side of \eqref{Eqn::Proof::PtwiseBounds::OffDiag::CutUpFinalEstimate} 
        using \eqref{Eqn::Proof::PtwiseBounds::OffDiag::xon}, \eqref{Eqn::Proof::PtwiseBounds::OffDiag::xoff}, \eqref{Eqn::Proof::PtwiseBounds::OffDiag::yoff}, and \eqref{Eqn::Proof::PtwiseBounds::OffDiag::yon}.
        The estimate for the two terms differs only by reversing the roles of \(x_0\) and \(y_0\), so we estimate only the first.
        We have, for \(0\leq j\leq l\),
        \begin{equation*}
        \begin{split}
             &\BLpNorm{\partial_t^{l-j} XK_{t_0/2}(x_0,\cdot)}{p_0'}[\MetricSpace\setminus \MetricBall{x_0}{\delta_1}] \BLpNorm{\partial_t^j \Yb K_{t_0/2}(\cdot, y_0)}{p_0}[\MetricSpace]
            \\&\lesssim_l \delta_1^{-2\kappa (l-j)}\delta_2^{-2\kappa j} S_X(\delta_1)^{-1}S_Y(\delta_2)^{-1} \measure(\MetricBall{x_0}{\delta_1})^{-1/p_0}\measure(\MetricBall{x_0}{\delta_2})^{-1/p_0'} 
            \exp\left( -c_3   \delta_1^{2\kappa/(2\kappa-1)} t^{-1/(2\kappa-1)} \right)
        \end{split}
        \end{equation*}
        Using \(\delta_1=(\metric[x_0][y_0]/2)\wedge \delta_0\approx (t^{1/2\kappa}+\metric[x_0][y_0])\wedge \delta_0\),
        Lemma \ref{Lemma::ProofDoubling::AllDoubling},  Assumption \ref{Assumption::DoublingOfSXandSY}, and \(\delta_2=(t_0/2)^{1/2\kappa}\wedge \delta_0\),
        we have with \(c_4=2^{-2\kappa/(2\kappa-1)}c_3>0\)
        \begin{equation*}
        \begin{split}
             &\BLpNorm{\partial_t^{l-j} XK_{t_0/2}(x_0,\cdot)}{p_0'}[\MetricSpace\setminus \MetricBall{x_0}{\delta_1}] \BLpNorm{\partial_t^j \Yb K_{t_0/2}(\cdot, y_0)}{p_0}[\MetricSpace]
             \\&\lesssim_l \left( \left( t_0^{1/2\kappa}+\metric[x_0][y_0] \right)\wedge \delta_0 \right)^{-2\kappa(l-j)} \left( t_0^{1/2\kappa}\wedge \delta_0 \right)^{-2\kappa j}
             \\&\quad\times S_X\left( \left( t_0^{1/2\kappa}+\metric[x_0][y_0] \right)\wedge \delta_0 \right)^{-1} S_Y\left(  t_0^{1/2\kappa}\wedge \delta_0  \right)^{-1}
             \\&\quad\times \measure\left( \BMetricBall{x_0}{ \left( t_0^{1/2\kappa}+\metric[x_0][y_0] \right)\wedge \delta_0} \right)^{-1/p_0}
             \measure\left( \BMetricBall{y_0}{t_0^{1/2\kappa}\wedge \delta_0 } \right)^{-1/p_0'}
             \\&\quad\times\exp\left( -c_4 \left( \frac{\left( \metric[x_0][y_0]\wedge \delta_0 \right)^{2\kappa}}{t} \right)^{1/(2\kappa-1)}    \right)
            \\&=  \left( \left( t_0^{1/2\kappa}+\metric[x_0][y_0] \right)\wedge \delta_0 \right)^{-2\kappa l}
            \\&\quad\times S_X\left( \left( t_0^{1/2\kappa}+\metric[x_0][y_0] \right)\wedge \delta_0 \right)^{-1} S_Y\left(  \left( t_0^{1/2\kappa}+\metric[x_0][y_0] \right)\wedge \delta_0  \right)^{-1}
            \\&\quad \times \measure\left( \BMetricBall{x_0}{ \left( t_0^{1/2\kappa}+\metric[x_0][y_0] \right)\wedge \delta_0} \right)^{-1/p_0} \measure\left( \BMetricBall{y_0}{ \left( t_0^{1/2\kappa}+\metric[x_0][y_0] \right)\wedge \delta_0} \right)^{-1/p_0'}
            \\&\quad\times\exp\left( -(c_4/2) \left( \frac{\left( \metric[x_0][y_0]\wedge \delta_0 \right)^{2\kappa}}{t} \right)^{1/(2\kappa-1)}    \right)
            \\&\quad\times \Error,
        \end{split}
        \end{equation*}
        where
        \begin{equation*}
        \begin{split}
             \Error
             =&\frac{\left(  \left( t_0^{1/2\kappa}+\metric[x_0][y_0] \right)\wedge \delta_0  \right)^j S_Y\left(  \left( t_0^{1/2\kappa}+\metric[x_0][y_0] \right)\wedge \delta_0  \right)}{\left( t_0^{1/2\kappa}\wedge \delta_0 \right)^j S_Y\left(  t_0^{1/2\kappa}\wedge \delta_0 \right)}
             \frac{\measure\left( \BMetricBall{y_0}{ \left( t_0^{1/2\kappa}+\metric[x_0][y_0] \right)\wedge \delta_0} \right)^{1/p_0'}}{\measure\left( \BMetricBall{y_0}{  t_0^{1/2\kappa}\wedge \delta_0   } \right)^{1/p_0'}}
             \\&\times \exp\left( -(c_4/2) \left( \frac{\left( \metric[x_0][y_0]\wedge \delta_0 \right)^{2\kappa}}{t} \right)^{1/(2\kappa-1)}    \right).
        \end{split}
        \end{equation*}
        The result will follow with \(c=c_4/2\) once we show \(\Error\lesssim_l 1\).
        That \(\Error\lesssim_j 1\) follows from Lemma \ref{Lemma::ProofDoubling::ExpBeatsDoubling} with
        \(F(\delta) = (\delta\wedge \delta_0)^j S_Y(\delta\wedge \delta_0) \measure\left( \BMetricBall{y_0}{ \delta \wedge \delta_0} \right)^{1/p_0'}\) and with
        \(\delta_1\) and \(\delta_2\) in that lemma replaced by \(t_0^{1/2\kappa}\) and \(\metric[x_0][y_0]\wedge \delta_0\), respectively;
        here we have used Lemma \ref{Lemma::ProofDoubling::AllDoubling} and Assumption \ref{Assumption::DoublingOfSXandSY}.
        Since \(j\leq l\), this implies \(\Error\lesssim_l 1\), completing the proof.
\end{proof}

\bibliographystyle{amsalpha}

\bibliography{heat}

\center{\it{University of Wisconsin-Madison, Department of Mathematics, 480 Lincoln Dr., Madison, WI, 53706}}

\center{\it{street@math.wisc.edu}}

\center{MSC 2020:  35K08 (Primary), 47D06, 35H10, and 35H20 (Secondary)}

\end{document}